%% file: arxiv_sub.tex
\documentclass{article}

\usepackage{microtype}
\usepackage{graphicx}
\usepackage{subfigure}
\usepackage{booktabs} 
\input{macros.tex}

\usepackage{hyperref}

\usepackage[accepted]{icml2019blank}

\icmltitlerunning{Stochastic Distributed Learning with Gradient Quantization and Variance Reduction}

\begin{document}

\twocolumn[
\icmltitle{Stochastic Distributed Learning with\\Gradient Quantization and Variance Reduction}



\icmlsetsymbol{equal}{*}

\begin{icmlauthorlist}
\icmlauthor{Samuel Horv\'{a}th}{kaust}
\icmlauthor{Dmitry Kovalev}{kaust}
\icmlauthor{Konstantin Mishchenko}{kaust}
\icmlauthor{Peter Richt\'{a}rik}{kaust,ed,mipt}
\icmlauthor{Sebastian U.\ Stich}{epfl}
\end{icmlauthorlist}

\icmlaffiliation{kaust}{King Abdullah University of Science and Technology, KSA}
\icmlaffiliation{ed}{University of Edinburgh, UK}
\icmlaffiliation{epfl}{\'{E}cole polytechnique f\'{e}d\'{e}rale de Lausanne, Switzerland}
\icmlaffiliation{mipt}{Moscow Institute of Physics and Technology, Russian Federation}

\icmlcorrespondingauthor{Peter Richt\'{a}rik}{peter.richtarik@kaust.edu.sa}

\icmlkeywords{distributed learning, quantization, variance reduction, DIANA}

\vskip 0.3in
]



\printAffiliationsAndNotice{}  

\input{main.tex}

\end{document}

%% file: macros.tex
\usepackage{amsfonts}
\usepackage{amsmath}
\usepackage{amsthm}
\usepackage{amssymb} 

\usepackage{xcolor}
\usepackage{color}
\usepackage{graphicx}

 \usepackage[algo2e]{algorithm2e} 
\usepackage{verbatim}
\usepackage{algorithm}

\usepackage{multirow}

\newcommand{\R}{\mathbb{R}} 
\newcommand{\N}{\mathbb{N}} 

\newcommand{\cD}{{\cal D}}

\newcommand{\cO}{{\cal O}}

\newcommand{\cX}{{\cal X}}

\usepackage[colorinlistoftodos,bordercolor=orange,backgroundcolor=orange!20,linecolor=orange,textsize=scriptsize]{todonotes}

\newcommand{\eqdef}{\overset{\text{def}}{=}} 
\newcommand{\dotprod}[2]{\left\langle #1,#2\right\rangle} 
\newcommand{\norm}[1]{\lVert#1\rVert}      
\newcommand{\abs}[1]{\lvert#1\rvert}      
\newcommand{\lin}[1]{\left\langle #1\right\rangle} 

\DeclareMathOperator{\sign}{sign}     
\DeclareMathOperator{\argmin}{argmin}        
\DeclareMathOperator{\prox}{prox}       

\newcommand{\E}[1]{{\rm E}\left[#1\right] } 
\newcommand{\EE}[2]{{\rm E}_{#1}\left[#2\right] }

\usepackage[colorlinks=true,linkcolor=blue]{hyperref} 
\usepackage[capitalize,noabbrev]{cleveref}

\newtheorem{lemma}{Lemma}
\newtheorem{corollary}{Corollary}

\newtheorem{definition}{Definition}
\newtheorem{remark}{Remark}
\newtheorem{assumption}{Assumption}
\newtheorem{theorem}{Theorem}
\newtheorem{example}{Example}

\usepackage{url}

\makeatletter
\newcommand{\removelatexerror}{\let\@latex@error\@gobble}
\makeatother
\LinesNumbered 
\SetCommentSty{itshape}
\SetKwComment{Comment}{$\triangleright$\ }{}

%% file: main.tex


\begin{abstract}
We consider distributed optimization where the objective function is spread among different devices, each sending incremental model updates to a central server. To alleviate the communication bottleneck, recent work proposed various schemes to compress (e.g.\ quantize or sparsify) the gradients, thereby introducing additional variance $\omega \geq 1$ that might slow down convergence. For strongly convex functions with condition number $\kappa$ distributed among $n$ machines, we 
(i) give a scheme that converges in $\cO((\kappa + \kappa \frac{\omega}{n} + \omega)$
$\log (1/\epsilon))$ steps to a neighborhood of the optimal solution.
For objective functions with a finite-sum structure, each worker having less than $m$ components, we
(ii) present novel variance reduced schemes that converge in $\cO((\kappa + \kappa \frac{\omega}{n} + \omega + m)\log(1/\epsilon))$ steps to arbitrary accuracy $\epsilon > 0$. These are the first methods that achieve linear convergence for arbitrary quantized updates.
We also (iii) give analysis for the weakly convex and non-convex cases and
(iv) verify in experiments that our novel variance reduced schemes are more efficient than the baselines.
\end{abstract}

\section{Introduction}
\label{sec:intro}
The training of large scale machine learning models poses many challenges that stem from the mere size of the available training data. The \emph{data-parallel} paradigm focuses on distributing the data across different compute notes, which operate in parallel on the data. Formally, we consider optimization problems distributed across $n$ nodes of the form
\begin{align}
\min_{x \in \R^d} \left[ f(x) := \frac{1}{n} \sum_{i=1}^n f_i(x) \right] \,, \label{eq:probR}
\end{align}
where  $f_i \colon \R^d \to \R$ for $i=1,\dots,n$ are given as
\begin{align}
f_i(x) := \EE{\zeta \sim \cD_i}{F(x,\zeta)}\,,
\end{align}
with $F\colon \R^d \times \Omega \to \R$ being a general loss function. The distributions $\cD_1, \dots, \cD_n$ can be different on every node, which means the functions $f_1,\dots,f_n$ can have completely different minimizers. This frameworks covers stochastic optimization (when either $n=1$ or all $\cD_i$ are identical) and empirical risk minimization (for e.g.\ when the $\cD_i$'s are discrete with disjoint support). We denote by $x^\star$ an optimal solution of \eqref{eq:probR}, let $f^\star\eqdef f(x^\star)$.

\subsection{Quantization to Reduce Communication}
In typical computing architectures, communication is much slower than computation and the communication bottleneck between workers is a major limiting factor for many large scale applications (such as e.g.\ deep neural network training), as reported in e.g.~\cite{Seide2015:1bit,Alistarh2017:qsgd,Zhang2017:zipml,Lin2018:deep}. 
A possible remedy to tackle this issue are for instance approaches that focus on increasing the computation  to communication ratio, such as increased mini-batch sizes~\cite{Goyal2017:large}, defining local problems for each worker~\cite{Shamir2014:approxnewton,Reddi:2016aide} or reducing the communication frequency~\cite{Mann2009:parallelSGD,Zinkevich2010:parallelSGD,you2017imagenet,Stich2018:localsgd}.

A direction orthogonal to these approaches tries to reduce the size of the messages---typically gradient vectors---that are exchanged between the nodes~\cite{Seide2015:1bit,Strom2015:1bit,Alistarh2017:qsgd,Wen2017:terngrad,grishchenko2018asynchronous}.
These \emph{quantization} techniques rely on (lossy) compression of the gradient vectors. 
In the simplest form, these schemes limit the number of bits that are used to represent floating point numbers~\cite{Gupta:2015limited,Na2017:limitedprecision}, reducing the size of a $d$-dimensional (gradient) vector by a constant factor. Random dithering approaches attain up to $\cO(\sqrt{d})$ compression~\cite{Seide2015:1bit,Alistarh2017:qsgd,Wen2017:terngrad}. The most aggressive schemes reach $\cO(d)$ compression by only sending a constant number of bits per iteration~\cite{Suresh2017,RDME,Alistarh2018:topk,stich2018sparse}. An alternative approach is not to compute the gradient and subsequently compress it, but to update a subset of elements of the iterate $x$ using coordinate descent type updates~\cite{Hydra, Hydra2, mishchenko201999}.

Recently, \citet{Mishchenko2019:diana} proposed the first method that successfully applies the gradient quantization technique to the distributed optimization problem~\eqref{eq:probR} with a non-smooth regularizer.

\subsection{Contributions}

We now briefly outline the key contributions of our work:

\textbf{General compression.} We generalize the method presented in \cite{Mishchenko2019:diana} allowing for {\em arbitrary compression} (e.g., quantization and sparsification) operators. Our analysis is tight, i.e.\ we recover their convergence rates 
as a special case. Our more general approach allows to choose freely the compression operators that perform best on the available system resources and can thus offer gains in training time over the previous method that is tied to a single class of operators.

\textbf{Variance reduction.} We present several {\em variance reduced} extensions of our methods for  distributed training. In particular, and in contrast to  \cite{Mishchenko2019:diana}, our methods converge to the optimum and not merely to a neighborhood, without any loss in convergence rate. The quantized SVRG method of~\citet{Alistarh2017:qsgd} did not only rely on a specific compression scheme, but also on exact communication from time to time (epoch gradients), both restrictions been overcome here.

\textbf{Convex and non-convex problems.} We provide concise convergence analysis for our novel  schemes for the strongly-convex, the (weakly) convex and non-convex setting. Our analysis recovers the respective rates of earlier schemes without communication compression and shows that compression can provide huge benefits when communication is a bottleneck.

\textbf{Experiments.} We compare the novel variance reduced schemes with various baselines. In the experiments we leverage the flexibility of our approach---allowing to freely chose a quantization scheme with optimal parameters---and demonstrate on par performance with the baselines in terms of iterations, but considerable savings in terms of total communication cost.

\section{Related Work}
Gradient compression schemes have successfully been used in many implementations as a heuristic to reduce communication cost, such as for instance in 1BitSGD~\cite{Seide2015:1bit,Strom2015:1bit} that rounds the gradient components to either -1 or 1, and error feedback schemes aim at reducing quantization errors among iterations, for instance as in~\cite{Lin2018:deep}.
In the discussion below we focus in particular on schemes that enjoy theoretical convergence guarantees. 

\subsection{Quantization and Sparsification}
\label{sec:introducequant}
A class of very common quantization operators is based on random dithering~\cite{Goodall1951:randdithering,Roberts1962:randdithering} and can be described as the random operators $Q \colon \R^d \to \R^d$,
\begin{align}
Q(x) = \sign(x) \cdot \norm{x}_p \cdot \frac{1}{s}\cdot \left\lfloor s \frac{ \abs{x}}{\norm{x}_p} + \xi \right\rfloor \label{eq:Q}
\end{align}
for random variable $\xi \sim_{\rm u.a.r.} [0,1]^{d}$, parameter $p \geq 1$, and $s \in \N_+$, denoting the \emph{levels} of the rounding. Its unbiasedness property, $\EE{\xi}{Q(x)}=x$, $\forall x \in \R^d$, is the main catalyst of the theoretical analyses. 
The quantization~\eqref{eq:Q} was for instance used in QSGD for $p=2$~\cite{Alistarh2017:qsgd}, in TernGrad for $s=1$ and $p=\infty$~\cite{Wen2017:terngrad} or for general $p \geq 1$ and $s=1$  in DIANA~\cite{Mishchenko2019:diana}. 
For $p=2$ the expected sparsity is $\EE{\xi}{\norm{Q(x)}_0} = \cO(s(s+\sqrt{d}))$~\cite{Alistarh2017:qsgd} 
and encoding a nonzero coordinate of $Q(x)$ requires $\cO(\log(s))$ bits. 

Much sparser vectors can be obtained by random sparsification techniques that randomly mask the input vectors and only preserve a constant number of coordinates~\cite{Suresh2017, RDME,Wangni2018:sparsification,stich2018sparse}. 
Experimentally is has been shown that deterministic masking---for instance selecting the largest components in absolute value~\cite{Dryden2016:topk,Aji2017:topk}---can outperform the random techniques. 
However, as these schemes are biased, they resisted careful analysis until very recently~\cite{Alistarh2018:topk,stich2018sparse}. 
We will not further distinguish between sparsification and quantization approaches, and refer to both of these compression schemes them as `quantization' in the following.

Also schemes with error compensation techniques have recently been successfully vanquished, for instance for unbiased quantization on quadratic functions~\cite{Wu:2018error} and for unbiased and biased quantization on strongly convex functions~\cite{stich2018sparse}. 
These schemes suffer much less from large quantization errors, and can tolerate higher variance (both in practice an theory) than the methods above without error compensation.

\subsection{Quantization in Distributed Learning}
In centralized approaches, all nodes communicate with a central node (parameter sever) that coordinates the optimization process. An algorithm of specific interest to solve problem~\eqref{eq:probR} is mini-batch SGD~\cite{Dekel2012:minibatch, pegasos2}, a parallel version of stochastic gradient descent (SGD)~\cite{Robbins:1951ko,Nemirovski-Juditsky-Lan-Shapiro-2009}. In mini-batch SGD, full gradient vectors have to be communicated to the central node and hence it is natural to incorporate gradient compression to reduce the cost of the communication rounds.
\citet{Khirirat2018:distributed} study quantization in the deterministic setting, i.e.\ when the gradients $\nabla f_i(x)$ can be computed without noise, and show that parallel gradient descent with unbiased quantization converges to a neighborhood of the optimal solution on strongly convex functions. \citet{Mishchenko2019:diana} consider the stochastic setting as in~\eqref{eq:probR} and show convergence of SGD with unbiased quantization to a neighborhood of a stationary point for non-convex and strongly convex functions; the analysis in~\cite{Wu:2018error} only applies to quadratic functions. The method of~\citet{stich2018sparse} can also be parallelized as shown by~\citet{cordonnier2018:sparse} and converges to the exact solution on strongly convex functions.

Decentralized methods do not require communication with a centralized node. \citet{Tang2018:decentralized} consider unbiased quantization and show convergence to the neighborhood of a stationary point on non-convex functions under very rigid assumptions on the quantization, i.e.\ allowing only $\cO(1)$ compression, \citet{Koloskova2019:decentralized} relax those constraints but only consider strongly convex functions.

\subsection{Quantization and Variance Reduction}
Variance reduced methods~\cite{LeRoux:2012sag, Johnson:2013svrg, ShalevShwartz:2013sdca, Defazio2014:saga, QUARTZ, SDCA-dual-free, SDNA,  dfSDCA, ADASDCA, nguyen2017sarah, JacSketch, zhou2018direct} admit linear convergence on empirical risk minimization problems, surpassing the rate of vanilla SGD~\cite{Bottou2010:sgd}. \citet{Alistarh2017:qsgd} proposed a quantized version of SVRG~\cite{Johnson:2013svrg}, but the proposed scheme relies on broadcasting exact (unbiased) gradients every epoch.
This restriction has been overcome in~\cite{kuenstner2017:svrg} but only for high-precision quantization. Here we alleviate these restrictions and present quantization not only for SVRG but also for SAGA~\cite{Defazio2014:saga}. Our analysis also supports a version of SVRG whose epoch lengths are not fixed, but random, similar as e.g.\ in~\cite{Lei2017:singlepass,Hannah2018:span}. Our base version (denoted as L-SVRG) is slightly different and inspired by observations made in~\cite{Hofmann2015:saga,Raj2018:ksvrg} and following closely~\cite{Kovalev2019:svrg}.

\subsection{Orthogonal Approaches}
There are other approaches aiming to reduce the communication cost, such as
 increased mini-batch sizes~\cite{Goyal2017:large}, defining local problems for each worker~\cite{Shamir2014:approxnewton,cocoa, COCOA+, Reddi:2016aide, COCOA+journal}, reducing the communication frequency~\cite{Mann2009:parallelSGD,Zinkevich2010:parallelSGD,you2017imagenet,Stich2018:localsgd} or distributing along features~\cite{Hydra, Hydra2}. However, we are not considering such approaches here.

\section{General Definitions}
\begin{definition}[$\mu$-strong convexity]
	A differentiable function $f\colon \R^d \to \R$ is $\mu$\nobreakdash-strongly convex for $\mu > 0$, if for all $x,y \in \R^d$
	\begin{align}
		f(y) \geq f(x) + \dotprod{\nabla f(x)}{y - x} + \frac{\mu}{2} \norm{x - y}_2^2 \,. \label{def:strongconvex}
	\end{align}
\end{definition}

\begin{definition}[$L$-smoothness]
	A differentiable function $f \colon \R^d \to \R$ is $L$-smooth for $L > 0$, if for all $x,y \in \R^d$
	\begin{align}
		f(y) \leq f(x) + \dotprod{\nabla f(x)}{y - x} + \frac{L}{2} \norm{x - y}_2^2\,.
		\label{def:smoothness}
	\end{align}
\end{definition}
\begin{definition}
\label{def:prox}
The prox operator $\prox_{\gamma R} \colon \R^d \to \R^d$ is defined as
\begin{align*}
 \prox_{\gamma R}(x) = \argmin_{y \in \R^d} \left\{\gamma R(y) + \frac{1}{2}\norm{y-x}_2^2 \right\}\,,
\end{align*}
for $\gamma > 0$ and a closed  convex  regularizer $R \colon \R^d \to \R \cup \{\infty\}$.
\end{definition}

\section{Quantization Operators}
Our analysis depends on a general notion of quantization operators with bounded variance.
\begin{definition}[$\omega$-quantization]
\label{def:omegaquant}
A random operator $Q \colon \R^d \to \R^d$ with the properties
\begin{align}
  \EE{Q}{Q(x)}=x\,, &   &\EE{Q}{\norm{Q(x)}_2^2} \leq (\omega+1) \norm{x}_2^2\,, \label{def:Q}
\end{align}
for all $x \in \R^d$ is a $\omega$-quantization operator. Here $\EE{Q}{\cdot}$ denotes the expectation over the (internal) randomness of~$Q$.
\end{definition}
\begin{remark}
As $\E{ \norm{X-\E{X}}_2^2} = \E{\norm{X}_2^2} - \norm{\E{X}}_2^2$ for any random vector $X$, equation~\eqref{def:Q} implies 
\begin{align}
\EE{Q}{\norm{Q(x) - x}_2^2} \leq \omega \norm{x}_2^2\,. \label{def:omega}
\end{align}
\end{remark}
For instance, we see that $\omega = 0$ implies $Q(x)=x$. 
\begin{remark}
Besides the variance bound, Definition~\ref{def:omegaquant} does not impose any further restrictions on $Q(x)$. However, for all applications it is advisable to consider operators $Q$ that achieve a considerable compression, i.e.\ $Q(x)$ should be cheaper to encode than $x$.
\end{remark}
We will now give examples of a few $\omega$-quantization operators that also achieve compression, either by quantization techniques, or by enforcing sparsity (cf. Sec.~\ref{sec:introducequant}).
\begin{example}[random dithering]
The operator given in~\eqref{eq:Q} satisfies~\eqref{def:Q}
for $\omega(x) := 2 + \frac{\norm{x}_1 \norm{x}_p}{s\norm{x}_2^2}$ for every fixed $x \in \R^d$, and $\omega(x)$ is a function monotonically decreasing in $p$. Moreover,
\begin{align*}
\omega = \cO \left( \frac{d^{1/p} + d^{1/2}}{s} \right)
\end{align*}
for $p \geq 1$, $s \geq 1$ and all $x \in \R^d$.
\end{example}
For $p=2$ this bound was proven by~\citet{Alistarh2017:qsgd}. Here we generalize the analysis to arbitrary $p \geq 1$.
\begin{proof}
In view of~\eqref{eq:Q} we have
\begin{align*}
\EE{Q}{\norm{Q(x)}_2^2} =  \frac{\norm{x}_p^2}{s^2} \sum_{i=1}^d \left(\underbrace{\ell_i^2 (1-p_i) + (\ell_i+1)^2 p_i}_{\leq \ell_i^2 + (2\ell_i + 1)p_i} \right)
\end{align*}
for integers $\ell_i \leq s \abs{x_i}/\norm{x}_p \leq \ell_i+1$ and probabilities $p_i = s \abs{x_i}/\norm{x}_p - \ell_i \leq s\abs{x_i}/\norm{x}_p$. Therefore,
\begin{align*}
\EE{Q}{\norm{Q(x)}_2^2} &\leq \frac{\norm{x}_p^2}{s^2} \sum_{i=1}^d \left(\ell_i^2 + (2\ell_i + 1)\frac{s\abs{x_i}}{\norm{x}_p}\right) \\
& \leq \frac{\norm{x}_p^2}{s^2} \left(s^2\frac{\norm{x}_2^2}{\norm{x}_p^2} + 2s^2\frac{\norm{x}_2^2}{\norm{x}_p^2} + s\frac{\norm{x}_1}{\norm{x}_p} \right)
\end{align*}
as $\ell_i \leq s \abs{x_i}/\norm{x}_p $. This proves the bound on $\omega(w)$. By H\"{o}lder's inequality $\norm{x}_1 \leq d^{1/2} \norm{x}_2$ and $\norm{x}_p \leq d^{1/p - 1/2} \norm{x}_2$ for $1 \leq p \leq 2$ and for $p\geq 2$ the inequality $\norm{x}_p \leq \norm{x}_2$ imply the upper bound on $\omega$ for all $x \in \R^d$.
\end{proof}

\begin{example}[random sparsifiction]
The random sparsification operator $Q(x) = \frac{d}{r} \cdot \xi \otimes x$ for random variable $\xi \sim_{\rm u.a.r.} \{y \in  \{0,1\}^d \colon \norm{y}_0 = r\}$ and sparsity parameter $r \in \N_+$ is a $\omega = \frac{d}{r}-1$ quantization operator.
\end{example}
For a proof see e.g.~\citep[Lemma A.1]{stich2018sparse}.

\begin{example}[block quantization]
The vector $x \in \R^d$ is first split into $t$ blocks and then each block $v_i$ is quantized using random dithering with $p=2$, $s = 1$. If every block has the same size $d/t$, then this gives a quantization operator with $\omega = \sqrt{d/t} + 1$, otherwise $\omega = \max_{i \in[t]} \sqrt{|v_i|} + 1$. 
\end{example}
Block quantization was e.g.\ used in~\cite{Alistarh2017:qsgd} and is also implemented (in a similar fashion) in the CNTK toolkit~\cite{seide2016cntk}.
The proofs of the claimed bounds can be found in~\cite{Mishchenko2019:diana}.

\section{DIANA with Arbitrary Quantization}
We are now ready to present the first algorithm and its theoretical properties.  In this section we consider the regularized problem~\eqref{eq:probR}:
\begin{align}
\min_{x \in \R^d} \left[ f(x) + R(x) := \frac{1}{n} \sum_{i=1}^n f_i(x) + R(x) \right] \,, \label{eq:probRR}
\end{align}
where $R \colon \R^d \to \R \cup \{+\infty\}$ denotes a closed convex regularizer.

\subsection{DIANA} Algorithm~\ref{alg:improvedDIANA} is identical to the DIANA algorithm in~\cite{Mishchenko2019:diana}. However, we allow for {\em arbitrary $\omega$-quantization operators}, whereas \citet{Mishchenko2019:diana} consider random dithering quantization operators \eqref{eq:Q}   with $p\geq 1$ and $s=1$ only.

In Algorithm~\ref{alg:improvedDIANA}, each worker $i=1,\dots,n$ queries the oracle and computes an unbiased stochastic gradient $g_i^k$ in iteration $k$, i.e. $\E{g_i^k \mid x^k} = \nabla f_i(x^k)$. A na\"ive approach would be to directly send the quantized gradients, $Q(g_i^k)$, to the master node. However, this simple scheme does not only introduce a lot of noise (for instance, even at the optimal solution $x^\star \in \R^d$ the norms of the stochastic gradients do not vanish), but also does in general not converge for nontrivial regularizers $R \not \equiv 0$.
Instead, in Algorithm~\ref{alg:improvedDIANA} each worker maintains a \emph{state} $h_i^k \in \R^d$ and quantizes only the \emph{difference} $g_i^k - h_i^k$ instead. If (and this we show below) $h_i^k$ converges to $\nabla f_i(x^\star)$ for ($k \to \infty$), the variance of the quantization can be massively  reduced compared to the na\"ive scheme. Both the worker and the master node update $h_i^k$ based on the transmitted (quantized) vector $\hat{\Delta}_i^k \in \R^d$. Note that the quantization operator $Q$ should be chosen so that the transmission of $\hat{\Delta}_i^k$ requires significantly less bits than the transmission of the full  $d$-dimensional vector.

\subsection{Convergence of Algorithm~\ref{alg:improvedDIANA}} We make the following technical assumptions:
\begin{assumption}
\label{assumption:diana}
In problem~\eqref{eq:probRR} we assume each $f_i \colon \R^d \to \R$ to be $\mu$-strongly convex and $L$-smooth. We assume that each $g_i^k$ in~Algorithm~\ref{alg:improvedDIANA} has bounded variance
\begin{align}
 \E{\norm{g_i^k - \nabla f_i(x^k)}_2^2 } &\leq \sigma_i^2\,, & &\forall k \geq 0, i=1,\dots,n \label{def:sigmai}
\end{align}
for constants $\sigma_i \leq \infty$, $\sigma^2 := \frac{1}{n} \sum_{i=1}^n \sigma_i^2$.\end{assumption}

\begin{figure}[t]
\vspace{-7.8pt}
\begin{algorithm}[H]
\removelatexerror
\begin{algorithm2e}[H]
\DontPrintSemicolon
  \KwIn{learning rates $\alpha > 0$, and $\gamma > 0$,
  initial vectors $x^0$, $h_1^0, \dots, h_n^0$ and $h^0 = \frac{1}{n} \sum_{i=1}^n h_i^0$}
  \For{$k=0,1,\dots$}{
  broadcast $x^k$ to all workers\;  
  \For(in parallel \hfill $\triangleright$\ \textit{worker side}){ $i=1,\dots,n$ }{
   sample $g_i^k$ such that $\E{g_i^k \mid x^k} = \nabla f_i(x^k)$\;   
   $\Delta_i^k = g_i^k - h_i^k$\;
   $\hat{\Delta}_i^k = Q(\Delta_i^k)$\;
   $h_i^{k+1} = h_i^k + \alpha \hat{\Delta}_i^k$\;
   $\hat{g}_i^k = h_i^k + \hat{\Delta}_i^k$\;
   }
   $q^k = \tfrac{1}{n} \sum_{i=1}^n \hat{\Delta}_i^k$    \Comment*[r]{gather quantized updates} 
   $\hat{g}^k = \tfrac{1}{n} \sum_{i=1}^n \hat{g}_i^k$\;
   $x^{k+1} = \prox_{\gamma R}(x^k - \gamma \hat{g}^k)$\;
   $h^{k+1}= \tfrac{1}{n} \sum_{i=1}^n h_i^{k+1}$
  }
\end{algorithm2e}  
\caption{DIANA with arbitrary unbiased quantization}
\label{alg:improvedDIANA}
\end{algorithm}
\end{figure}

The main theorem of this section, presented next, establishes linear convergence of Algorithm~\ref{alg:improvedDIANA} with arbitrary $\omega$-quantization schemes.

\begin{theorem}
\label{thm:improvedDIANA}
Consider Algorithm~\ref{alg:improvedDIANA} with $\omega$-quantization $Q$ and stepsize $\alpha \leq \frac{1}{\omega+1}$.
Define the Lyapunov function 
\begin{align*}
 \Psi^k := \norm{x^k - x^\star}_2^2 + \frac{c \gamma^2}{n} \sum_{i=1}^n \norm{h_i^k - \nabla f_i(x^\star)}_2^2
\end{align*}
for $c \geq \frac{4 \omega}{\alpha n}$
and assume $\gamma \leq \frac{2}{(\mu + L)(1+\frac{2\omega}{n} + c\alpha)}$ and $\gamma \leq \frac{\alpha}{2\mu}$. Then under Assumption~\ref{assumption:diana}:
\begin{align}
 \E{\Psi^k} \leq  (1- \gamma \mu)^k \Psi^0 + \frac{2}{\mu(\mu+L)} \sigma^2 \,. \label{eq:h98gf98f}
\end{align}
\end{theorem}
\begin{corollary}
Let $c = \frac{4 \omega}{\alpha n}$, $\alpha = \frac{1}{\omega+1}$ and $\gamma = \min\left\{ \frac{2}{(\mu + L)(1+\frac{6\omega}{n})}, \frac{1}{2\mu(\omega +1)} \right\}$. Furthermore, define $\kappa = \frac{L+\mu}{2\mu}$. Then the conditions of Theorem~\ref{thm:improvedDIANA} are satisfied and the leading term in the iteration complexity bound is 
\[
\frac{1}{\gamma \mu} = \kappa + \kappa \frac{2\omega}{n} + 2(\omega+1)\,.
\]
\end{corollary}
\begin{remark}
For the special case of quantization~\eqref{eq:Q} in arbitrary $p$-norms and $s=1$, this result recovers~\cite{Mishchenko2019:diana} up to small differences in the constants. 
\end{remark}

\section{Variance Reduction for Quantized Updates}

\begin{figure}[th!]
\vspace{-7.8pt}
\begin{algorithm}[H]
\removelatexerror
\begin{algorithm2e}[H]
\DontPrintSemicolon
		\KwIn{learning rates $\alpha > 0$ and $\gamma > 0$, initial vectors $x^0, h_{1}^0, \dots, h_{n}^0$, $h^0 = \frac{1}{n}\sum_{i=1}^n h_i^0$}
		\For{$k = 0,1,\ldots$}{
			sample random \Comment*[r]{only for Variant 1}
			$
				u^k = \begin{cases}
					1,& \text{with probability } \frac{1}{m}\\
					0,& \text{with probability } 1 - \frac{1}{m}\\
				\end{cases}
			$\;
			broadcast $x^k$, $u^k$ to all workers\;
			\For(in parallel \hfill $\triangleright$\ \textit{worker side}){$i = 1, \ldots, n$}{
				pick random $j_i^k \sim_{\rm u.a.r.} [m]$\;
				$\mu_i^k = \frac{1}{m} \sum\limits_{j=1}^{m} \nabla f_{ij}(w_{ij}^k)$\label{ln:mu} \;
				$g_i^k = \nabla f_{ij_i^k}(x^k) - \nabla f_{ij_i^k}(w_{ij_i^k}^k) + \mu_i^k$\;
				$\hat{\Delta}_i^k = Q(g_i^k - h_i^k)$\;
				$h_i^{k+1} = h_i^k + \alpha \hat{\Delta}_i^k$\;
				\For{$j = 1, \ldots, m$}{
					\Comment*[l]{Variant 1 (L-SVRG): update epoch gradient if $u^k = 1$}
					$
					w_{ij}^{k+1} =
					\begin{cases}
						x^k, & \text{if } u^k = 1 \\
						w_{ij}^k, &\text{if } u^k = 0\\
					\end{cases}
					$\;
					\Comment*[l]{Variant 2 (SAGA): update gradient table}
					$
					w_{ij}^{k+1} =
					\begin{cases}
					x^k, & j = j_i^k\\
					w_{ij}^k, & j \neq j_i^k\\
					\end{cases}
					$
				}
			}
			$h^{k+1} \! = \! h^k \!+\! \frac{\alpha}{n} \displaystyle \sum_{i=1}^n \hat{\Delta}_i^k$ \Comment*[r]{gather quantized updates} 
			$g^k = \frac{1}{n}\sum\limits_{i=1}^{n} (\hat{\Delta}_i^k + h_i^k)$\;
			$x^{k+1} = x^k - \gamma g^k$\;
		}
\end{algorithm2e}  
\caption{VR-DIANA based on L-SVRG (Variant 1), SAGA (Variant 2)}
\label{alg:VR-DIANA}
\end{algorithm}
\end{figure}

We now move to the main contribution of this paper and present variance reduced methods with quantized gradient updates. In this section we assume that each component $f_i$ of $f$ in~\eqref{eq:probR} has finite-sum structure:
\begin{equation}
f_i(x) = \frac{1}{m}\sum\limits_{j=1}^{m} f_{ij}(x)\,, \label{eq:finsum}
\end{equation}
The assumption that the number of components $m$ is the same for all functions $f_i$ is made for simplicity only.\footnote{This may seem limiting, but if this was not the case our analysis would still hold, but instead of $m$, we would have $\max_{i \in [n]} m_i$ appearing in the rates, where $m_i$ is the number of functions on the $i$-th machine. This suggests that we would like to have the functions distributed equally on the machines.
}
As can seen from \eqref{eq:h98gf98f}, one of the main disadvantages of Algorithm~\ref{alg:improvedDIANA} is the fact that we can only guarantee linear convergence to a $ \frac{2}{\mu(\mu+L)}  \frac{\sigma^2}{n}$-neighborhood of the optimal solution. In particular, the size of the neighborhood depends on the size of $\sigma^2$, which  measures the average variance of the stochastic gradients $g_i^k$ across the workers $i\in [n]$. In contrast, variance reduced methods converge linearly for arbitrary accuracy $\epsilon > 0$.

\subsection{The main challenge}
Let us recall that the variance reduced method SVRG~\cite{Johnson:2013svrg} computes the full gradient $\nabla f(x)$ in every epoch. To achieve this in the distributed setting~\eqref{eq:probR}, each worker $i$ must compute the gradient $\nabla f_i(x)$ and send this vector (exactly) either to the master node or broadcast it to all other workers. It might be tempting to replace this expensive communication step with quantized gradients instead, i.e.\ to rely on the aggregate $y_Q := \frac{1}{n}\sum_{i=1}^n Q(\nabla f_i(x))$ instead. However, the error $\norm{y_Q-\nabla f(x)}$ can be arbitrarily large (e.g.\ it will even not vanish for $x=x^\star$) and this simple scheme \emph{does not} achieve variance reduction (cf. also the discussion in~\cite{kuenstner2017:svrg}). Our approach to tackle this problem is via the {\em quantization of gradient differences}. Similarly to the previous section, we propose that each worker maintains a state $h_i^k \in \R^d$, and only quantizes  gradient \emph{differences}; for instance update $h_i^{k+1} = h_i^k + Q(h_i^k - \nabla f_i(x))$ (we assume this for  ease of exposition, the actual scheme is slightly different). We can now set $h^k = \frac{1}{n}\sum_{i=1}^n h_i^k$, which turns out to be a much more robust estimator of $\nabla f(x)$. By means of proving that $\norm{h_i^k - \nabla f_i(x^\star)} \to 0$ for $(k \to \infty)$ we are able to derive the {\em first variance reduced method that only exchanges quantized gradient updates among workers.}

\subsection{Three new   algorithms}
We propose in total three variance reduced algorithms, which are derived from either SAGA (displayed in Algorithm~\ref{alg:VR-DIANA}, Variant 2), SVRG (Algorithm~\ref{alg:SVRG_DIANA} provided in the appendix) and L-SVRG (Algorithm~\ref{alg:VR-DIANA}, Variant 1), a variant of SVRG with random epoch length and
described in~\cite{Kovalev2019:svrg}.
We prove  global linear convergence in the strongly convex case and $\cO(1/k)$ convergence in convex and non-convex cases. Moreover, our analysis is very general in the sense that the original complexity results for all three algorithms can be obtained by setting $\omega = 1$ (no quantization).

\textbf{Comments on Algorithm~\ref{alg:VR-DIANA}.}
In analogy to Algorithm~\ref{alg:improvedDIANA}, each node maintains a state $h_i^k \in \R^d$ that aims to reduce the variance introduced by the quantized communication. In contrast to the variance reduced method in~\cite{Alistarh2017:qsgd} that required the communication of the uncompressed gradients for every iteration of the inner loop (SVRG based scheme), here we communicate only quantized vectors.

Each worker $i=1,\dots,n$ computes its stochastic gradient $g_i^k$ in iteration $k$ by the formula given by the specific variance reduction type (SVRG, SAGA, L-SVRG), such that $\E{g_i^k \mid x^k} = \nabla f_i(x^k)$. Subsequently, the DIANA scheme is applied. Each worker in  Algorithms~\ref{alg:VR-DIANA} and \ref{alg:SVRG_DIANA} maintains a \emph{state} $h_i^k \in \R^d$ and quantizes only the difference $g_i^k - h_i^k$. This quantized vector $\hat{\Delta}_i^k$ is then sent to the master node which in turn updates its local copy of $h_i^k$. Thus both the $i$-th worker and the master node have access to $h_i^k$ even though it has never been transmitted in full (but instead incrementally constructed from the $\hat{\Delta}_i^k$'s).

The algorithm based on SAGA maintains on each worker a table of gradients, $\nabla f_{ij}(w_{ij}^k)$ (so the computation of $\mu_i^k$ on line~\ref{ln:mu} can efficiently be implemented). The algorithms based on SAGA just need to store the epoch gradients $\frac{1}{m}\sum_{i=1}^n \nabla f_{ij}(w_ij^k)$ that needs only to be recomputed whenever the $w_{ij}^k$'s change. Which is either after a fixed number of steps in SVRG, or after a random number of steps as in L-SVRG. These aspects of the algorithm are not specific to quantization and we refer the readers to e.g.~\cite{Raj2018:ksvrg} for a more detailed exposition.

\begin{table*}[t]
\begin{center}
\begin{tabular}{|c||c|c|c|c|}
\hline
\bf \multirow{2}{*}{Algorithm} & \bf \multirow{2}{*}{$\omega$} & \bf Convergence rate & \bf Convergence rate & \bf Communication \\
&  & \bf strongly convex & \bf non-convex & \bf cost per iter. \\
\hline 
\hline
VR without   & \multirow{2}{*}{$1$} & \multirow{2}{*}{$\hat{\cO}\left(
	\kappa + m
	\right)$}  &  \multirow{2}{*}{$\cO\left(
	\frac{ m^{2/3}}{\varepsilon}
	\right)$}   &  \multirow{2}{*}{$\cO(dn)$}\\
quantization &   & & &\\
\hline
VR with random & \multirow{2}{*}{$\sqrt{d}$} & \multirow{2}{*}{$\hat{\cO}\left(
	\kappa + \kappa \frac{\sqrt{d}}{n} + m + \sqrt{d}
	\right)$}  &  \multirow{2}{*}{$\cO\left(
	\left(\frac{\sqrt{d}}{n}\right)^{1/2}\frac{m^{2/3}}{\epsilon} 
	\right)$}  &  \multirow{2}{*}{$\cO(n\sqrt{d})$}\\
 dithering ($p=2, s =1$) & & & &\\
\hline
VR with random & \multirow{2}{*}{$\frac{d}{r}$} &  \multirow{2}{*}{$\hat{\cO}\left(
	\kappa + \kappa \frac{d}{n} + m + d
	\right)$}  &  \multirow{2}{*}{$\cO\left(
	 \frac{d}{\sqrt{n}}\frac{m^{2/3}}{\epsilon}
	\right)$}  &  \multirow{2}{*}{$\cO(n)$} \\
 sparsification ($r= \text{const}$) & & & & \\
 \hline
VR with    &  \multirow{2}{*}{$n$} & \multirow{2}{*}{$\hat{\cO}\left(
	\kappa  + m + n
	\right)$}  &  \multirow{2}{*}{$\cO\left(
	\frac{m^{2/3}}{\epsilon}
	\right)$}  &  \multirow{2}{*}{$\cO(n^2)$}\\
 block quantization ($t = d/n^2$) & & & &\\
\hline
\end{tabular}
\end{center} 
\caption{This table compares variance reduced methods with different levels of quantization. The $\hat{\cO}$ notation omits $\log 1/\varepsilon$ factors and we assume $\omega \leq m$ ($\omega \leq m^{2/3}$ for non-convex case) for ease of presentation. For block quantization (last row), the convergence rate is identical to the method without quantization (first row) but the savings in communication is at least $d/n$, assuming $d\geq n$.
(in number of total coordinates, here we did for simplicity not consider further savings that the random dithering approach offers in terms of less bits per coordinate).
 This table shows that quantization is meaningful and can provide huge benefits, especially when communication is a bottleneck.}
\label{tab:algo-comparison}
\end{table*}

\section{Convergence of VR-DIANA (Algorithm~\ref{alg:VR-DIANA})}
We make the following technical assumptions (out of which only the first one is shared among all theorems in this section):
\begin{assumption}\label{assumption:VRdiana}
In problem~\eqref{eq:probR} assume the finite-sum structure~\eqref{eq:finsum} for each $f_i$. Further assume each function $f_{ij} \colon \R^d \to \R$ to be $L$-smooth.
\end{assumption}

\begin{assumption}\label{assumption:VRdianasc}
Assume each function $f_i \colon \R^d \to \R$ to be  $\mu$-strongly convex, $\mu > 0$ and each $f_{ij} \colon \R^d \to \R$ to be  convex.
\end{assumption}

\begin{assumption}\label{assumption:VRdianac}
Assume each function $f_{ij} \colon \R^d \to \R$ to be  convex.
\end{assumption}

We are now ready to proceed with the main theorems.

\subsection{Strongly convex case}

\begin{theorem}[Strongly convex case]\label{thm:VR-DIANA}
    Consider Algorithm~\ref{alg:VR-DIANA} with $\omega$-quantization $Q$,
	and step size $\alpha \leq \frac{1}{\omega + 1}$.
	For 
	$b = \frac{4(\omega+1)}{\alpha n^2}$,
	$c = \frac{16(\omega+1)}{\alpha n^2}$,
	$\gamma = \frac{1}{L\left(1 + 36(\omega+1)/n\right)}$,
	define the Lyapunov function
\begin{equation*}
	\psi^k = \norm{x^k - x^\star}_2^2 + b \gamma^2 H^k + c \gamma^2 D^k\,,
\end{equation*}
where
\begin{equation*}
	H^k = \sum\limits_{i=1}^{n} \norm{h_i^k - \nabla f_i(x^\star)}_2^2\,,
\end{equation*}
and
\begin{equation*}
	D^k =
	\sum\limits_{i=1}^{n} \sum\limits_{j=1}^{m}
		\norm{\nabla f_{ij}(w_{ij}^k) - \nabla f_{ij}(x^\star)}_2^2\,.
\end{equation*}
Then	
under Assumption~\ref{assumption:VRdiana} and \ref{assumption:VRdianasc}
	\begin{equation*}
		\E{\psi^{k+1}} \leq
		(1 - \rho)\psi^k,
	\end{equation*}
	where $\rho\eqdef \min\left\{
			\frac{\mu}{L\left(1 + 36\frac{\omega + 1}{n} \right)},
			 \frac{\alpha}{2},
			\frac{3}{8m}
		\right\}$ and the expectation is conditioned on the previous iterate.
\end{theorem}

\begin{corollary}
    Let $\alpha = \frac{1}{\omega + 1}$.
	To achieve precision $\E{\|x^k - x^\star\|_2^2} \leq \varepsilon \psi^0$ VR-DIANA needs
	$
	\cO\left(
	(\kappa + \kappa \frac{\omega}{n} + m + \omega)\log\frac{1}{\epsilon}
	\right)
	$
	iterations.
\end{corollary}

\begin{remark}
We would like to mention that even we do not consider the regularized problem in the second part, using our analysis, one can easily extend our result to non-smooth regularizer just by exploiting non-expansiveness of the proximal operator.
\end{remark}

Recall that variance reduced methods (such as SAGA or SVRG) converge at rate $\hat{\cO}(\kappa + m)$ in this setting (cf. Table~\ref{tab:algo-comparison}). The additional variance of the quantization operator enters the convergence rate in two ways: firstly, (i), as an additive component. However, in large scale machine learning applications the number of data samples $m$ on each machine is expected to be huge. Thus this additive component affects the rate only mildly.
Secondly, (ii), and more severely, as a multiplicative component $\frac{\kappa \omega}{n}$. However, we see that this factor is discounted by $n$, the number of workers. Thus by choosing a quantization operator with $\omega = \cO(n)$, this term can be controlled. In summary, by choosing a quantization operator with $\omega = \cO(\min\{n,m\})$, the rate of VR-DIANA becomes identical to the convergence rate of the vanilla variance reduced schemes without quantization. This shows the importance of algorithms that support arbitrary quantization schemes and that do not depend on a specific quantization scheme.

\subsection{Convex case}

Let us now look at the convergence under (standard) convexity assumption, that is, $\mu = 0$. Then by taking output to be some iterate $x^k$ with uniform probability instead of the last iterate, one gets the following convergence rate.

\begin{theorem}[Convex case]\label{thm:VR-DIANAweak}
Let Assumptions~\ref{assumption:VRdiana} and~\ref{assumption:VRdianac} hold, then a randomly chosen iterate  $x^a$ of Algorithm~\ref{alg:VR-DIANA}, i.e.\  $x^a \sim_{u.a.r.} \{x^0, x^1,\dots, x^{k-1}\}$ satisfies
\begin{equation*}
\E{f(x^a) - f^\star} \leq \frac{\psi_0}{2k\left(
		\gamma - 
		L\gamma^2
		\left[
			1 + \frac{36(\omega+1)}{n}	 		
		\right]
	\right)}, 
\end{equation*} 
where $k$ denotes the number of iterations.
\end{theorem}

\begin{corollary}
    Let $\gamma = \frac{1}{2L\sqrt{m}\left(1+36\frac{\omega + 1}{n}\right)}$, $b = \frac{2(\omega+1)}{\alpha n^2}$, $c= \frac{6(\omega+1)}{n^2}$ and $\alpha = \frac{1}{\omega + 1}$.
	To achieve precision $\E{f(x^a) - f^\star} \leq\varepsilon$ VR-DIANA needs
	$
	\cO\left(
	\frac{\left(1+\frac{\omega}{n}\right)\sqrt{m} + \frac{\omega}{\sqrt{m}}}{\epsilon}
	\right)
	$
	iterations.
\end{corollary}
Here we see that along the quantization operator is chosen to satisfy $\omega = \cO(\min\{m,n\})$ the convergence rate is not worsened compared to a scheme without quantization.

\subsection{Non-convex case}

Finally,  convergence guarantee in the non-convex case is provided by the following theorem.

\begin{theorem}\label{thm:VR-DIANAnc}
Let Assumption~\ref{assumption:VRdiana} hold. Moreover, let   $\gamma = \frac{1}{10L\left( 1 + \frac{\omega}{n} \right)^{1/2} (m^{2/3} + \omega + 1)}$ and $\alpha = \frac{1}{\omega+1}$, then  a randomly chosen iterate $x^a \sim_{u.a.r.} \{x^0, x^1,\dots, x^{k-1}\}$ of Algorithm~\ref{alg:VR-DIANA} satisfies
\begin{align*}
&\E{\norm{\nabla f(x^a)}_2^2} \leq \\ 
&\frac{40(f(x^0) - f^\star)L\left(1 +  \frac{\omega}{n} \right)^{1/2}(m^{2/3} + \omega + 1)}{k}, 
\end{align*}
where $k$ denotes the number of iterations.
\end{theorem}

\begin{corollary}
	To achieve precision $\E{\norm{\nabla f(x^a)}_2^2} \leq\varepsilon$ VR-DIANA needs
	$
	\cO\left(
	\left(1 + \frac{\omega}{n} \right)^{1/2} \frac{ m^{2/3} + \omega}{\varepsilon}
	\right)
	$
	iterations.
\end{corollary}

As long as $\omega\le m^{2/3}$, the iteration complexity above is $\cO(\omega^{1/2})$ assuming the other terms are fixed. At the same time, the communication complexity is proportional to the number of nonzeros, which for random dithering and random sparsification decreases as  $\cO(1/\omega)$. Therefore, one can trade-off iteration and communication complexities by using quantization. Some of these trade-offs are mentioned in Table~\ref{tab:algo-comparison} above.
\section{Experiments}
\begin{figure*}[t]
\centering
\subfigure[Real-sim, \texttt{$\lambda_2 = 6\cdot 10^{-5}$}]{\includegraphics[width=0.35\textwidth]{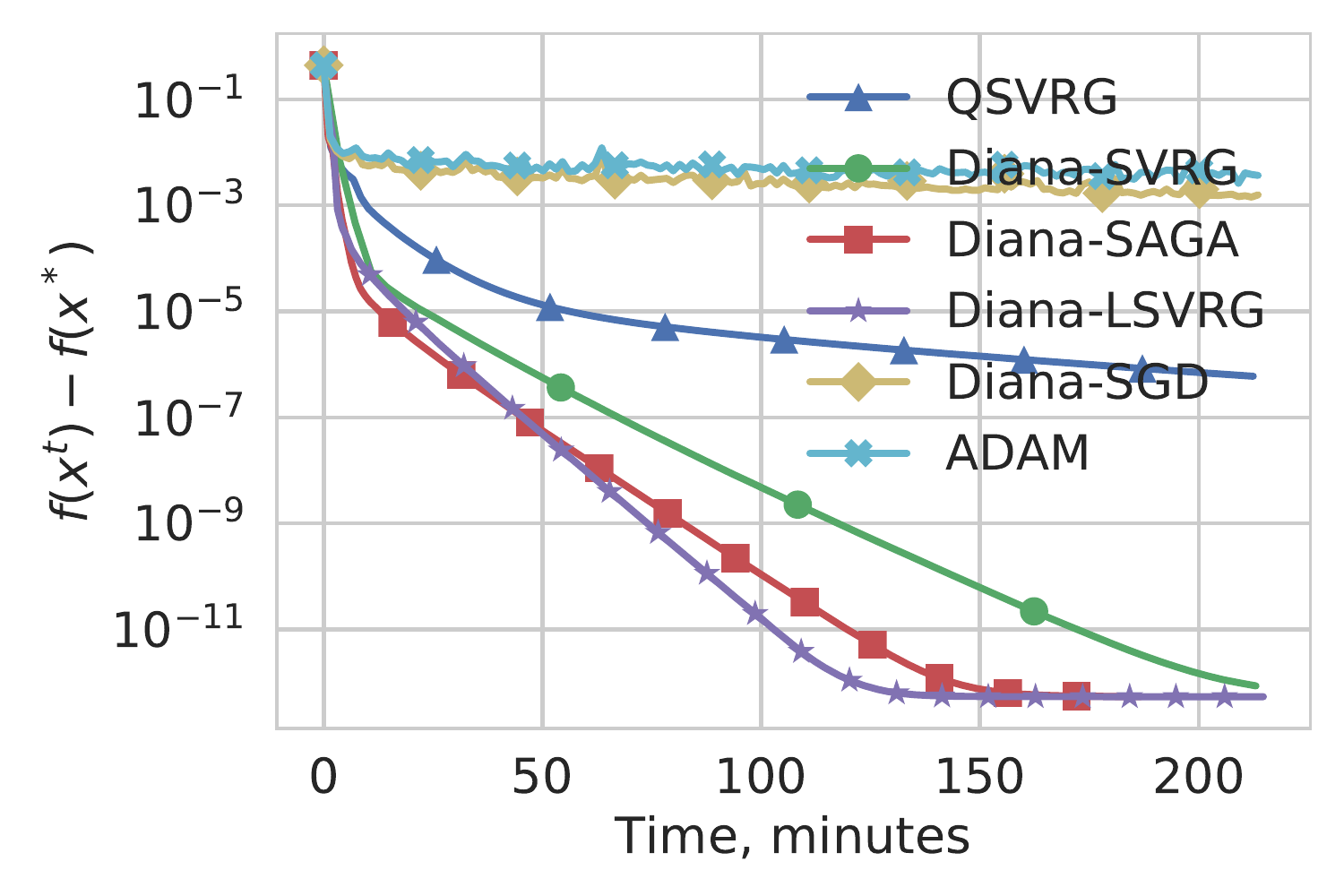}}
\subfigure[Real-sim, \texttt{$\lambda_2 = 6\cdot 10^{-5}$}]{\includegraphics[width=0.35\textwidth]{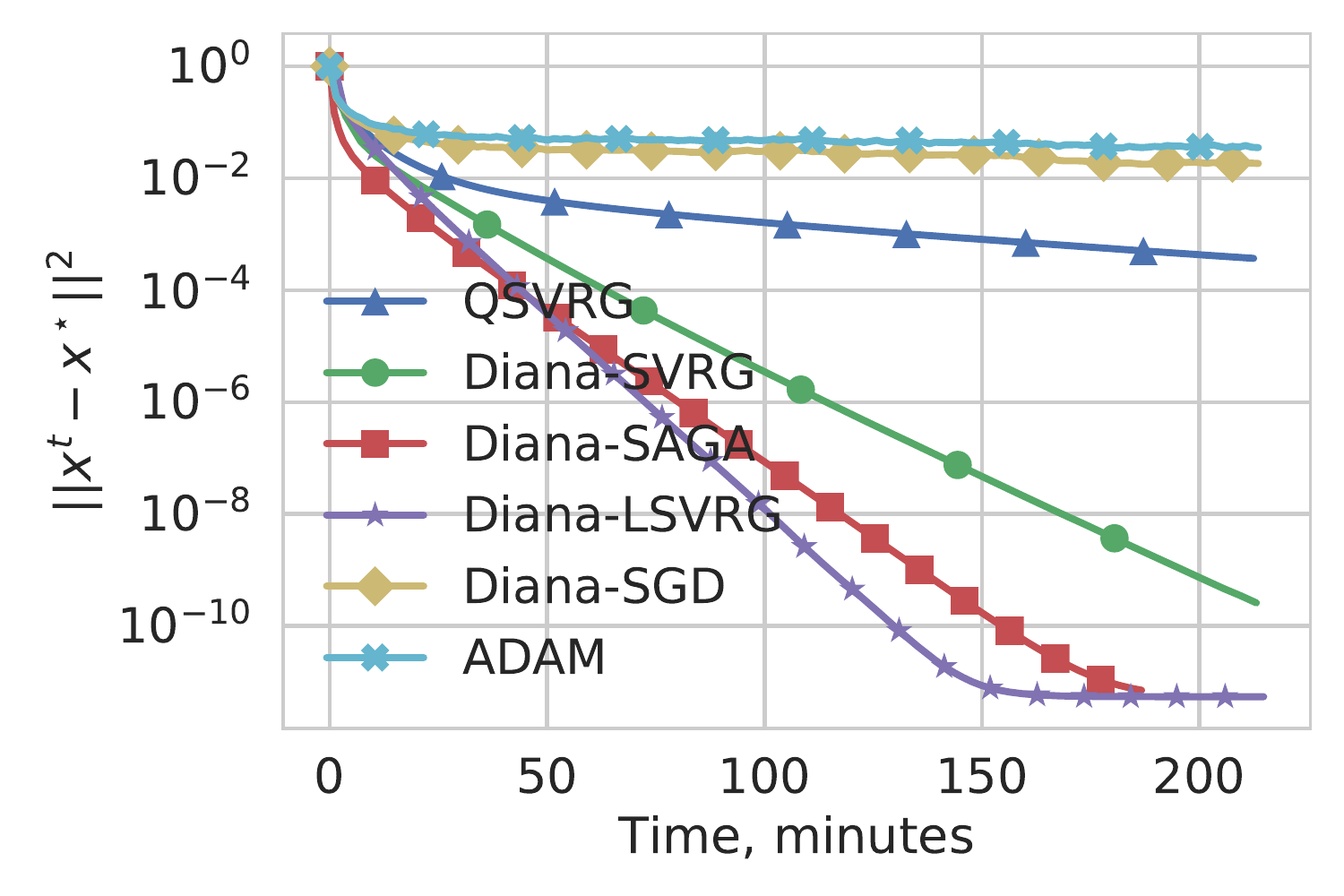}}
\caption{Comparison of VR-DIANA, Diana-SGD, QSVRG and TernGrad-Adam with $n=12$ workers on real-sim dataset, whose size is 72309 and dimension $d=20598$. We plot functional suboptimality on the left and distance from the optimum on the right. $\ell_{\infty}$ dithering is used for every method except for QSVRG, which uses $\ell_2$ dithering. We chose small value of $\lambda_2$ for this dataset to give bigger advantage to sublinear rates of Diana-SGD and TernGrad-ADAM, however, they are still much small than linear rates of variance reduced methods.}
\label{fig:with_adam}
\end{figure*}

\begin{figure*}[t!]
\centering
\subfigure[\texttt{SAGA}]
{\includegraphics[trim={3mm 0 3mm 0},clip,width=0.325\textwidth]{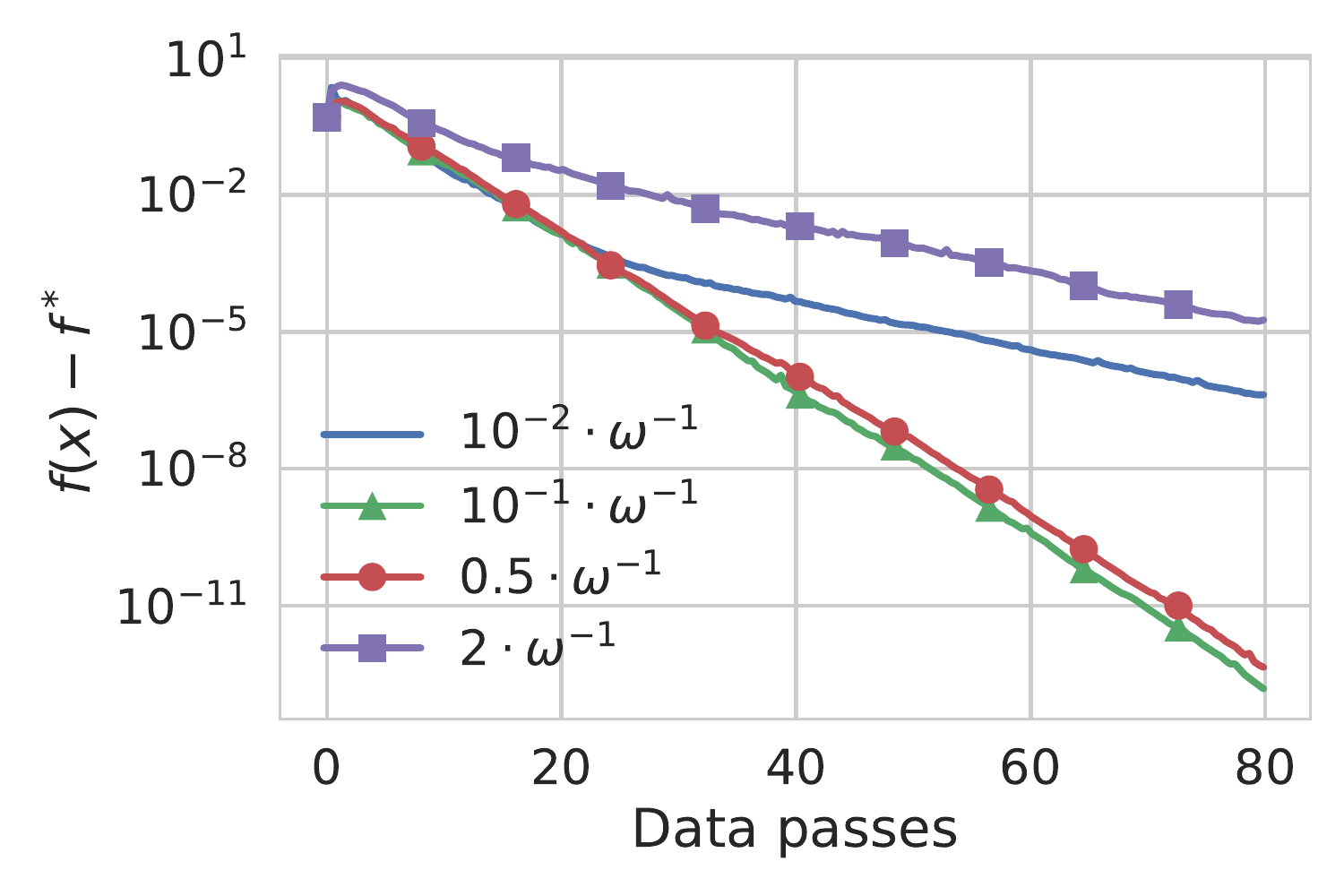}}
\subfigure[\texttt{SVRG}]
{\includegraphics[trim={3mm 0 3mm 0},clip,width=0.325\textwidth]{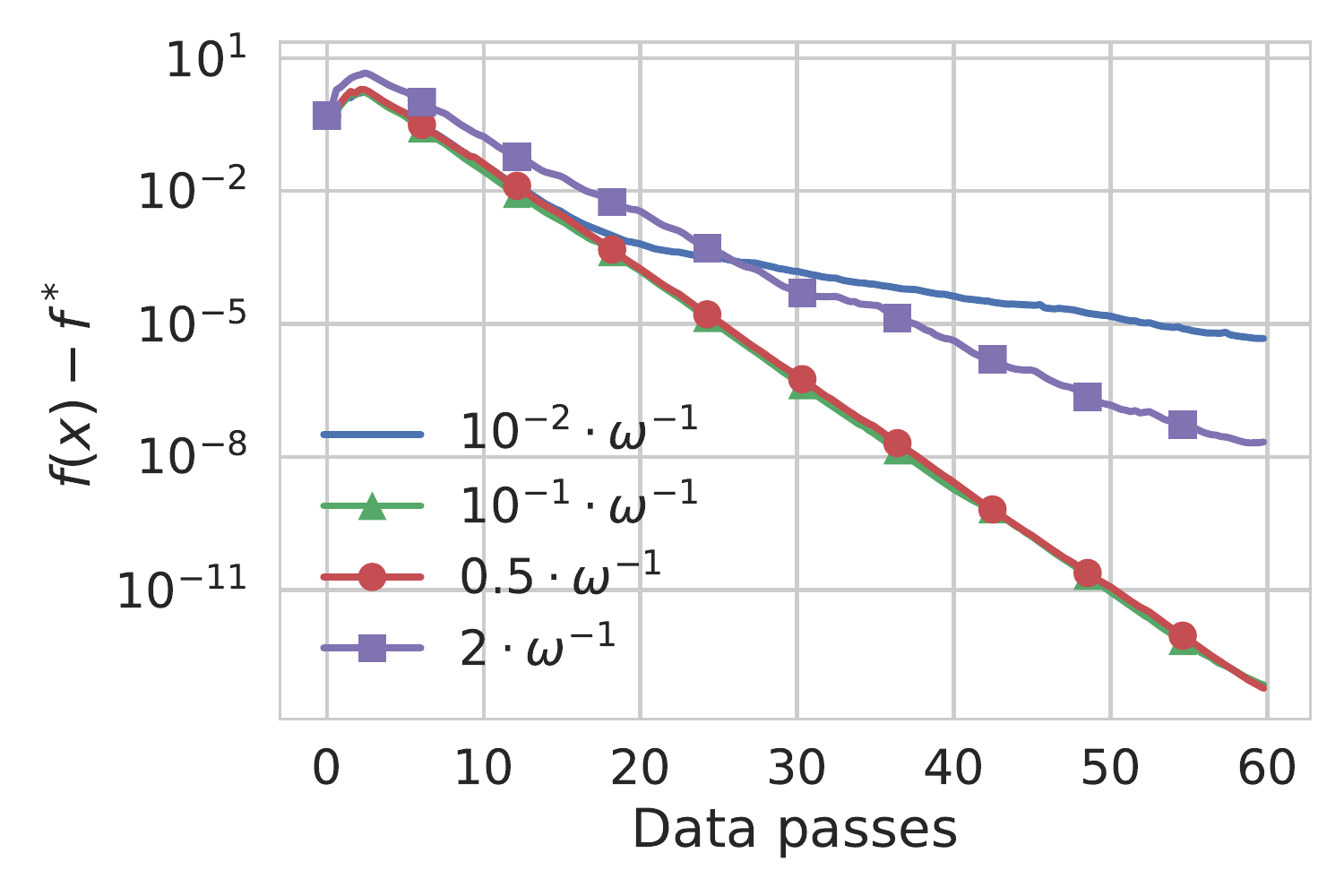}}
\subfigure[\texttt{L-SVRG}]
{\includegraphics[trim={3mm 0 3mm 0},clip,width=0.325\textwidth]{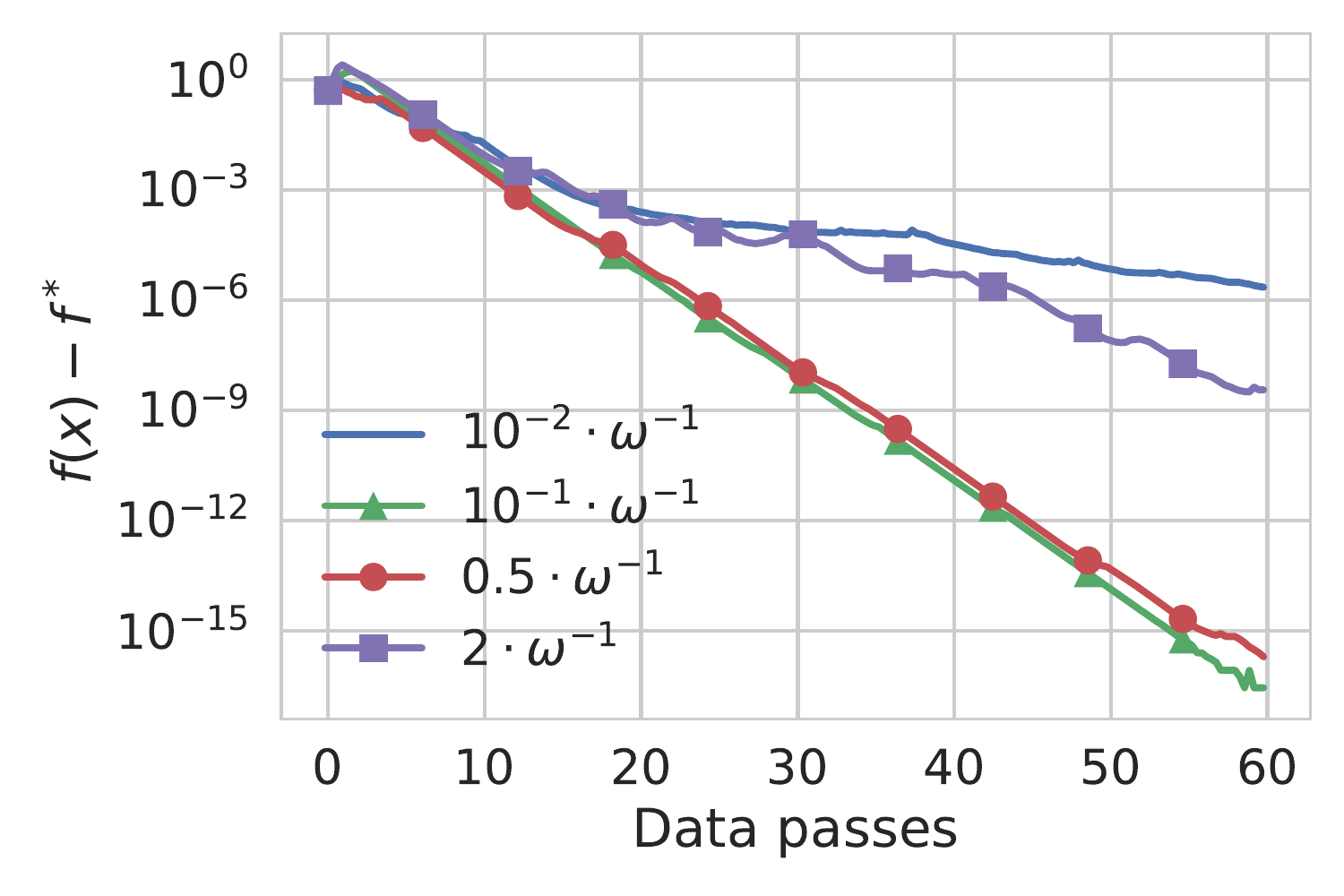}}
\caption{Comparison of VR methods with different parameter $\alpha$ for solving the Gisette dataset with block size 2000, $\ell_2$-penalty $\lambda_2=2\cdot 10^{-1}$, and $\ell_2$ random dithering.\label{fig:alpha_comparison}}
\end{figure*}

\begin{figure*}[t]
\centering
\subfigure[\texttt{Mushrooms,$\lambda_2 = 6\cdot 10^{-4}$}]{
\includegraphics[width=0.24\textwidth]{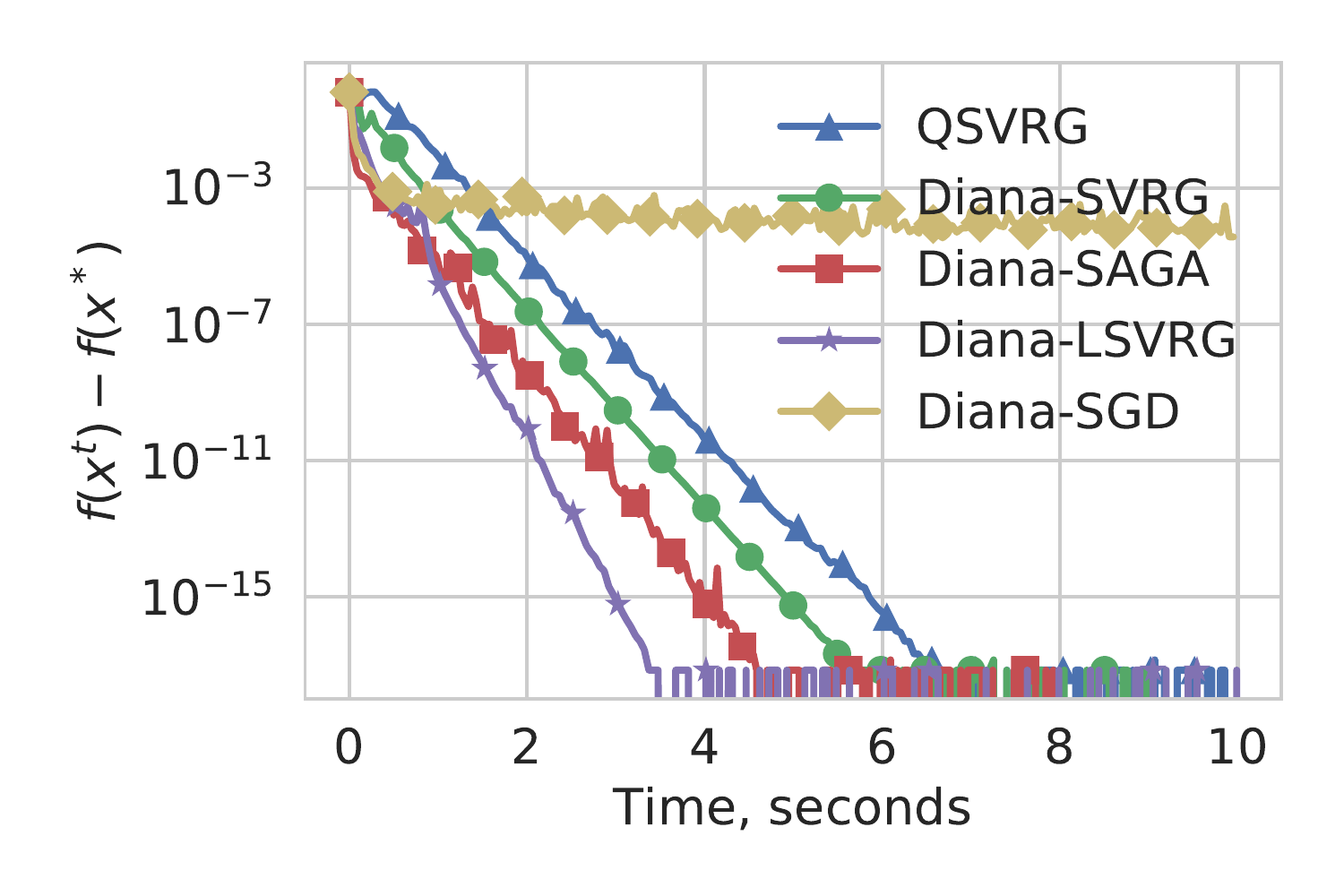}}
\hfill
\subfigure[\texttt{Mushrooms,$\lambda_2 = 6\cdot 10^{-5}$}]{\includegraphics[width=0.24\textwidth]{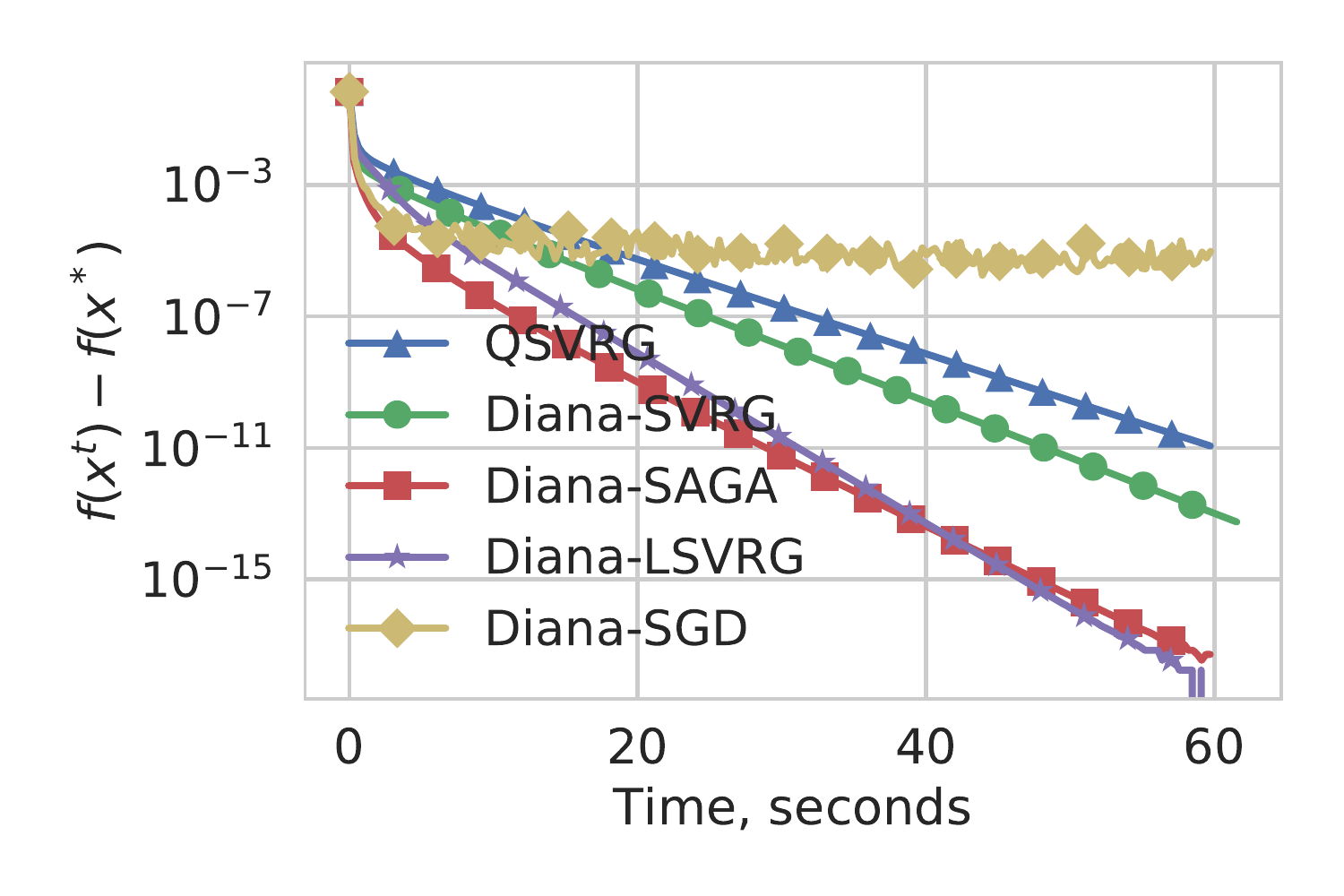}}
\hfill
\subfigure[\texttt{a5a,$\lambda_2 = 5\cdot 10^{-4}$}]{\includegraphics[width=0.24\textwidth]{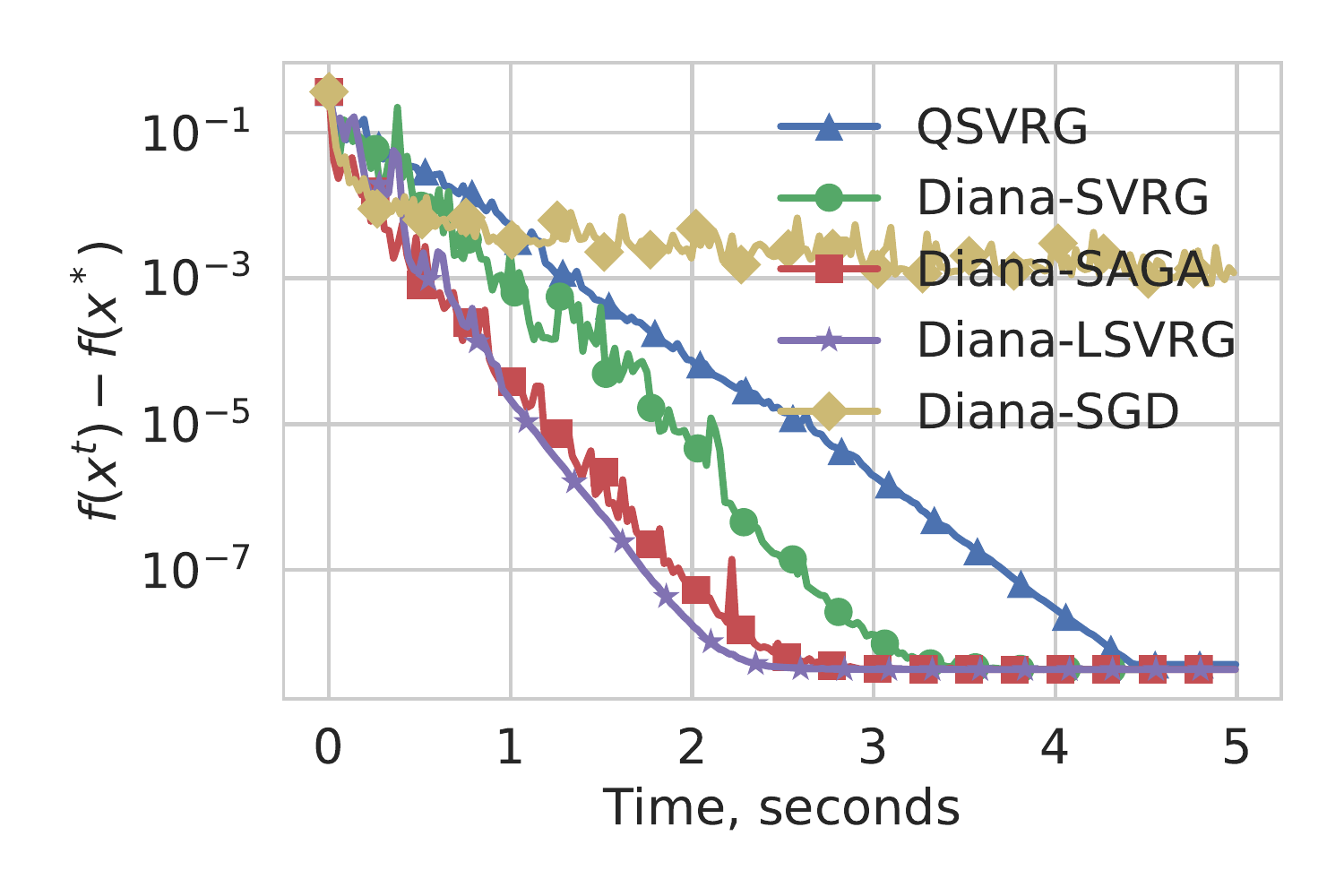}}
\hfill
\subfigure[\texttt{a5a,$\lambda_2 = 5\cdot 10^{-5}$}]{\includegraphics[width=0.24\textwidth]{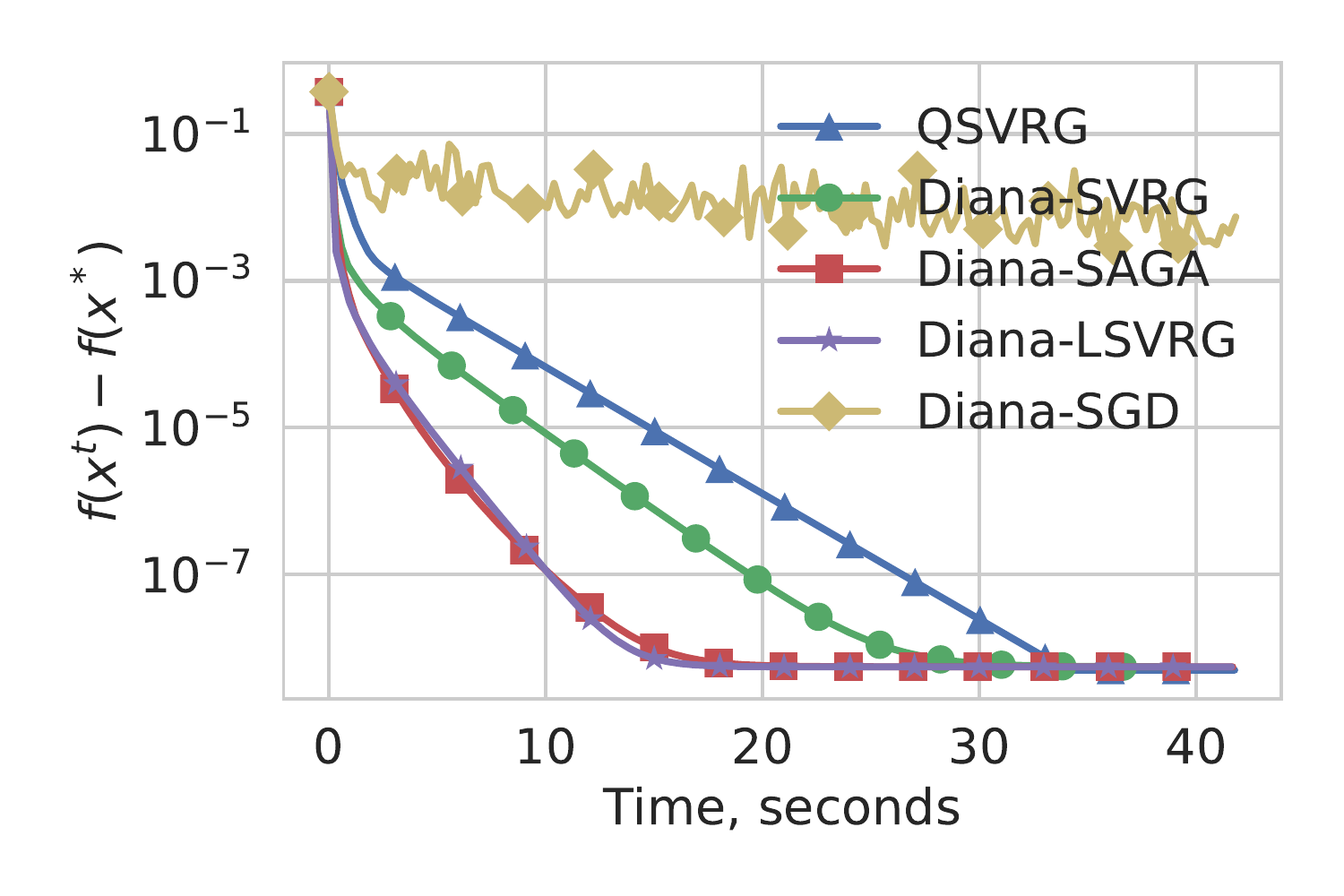}}\\
\subfigure[\texttt{Mushrooms,$\lambda_2 = 6\cdot 10^{-4}$}]{\includegraphics[width=0.24\textwidth]{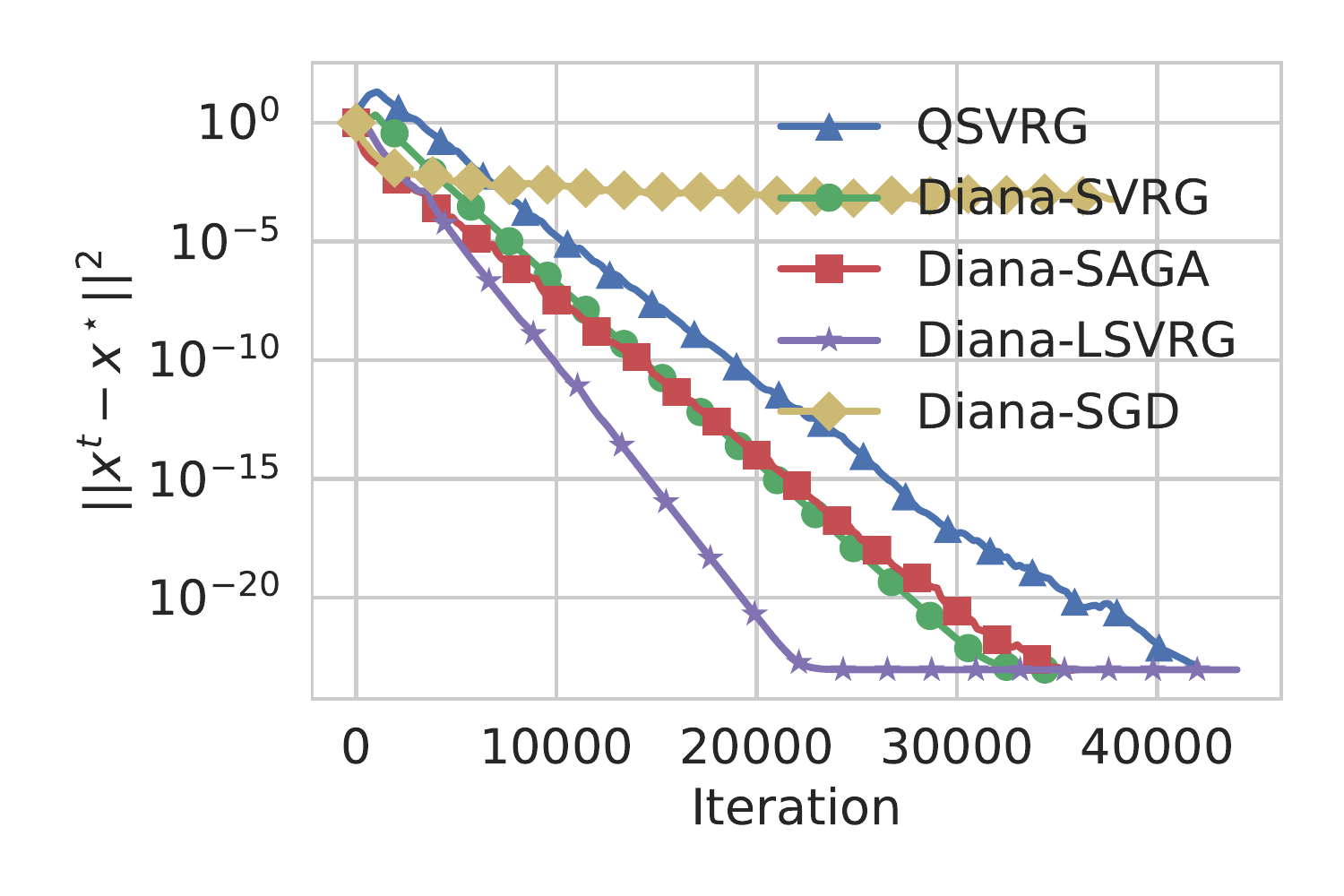}}
\hfill
\subfigure[\texttt{Mushrooms,$\lambda_2 = 6\cdot 10^{-5}$}]{\includegraphics[width=0.24\textwidth]{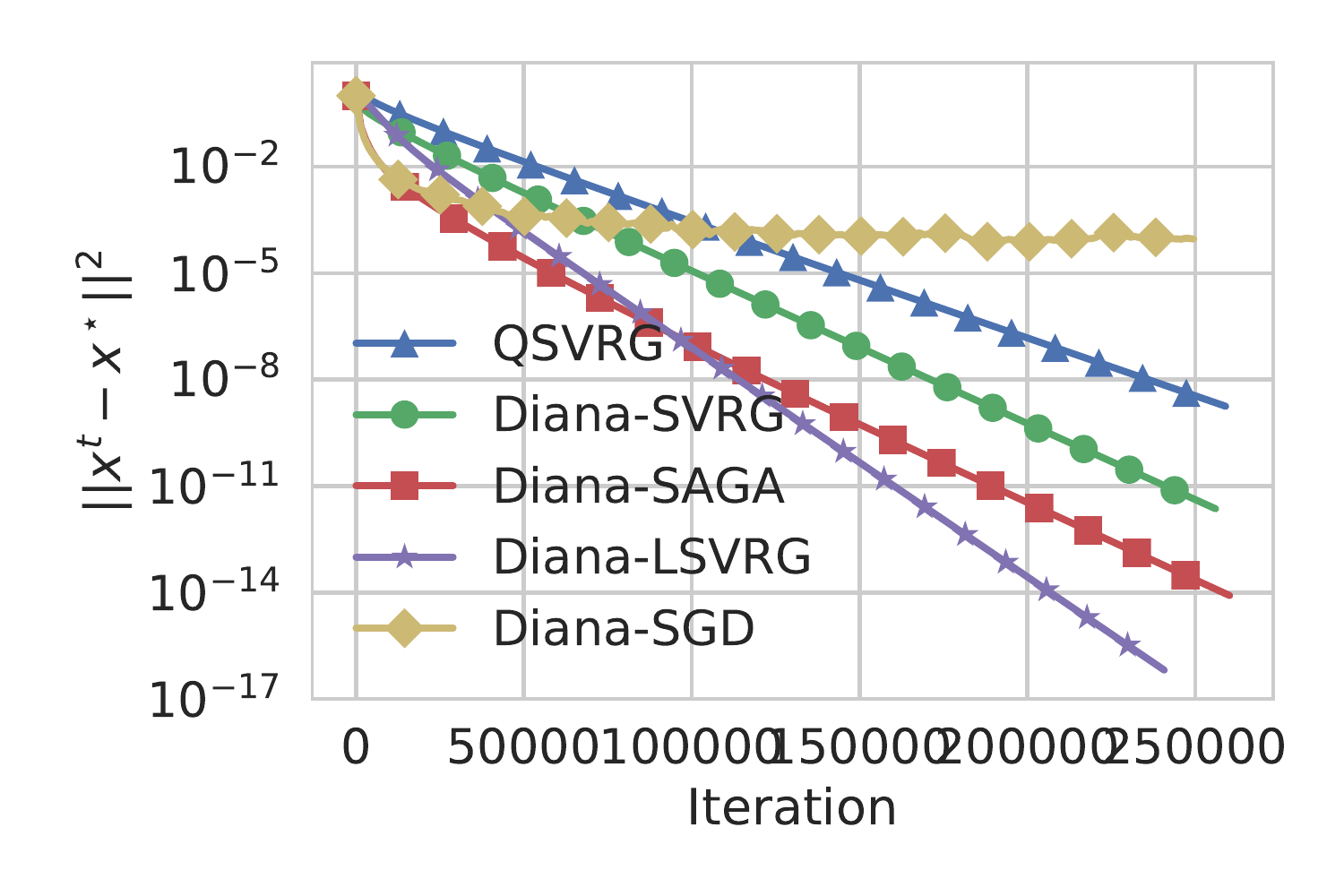}}
\hfill
\subfigure[\texttt{a5a,$\lambda_2 = 5\cdot 10^{-4}$}]{\includegraphics[width=0.24\textwidth]{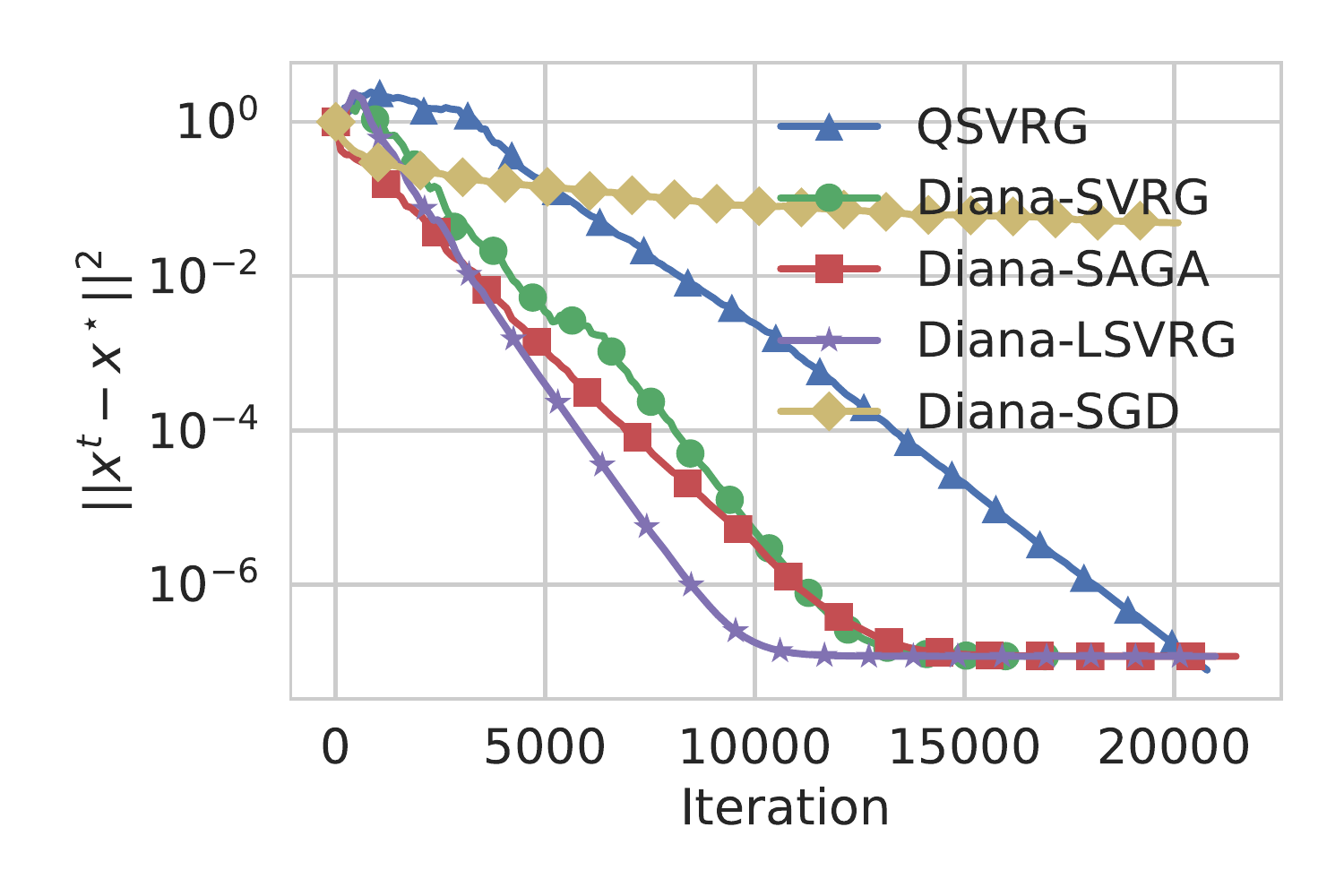}}
\hfill
\subfigure[\texttt{a5a,$\lambda_2 = 5\cdot 10^{-5}$}]{\includegraphics[width=0.24\textwidth]{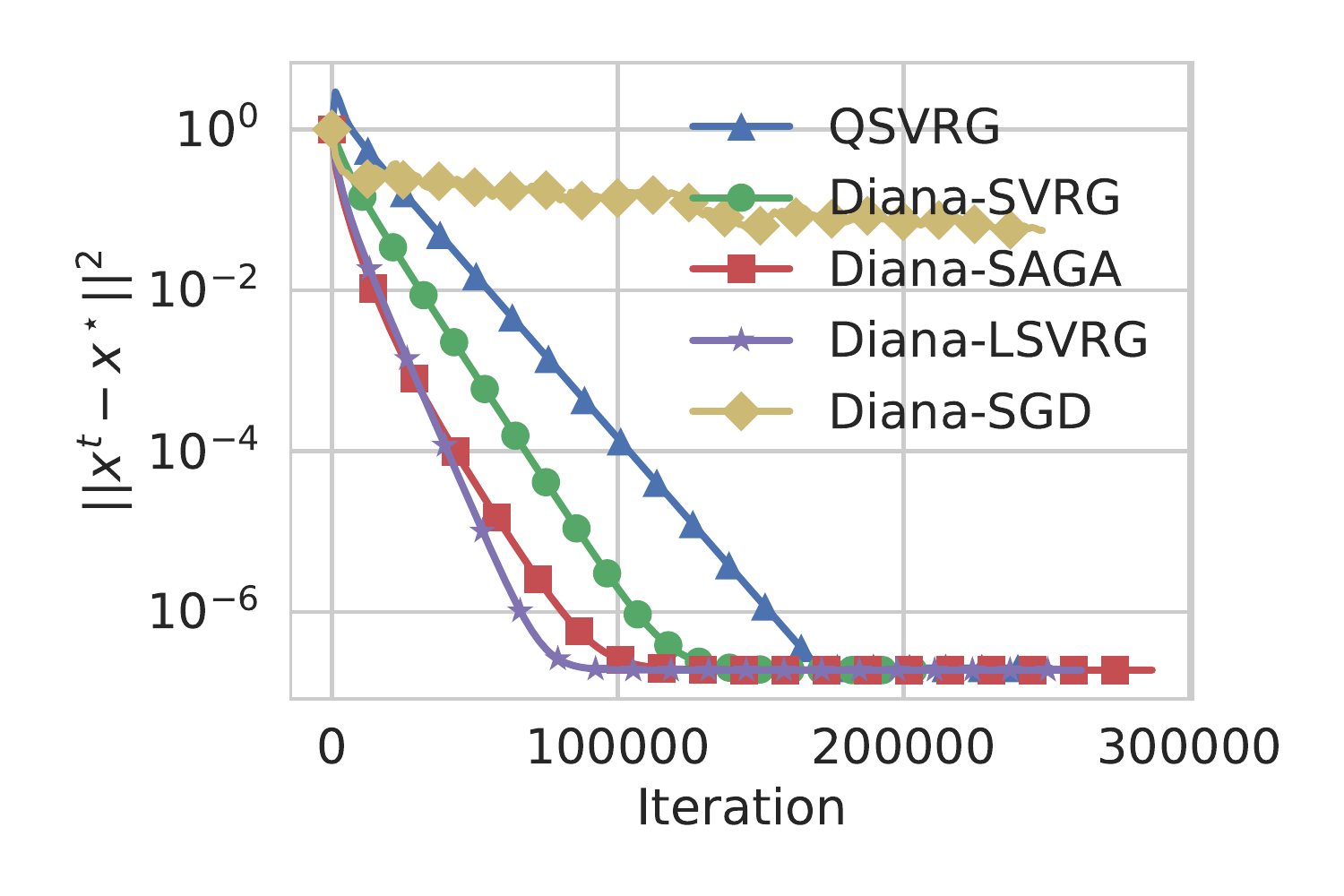}}
\caption{Comparison of VR-DIANA and Diana-SGD against QSVRG \cite{Alistarh2017:qsgd} on mushrooms (the first two columns) and a5a datasets (the last two columns). Plots in the first row show functional suboptimality over time and in the second row are the distances to the solution over iterations.}
\label{fig:qsgd}
\end{figure*}

We illustrate the performance of the considered methods on standard logistic regression problems for binary classification. We regularize the loss with $\ell_2$-penalty, so the full objective is $ \log(1 + \exp(-b_{ij} A_{ij}^\top x)) + \frac{\lambda_2}{2}\|x\|_2^2$, where $A_{ij}, b_{ij}$ are data points and $\lambda_2$ is the regularization parameter. This problem is attractive because some of the methods to which we want to compare do not have convergence guarantees for non-convex or non-smooth objective. $\lambda_2$ is set to be of order $1/(nm)$. The datasets used are from the LIBSVM library~\cite{CC01a}.

We implement all methods in Python using MPI4PY~\cite{dalcin2011parallel} for collective communication operations.
We run all experiments on a machine with 24 Intel(R) Xeon(R) Gold 6146 CPU @ 3.20GHz cores. The cores are connected to two sockets, 12 cores to each of them, and we run workers and the parameter server on different cores. The communication is done either using 32- or 64-bit numbers, so the convergence is correspondingly up to precision $10^{-8}$ or $10^{-16}$ in different experiments.

To have a trade-off between communication and iteration complexities, we use in most experiments the block quantization scheme with block sizes equal $n^2$, and in one additional experiments we explore what changes under different block sizes.
We analyze the effect of changing the parameter $\alpha$ over different values in Figure~\ref{fig:alpha_comparison} and find out that it is very easy to tune, and in most cases the speed of convergence is roughly the same unless $\alpha$ is set too close to 1 or to 0. For instance, we see almost no difference between choosing any element of $\{10^{-2}, 10^{-3}, 5\cdot 10^{-4}\}$ when $\omega^{-1} = 2\cdot 10^{-2}$, although for tiny values the convergence becomes much slower. For more detailed consideration of block sizes and choices of $\alpha$ on a larger dataset see Figure~\ref{fig:blocks_and_alphas}.

We also provide a comparison between most relevant quantization method: QSVRG~\cite{Alistarh2017:qsgd}, TernGrad-Adam and Diana~\cite{Mishchenko2019:diana} in Figure~\ref{fig:qsgd}. Since TernGrad-Adam was the slowest in our experiments, we provide it only in Figure~\ref{fig:with_adam}.

\section{Conclusion}
In this work we analyzed various distributed algorithms that support quantized communication between worker nodes. Our analysis is general, that is, not bound to a specific quantization scheme. This fact is especially interesting as we have showed that by choosing the quantization operator (respectively the introduced noise) in the right way, we obtain communication efficient schemes that converge as fast as their communication intensive counterparts. We develop the first variance reduced methods that support quantized communication and derive concise convergence rates for the strongly-convex, the convex and the non-convex setting.

\subsection*{Acknowledgments}

The authors would like to thank Xun Qian for the careful checking of the proofs and for  spotting several typos in the analysis. 


\small
\bibliography{papers}
\bibliographystyle{icml2019}

\normalsize

\appendix
\onecolumn

\section{Extra Experiments}
\begin{figure}[H]
	\subfigure[\texttt{$\ell_{\infty}$, $\lambda_2 = 2\cdot 10^{-1}$}]
	{\includegraphics[scale=0.22]{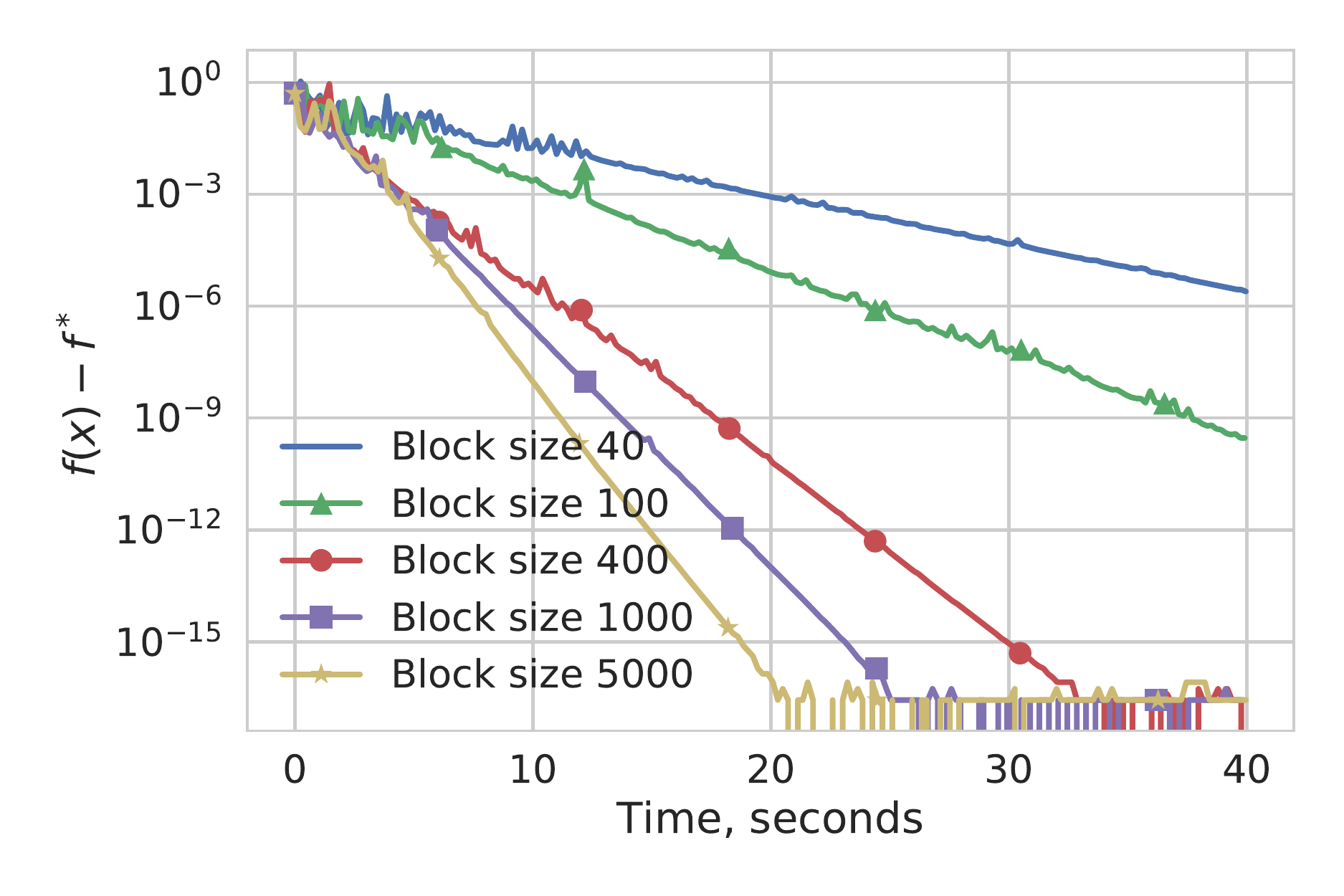}}
	\subfigure[\texttt{$\ell_{\infty}$, $\lambda_2 = 2\cdot 10^{-2}$}]
	{\includegraphics[scale=0.22]{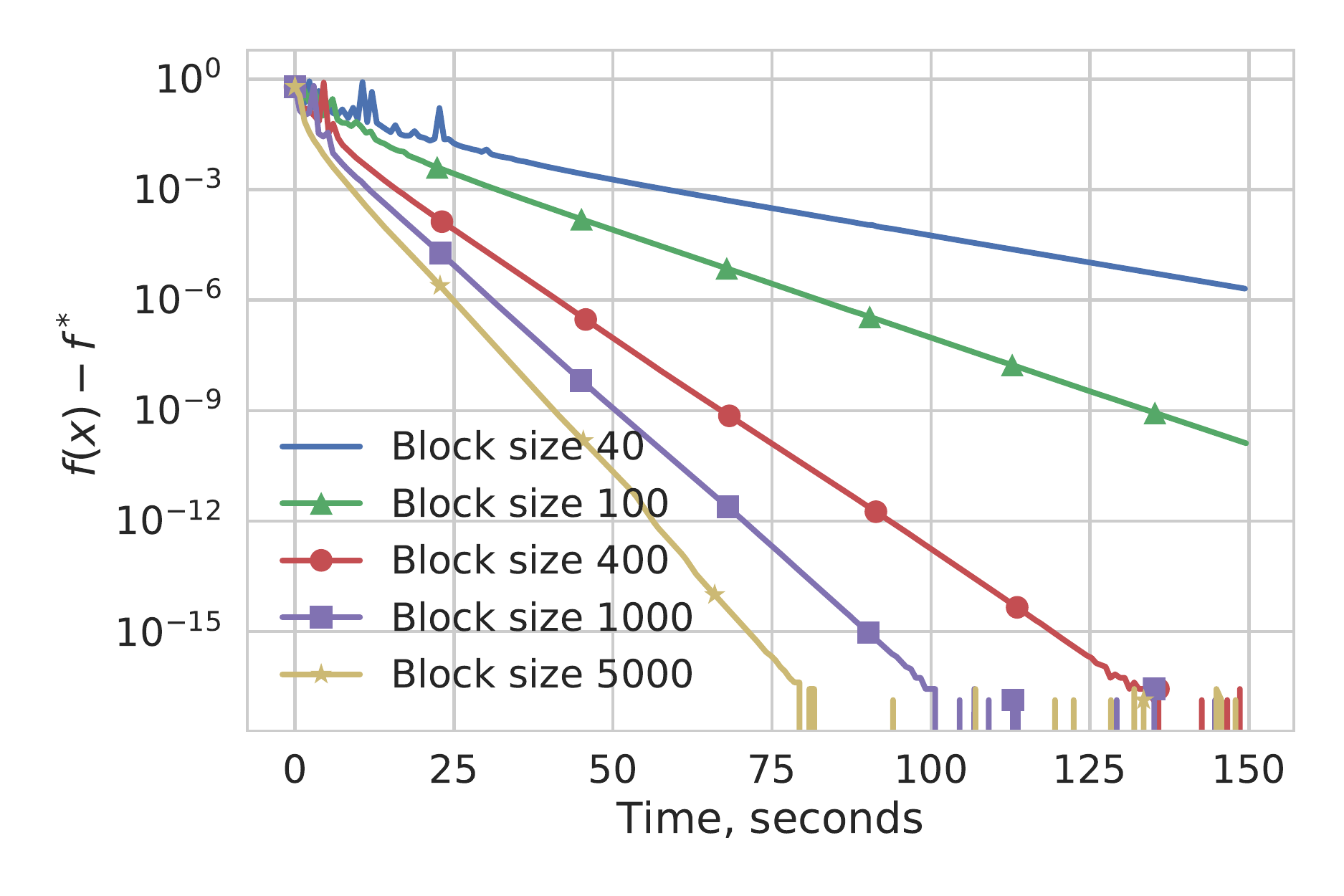}}
	\subfigure[\texttt{$\ell_{\infty}$, $\lambda_2 = 2\cdot 10^{-1}$}]
	{\includegraphics[scale=0.22]{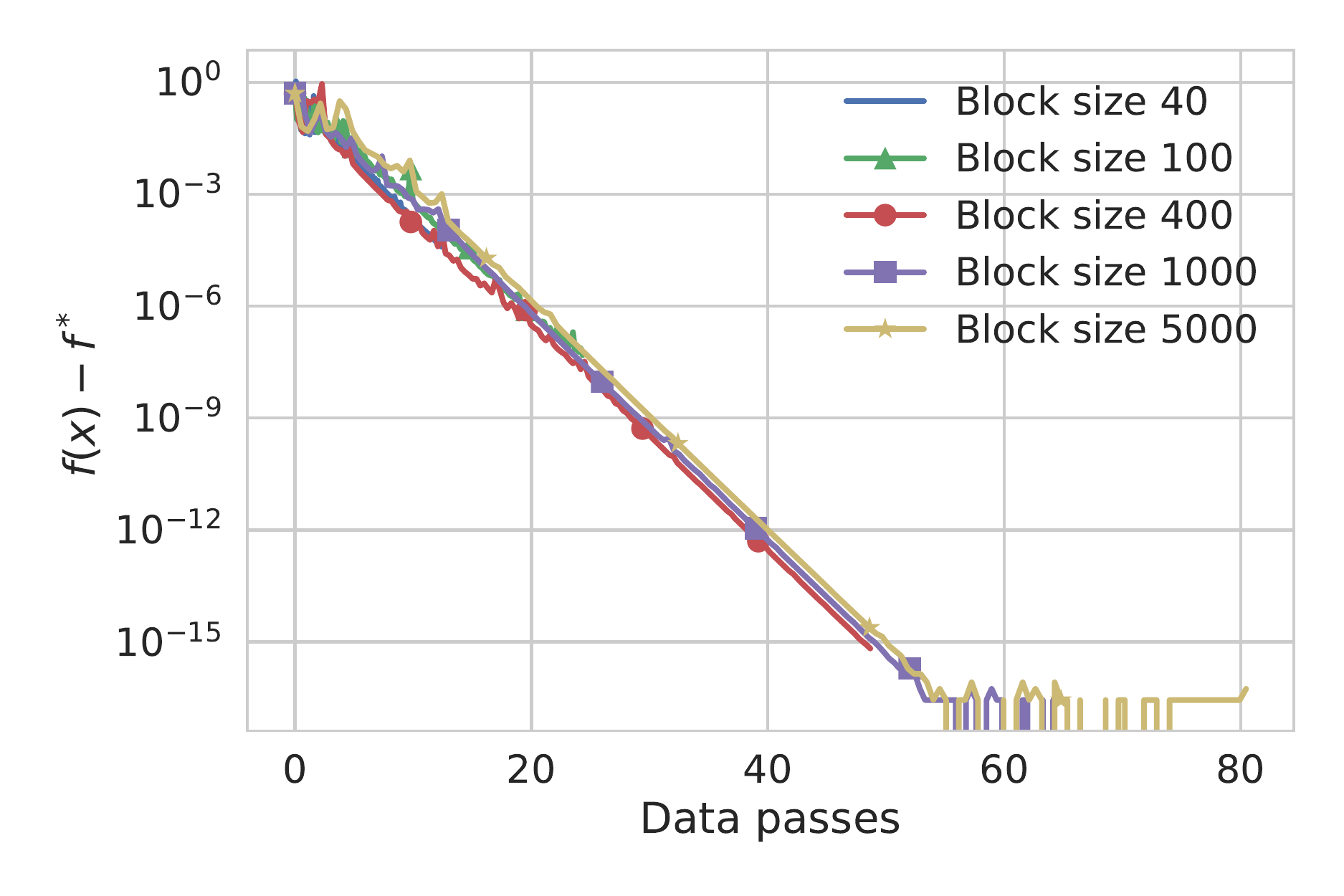}}
	\subfigure[\texttt{$\ell_{\infty}$, $\lambda_2 = 2\cdot 10^{-2}$}]
	{\includegraphics[scale=0.22]{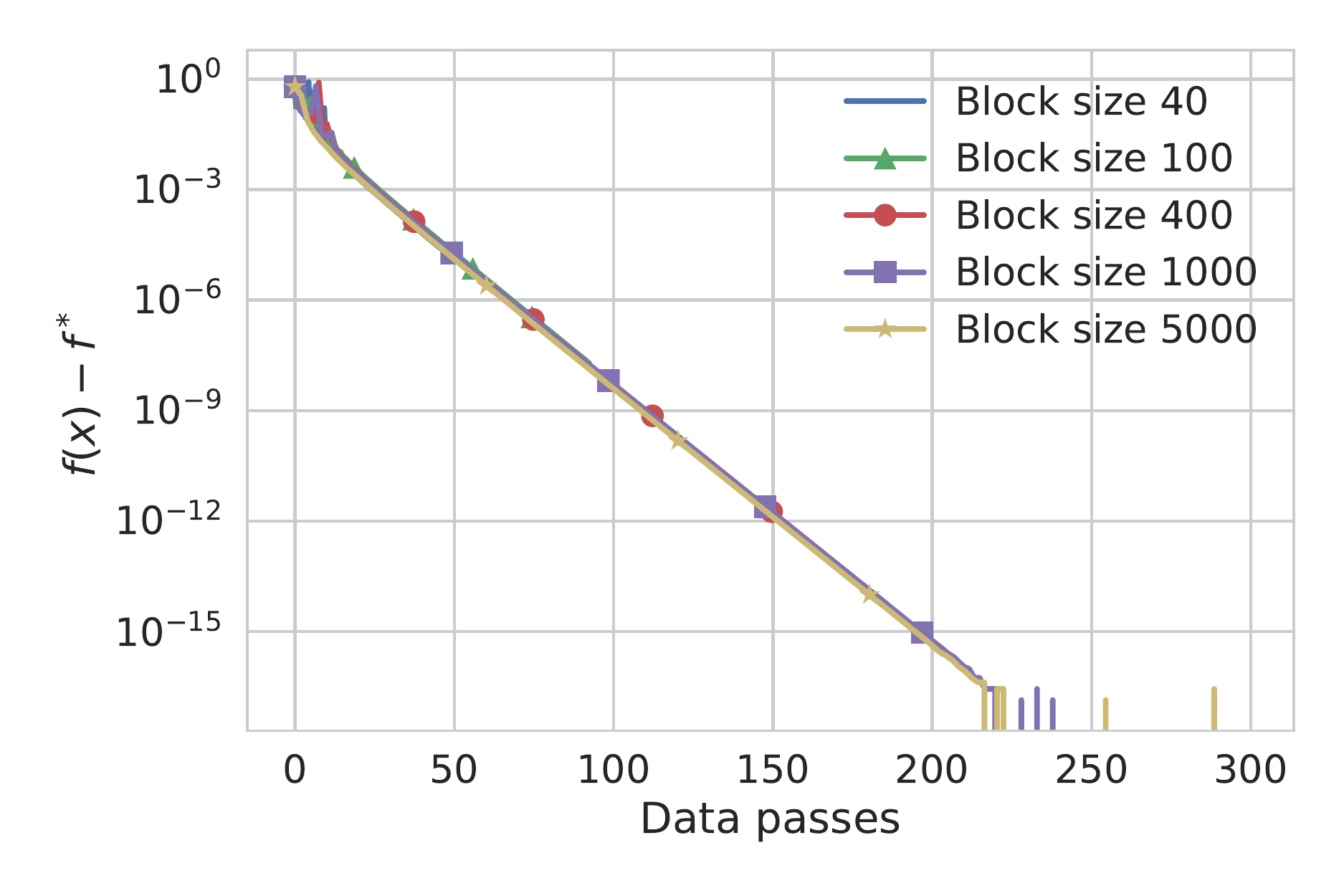}}\\
	\subfigure[\texttt{$\ell_{2}$, $\lambda_2 = 2\cdot 10^{-1}$}]
	{\includegraphics[scale=0.22]{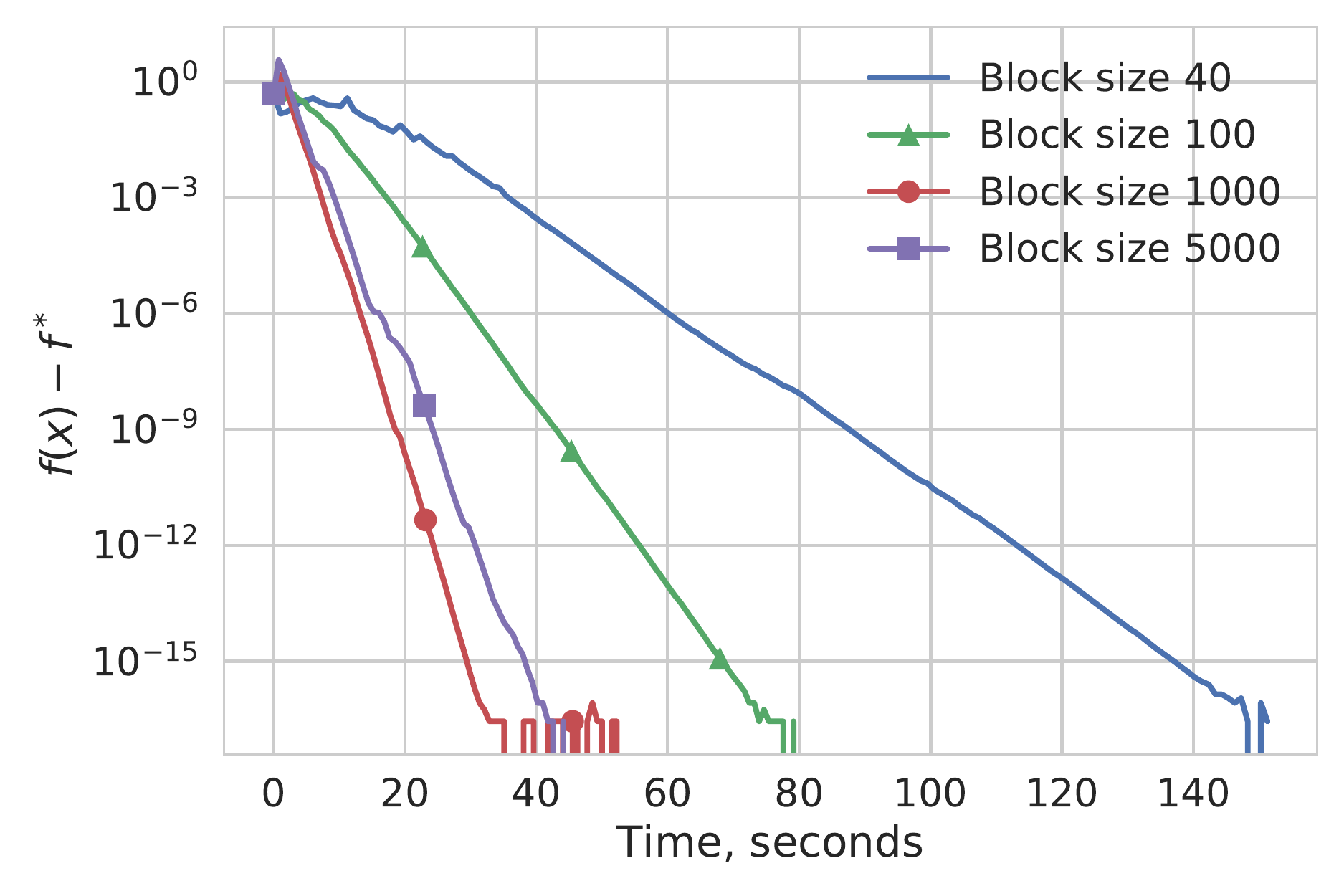}}
	\subfigure[\texttt{$\ell_{2}$, $\lambda_2 = 2\cdot 10^{-2}$}]
	{\includegraphics[scale=0.22]{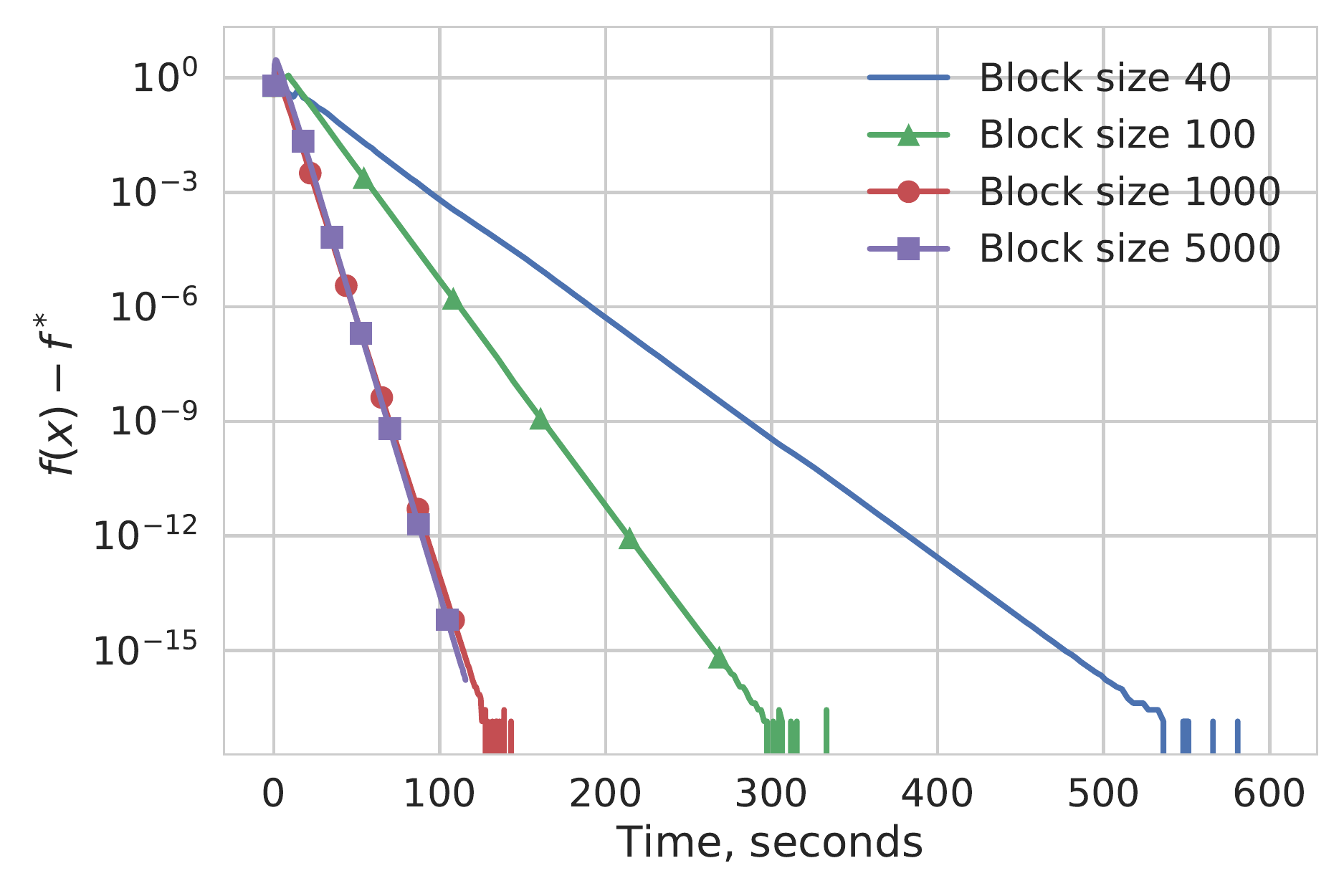}}
	\subfigure[\texttt{$\ell_{2}$, $\lambda_2 = 2\cdot 10^{-1}$}]
	{\includegraphics[scale=0.22]{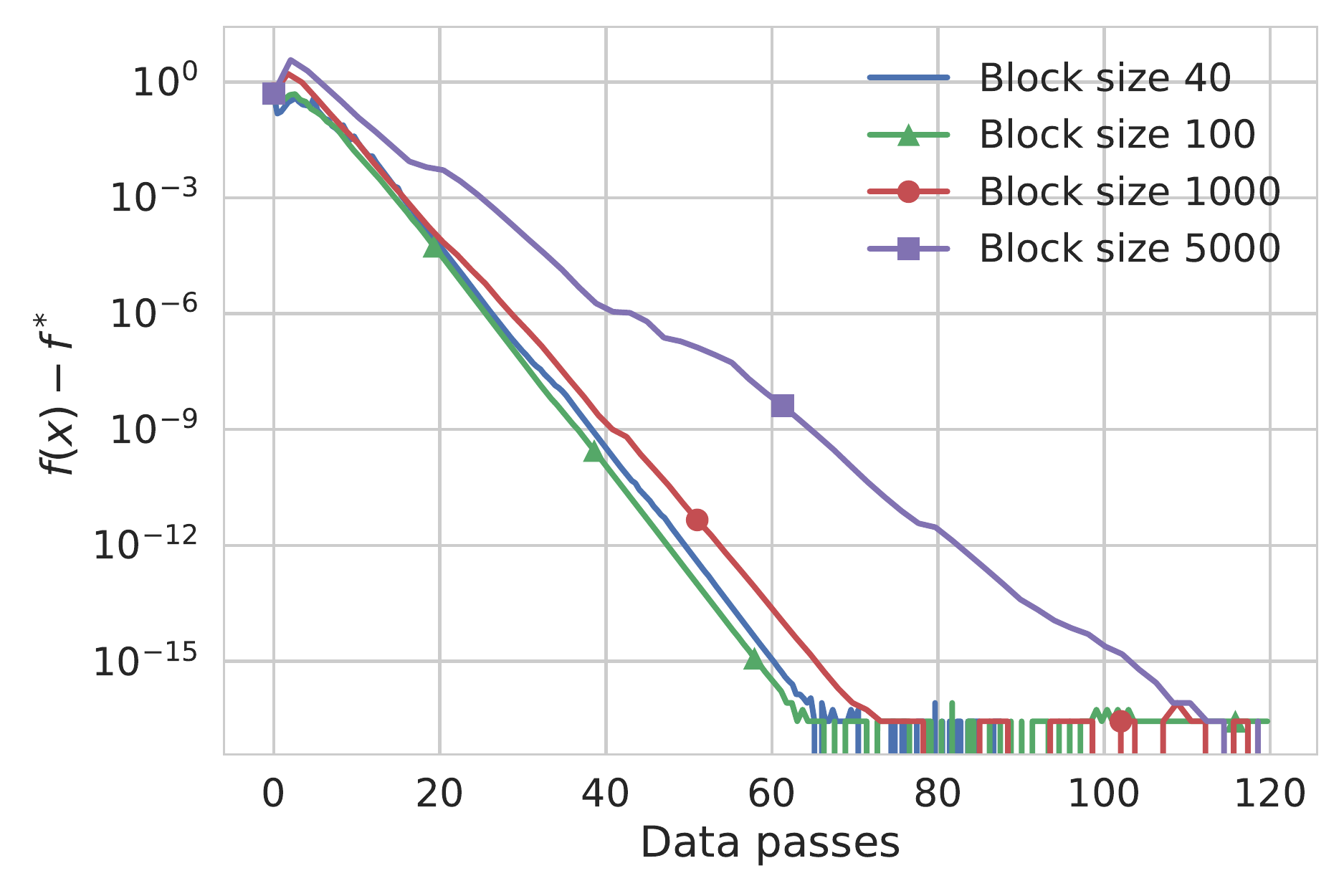}}
	\subfigure[\texttt{$\ell_{2}$, $\lambda_2 = 2\cdot 10^{-2}$}]
	{\includegraphics[scale=0.22]{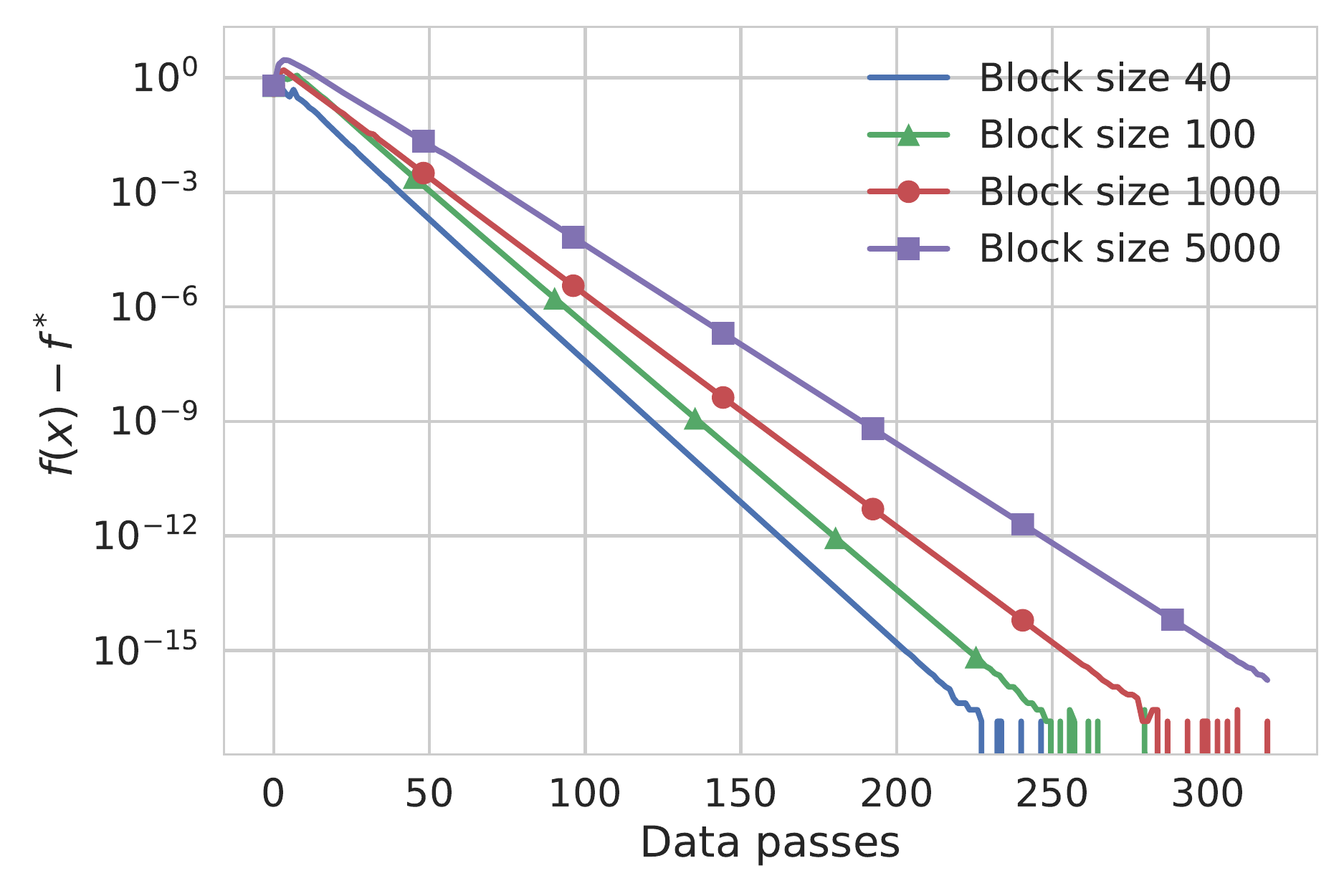}}
	\caption{Experiments with Diana-SVRG and different block sizes applied to the Gisette dataset ($d=5000$) with $n=20$ workers. In the first two columns we show convergence over time and in the last two we show convergence over epochs. We used 1-bit random dithering with $\ell_{\infty}$ (first row) and $\ell_2$ norm and found out that even quantization with full vector quantization often does not slow down iteration complexity, but helps significantly with communication time. At the same time, $\ell_2$ dithering is noisier and yields a more significant impact of the block sizes on the iteration complexity. For each line in the plots, an optimal stepsize was found and we found out that larger block sizes require slightly smaller steps in case of $\ell_2$ random dithering. In all cases, we chose $\alpha$ to be $\frac{1}{2\omega}$, where $\omega$ was computed using the block size.}
\end{figure}
\begin{figure}[H]
\begin{center}
	\subfigure[\texttt{SVRG}]
	{\includegraphics[scale=0.37]{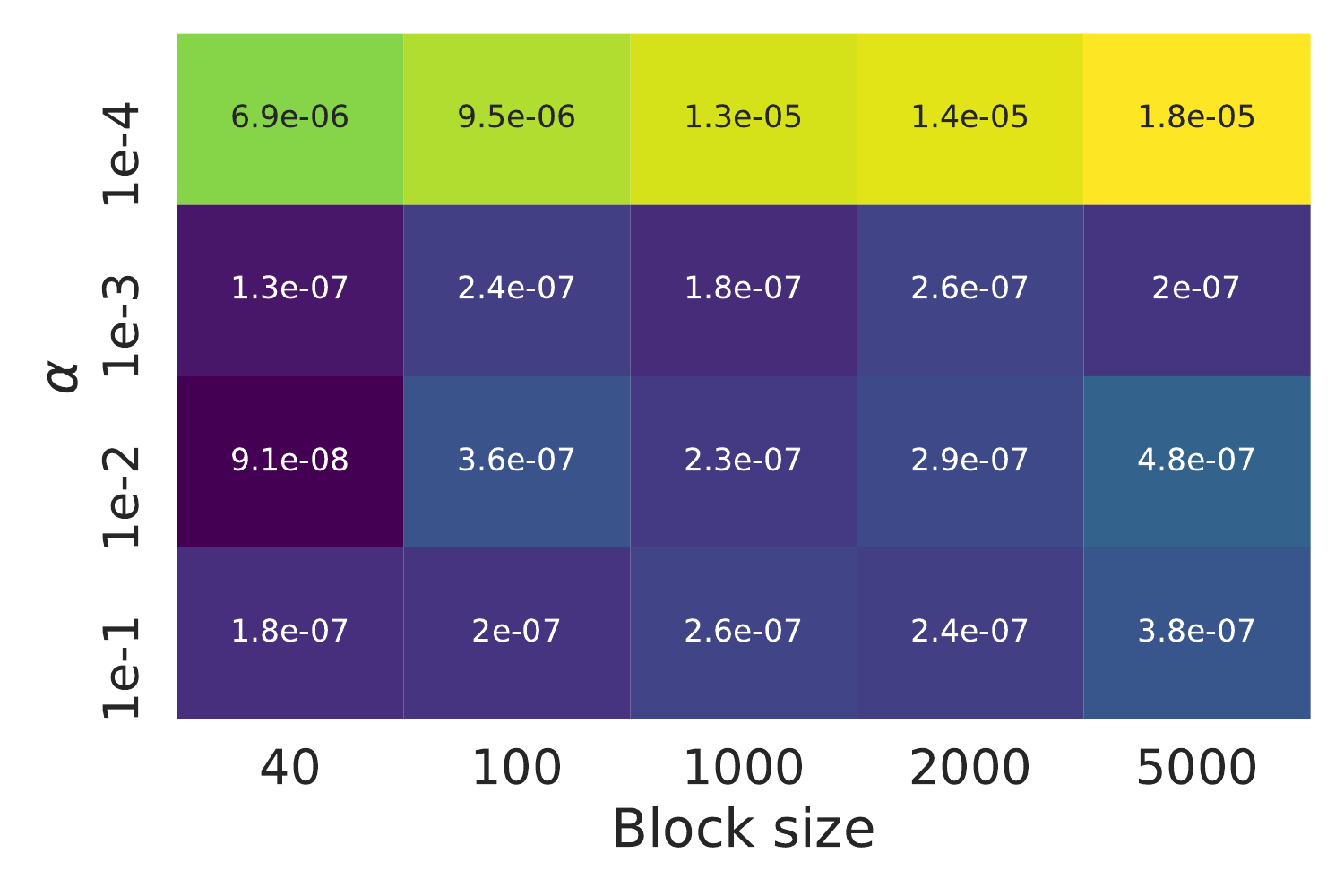}}
	\subfigure[\texttt{SAGA}]
	{\includegraphics[scale=0.37]{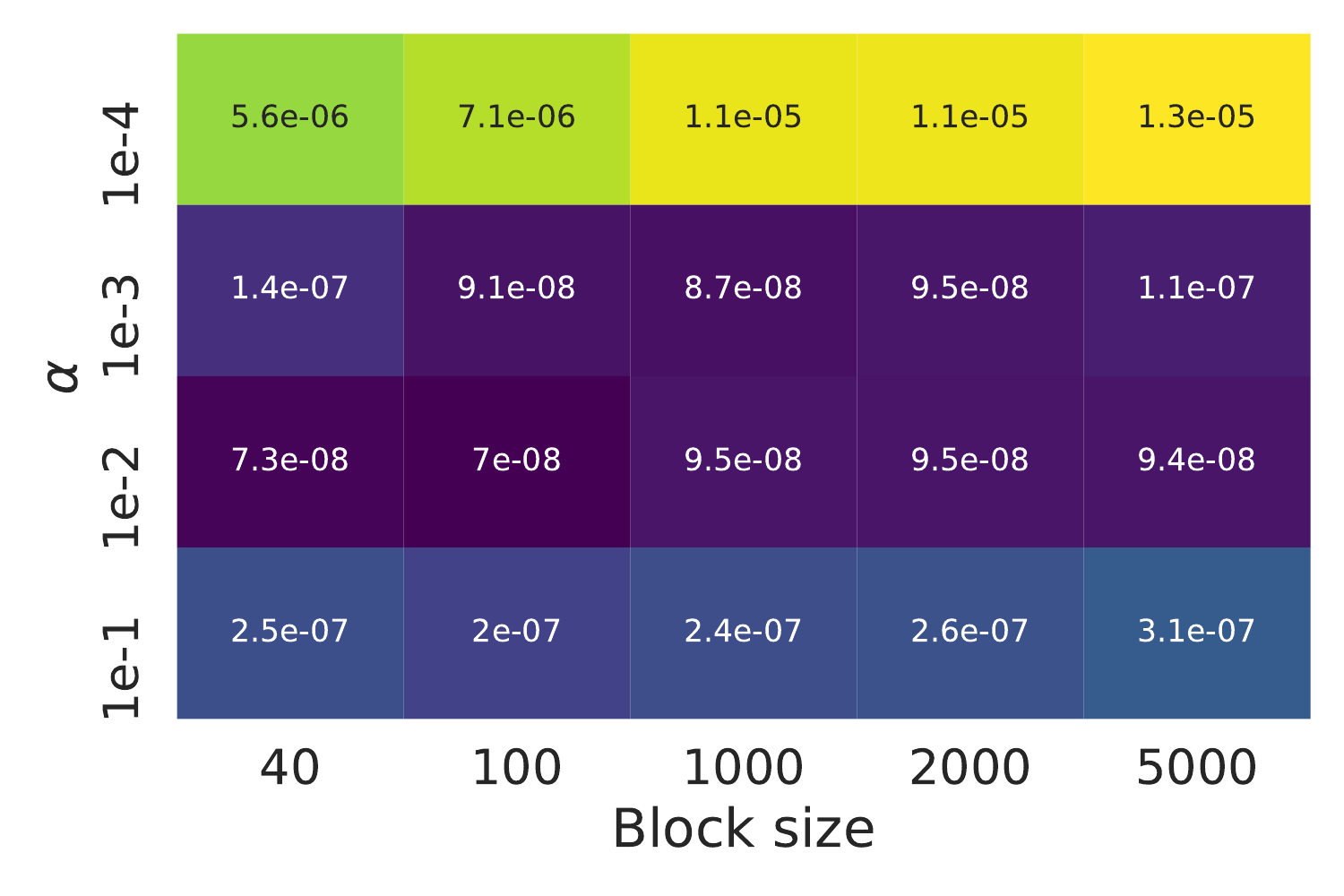}}
	\subfigure[\texttt{L-SVRG}]
	{\includegraphics[scale=0.37]{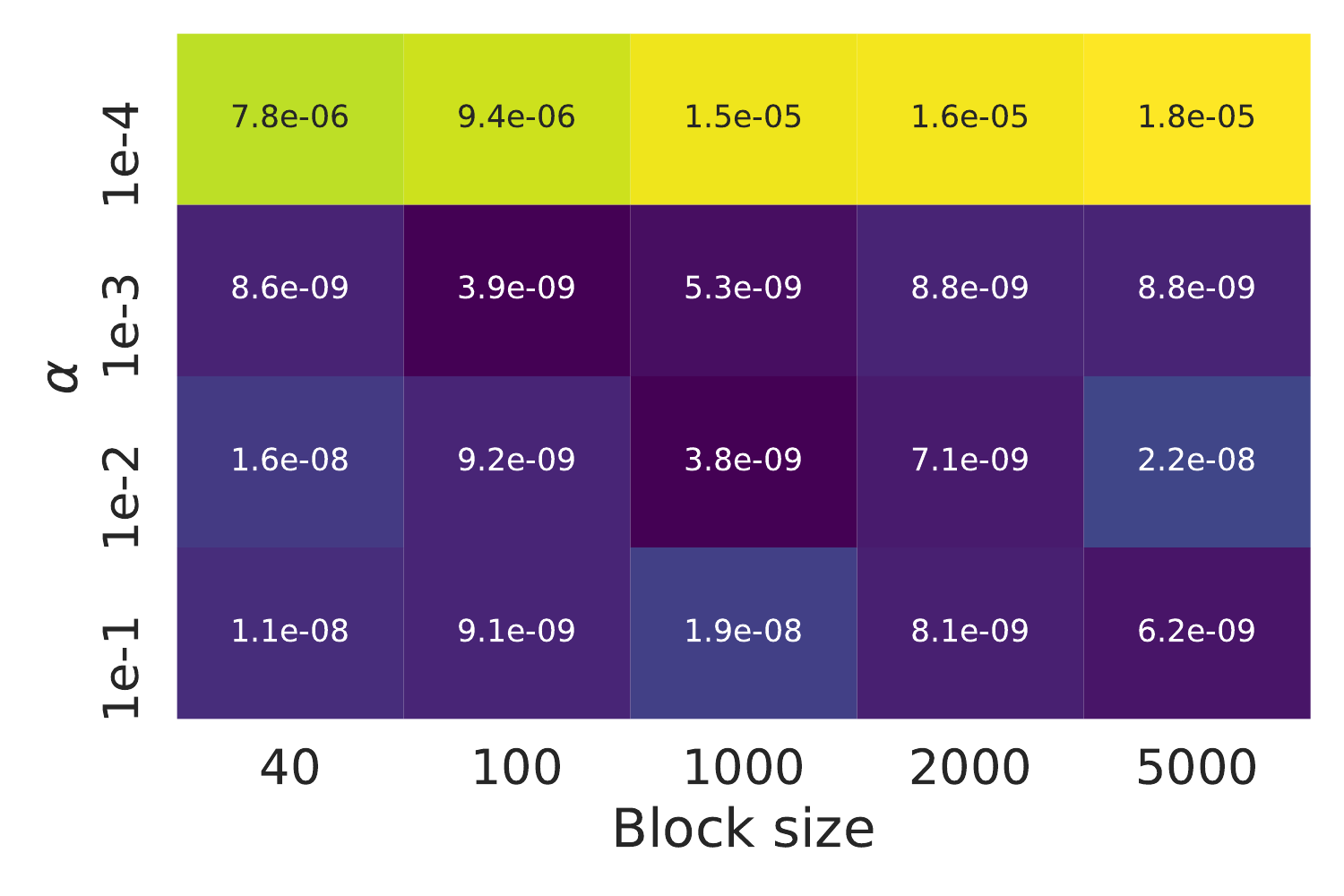}}
	\caption{\label{fig:blocks_and_alphas}Functional gap after 20 epochs over the Gisette dataset ($d=5000$) with Diana-VR and different combinations of block sizes and $\alpha$ with $n=20$ workers. The same (optimal) stepsize is used in all cases with $\ell_{\infty}$ random  dithering. For each cell, we average the results over 3 runs with the same parameters as the final accuracy is random. The best iteration performance is achieved using small blocks as expected, and the value of $\alpha$ does not matter much unless chosen too far from $\frac{1}{\omega + 1}$. The values of $\frac{1}{\omega + 1}$ for the columns from the left to the right approximately are $\{0.14,  0.09, 0.03, 0.02, 0.01\}$.}
\end{center}
\end{figure}
\newpage
\section{Notation Table}
To enhance the readers convenience when navigating to the extensive appendix, we here reiterate our notation:

\begin{table}[H]
\begin{center}

{
\footnotesize
\begin{tabular}{|c|l|c|}
\hline
\multicolumn{3}{|c|}{{\bf General} }\\
\hline
$\E{.}, \EE{Q}{.}$ & Expectation, Expectation over Quantization & \\
$\mu$ & Strong convexity constant & \eqref{def:strongconvex} \\
$L$ & Smoothness constant & \eqref{def:smoothness} \\
$\kappa$ & condition number of problem $l/\mu$ & \\
$d$ & Dimension of $x$ in $f(x)$ & \\
$n$ & Number of function in finite sum & \\
$f$ & Objective to be minimized over set $\R^d$ & \eqref{eq:probR}, \eqref{eq:probRR} \\
$Q(x)$ & Quantization operator & \eqref{def:omega} \\
$\omega$ & Quantization parameter & \eqref{def:omega}\\
$\prox$ & Proximal operator & \\
$x^\star$ & Global minimizer of $f$ & \\
$f^\star$ & function value in optimum $x^\star$ & \\
$R$ & Regularizer & \eqref{eq:probRR} \\

\hline
 \multicolumn{3}{|c|}{{\bf General Diana}}\\
 \hline
$ \sigma_i$'s & constants, upper bound on variance & \eqref{def:sigmai} \\
$\alpha, \gamma$ & parameters/ step sizes & Alg.~\ref{alg:improvedDIANA} \\
$\Psi^k$ & Lyapunov function & Thm.~\ref{thm:improvedDIANA} \\
$c$ & parameter of Lyapunov function & Thm.~\ref{thm:improvedDIANA}  \\
 \hline

 \multicolumn{3}{|c|}{{\bf VR-Diana}}\\
 \hline
$\alpha, \gamma$ & Parameters/ step sizes & Alg.~\ref{alg:VR-DIANA} \\
$m$ & number of functions on each node &  \eqref{eq:finsum} \\
$\psi^k$ & Lyapunov function for strongly convex case and convex case  & Thm.~\ref{thm:VR-DIANA} \\
$H^k, D^k$ & Elements of  Lyapunov function $\psi^k$& Thm.~\ref{thm:VR-DIANA}\\ 
$b,c $ & Parameters of Lyapunov function $\psi^k$ & Thm.~\ref{thm:VR-DIANA} \\
 $R^k$ & Lyapunov function for non-convex case  & Thm.~\ref{thm:VR-DIANA-non-convex} \\
$F^k, W^k$ & Elements of  Lyapunov function $R^k$& Thm.~\ref{thm:VR-DIANA-non-convex}\\ 
$c^k,d^k $ & Parameters of Lyapunov function $\psi^k$ & Thm.~\ref{thm:VR-DIANA-non-convex} \\

 \hline
 
 \multicolumn{3}{|c|}{{\bf SVRG-Diana}}\\
 \hline
$\alpha, \gamma$ & Parameters/ step sizes & Alg.~\ref{alg:SVRG_DIANA} \\
$m$ & number of functions on each node &  \eqref{eq:finsum} \\
$l$ & Outer loop length for Algorithm~\ref{alg:SVRG_DIANA} & Alg.~\ref{alg:SVRG_DIANA}\\
$\{p_r\}_{r=0}^{l-1}$ & coefficients for the reference point $z_s$  & Alg.~\ref{alg:SVRG_DIANA}\\
$\psi^k$ & Lyapunov function for strongly convex case and convex case  & Thm.~\ref{thm:SVRG-DIANA} \\
$H^k$ & Element of  Lyapunov function $\psi^k$& \eqref{def:SVRG_psi^k}\\ 
$b $ & Parameter of Lyapunov function $\psi^k$ & \eqref{def:SVRG_H^k} \\
$R^k$ & Lyapunov function for non-convex case  & Thm.~\ref{thm:SVRG-DIANA-non-convex} \\
$F^k, W^k$ & Elements of  Lyapunov function $R^k$& Thm.~\ref{thm:SVRG-DIANA-non-convex}\\ 
 \hline

\end{tabular}
}

\end{center}
\caption{Summary of frequently used notation.}
\label{tbl:notation}
\end{table}

\newpage

\section{Basic Identities and Inequalities}
For random variable $X$ and any $y \in \R^d$, the variance can be decomposed as
\begin{align}
\E{\norm{X-\E{X}}_2^2} = \E{\norm{X-y}_2^2} - \E{\norm{\E{X}-y}_2^2}\,.
\label{eq:variance}
\end{align}

For any vectors $a_1, a_2, \dots, a_k \in \R^d$, we have as a consequence of Jensen's inequality:
\begin{align}
\left\|\sum_{i=1}^k a_i\right\|_2^2 \leq k \sum_{i=1}^k \norm{a_i}_2^2  \,. \label{eq:sum}
\end{align}

For any independent random variables $X_1,X_2, \cdots, X_n \in \R^d$ we have
\begin{align}
\E{\norm{\frac{1}{n}\sum_{i = 1}^n \left(X_i -\E{X_i}\right)}_2^2} = \frac{1}{n^2}\sum_{i = 1}^n\E{\norm{ X_i -\E{X_i}}_2^2}
\label{eq:indep}
\end{align}

For a $L$-smooth and $\mu$-strongly convex function $f \colon \R^d \to \R$ we have
\begin{align}
 \lin{\nabla f(x)-\nabla f(y),x-y} \geq \frac{\mu L}{\mu + L} \norm{x-y}_2^2 + \frac{1}{\mu+L} \norm{\nabla f(x) - \nabla f(y)}_2^2\,, & &\forall x,y \in \R^d\,. \label{eq:coerc}
\end{align}

For $\mu$-strongly convex function $f \colon \R^d \to \R$ we have
\begin{align}
\lin{\nabla f(x)-\nabla f(y),x-y} \geq \mu \norm{x-y}_2^2\,, & &\forall x,y \in \R^d\,. \label{eq:coerc2}
\end{align}

The prox operator of a closed convex function is non-expansive. That is, for $\gamma > 0$,
\begin{align}
 \norm{\prox_{\gamma R}(x) - \prox_{\gamma R}(y)} &\leq \norm{x-y}\,, & &\forall x,y \in \R^d\,.
\end{align}

Throughout the whole appendix we use conditional expectation $\E{\cX| x^k, h_i^k}$ for DIANA and $\E{\cX|  x^k, h_i^k, w_{ij}^k}$ for VR-DIANA and $\E{\cX|  x^k, h_i^k, z^s}$ for SVRG-DIANA, but for simplicity, we will denote these expectations as $\E{\cX}$. If $\E{\cX}$ refers to unconditional expectation, it is directly mentioned.

\section{DIANA with general quantization operators, Proof of~\Cref{thm:improvedDIANA}}

\begin{lemma}
\label{lemma:boundg}
For all iterations $k \geq 0$ of Algorithm~\ref{alg:improvedDIANA} it holds
\begin{align}
 \EE{Q}{\hat{g}^k} &= g^k := \frac{1}{n}\sum_{i=1}^{n}g_i^k\,, &
 \EE{Q}{\norm{\hat{g}^k-g^k}_2^2} &\leq \frac{\omega}{n^2} \sum_{i=1}^{n}\norm{\Delta_i^k}_2^2\,, & 
 \E{g^k} &= \nabla f(x^k)\,.
 \label{eq:explain0}
\end{align}
Furthermore, for $h^\star = \nabla f(x^\star)$, $h_i^\star := \nabla f_i(x^\star)$
\begin{align}
 \E{\norm{\hat{g}^k-h^\star}_2^2} &\leq \left(1 + \frac{2\omega}{n} \right)  \frac{1}{n} \sum_{i=1}^{n} \E{\norm{\nabla f_i(x^k) - h_i^\star}_2^2} + \left(1 + \omega\right)\frac{\sigma^2}{n}  + \frac{2\omega}{n^2} \sum_{i=1}^n \E{\norm{h_i^\star - h_i^k}_2^2}\,. \label{eq:explain2}
\end{align}
\end{lemma}
\begin{proof}
The first equation in~\eqref{eq:explain0} follows from the unbiasedness of the quantization operator.
By the contraction property~\eqref{def:omega} we have
\begin{align*}
 \EE{Q}{\norm{\hat{g}_i^k - g_i^k}_2^2} \leq \omega \norm{\Delta_i^k}_2^2\,,
\end{align*}
for every $i=1,\dots, n$ and the second relation in~\eqref{eq:explain0} follows from independence of $\hat{g}_1^k,\dots,\hat{g}_n^k$.
The last equality in~\eqref{eq:explain0} follows from the assumption that each $g_i^k$ is an unbiased estimate of $\nabla f_i(x^k)$.

By applying two times the identity 
$\E{\norm{X-y}_2^2} = \E{\norm{X-\E{X}}_2^2} + \E{\norm{\E{X}-y}_2^2}$ for random variable $X$ and $y \in \R^d$, 
we get
\begin{align*}
\E{\norm{\hat{g}^k-h^\star}_2^2} 
&\stackrel{\eqref{eq:variance}}{=} \E{\norm{\hat{g}^k - g^k}_2^2} + \E{\norm{g^k - h^\star}_2^2} \\
&\stackrel{\eqref{eq:variance}}{=} \E{\norm{\hat{g}^k - g^k}_2^2} + \E{\norm{g^k - \nabla f(x^k)}_2^2} + \E{\norm{\nabla f(x^k) - h^\star}_2^2}
\end{align*}
and thus
\begin{align}
 \E{\norm{\hat{g}^k-h^\star}_2^2} &\stackrel{\eqref{def:sigmai}+\eqref{eq:explain0}}{\leq} \frac{\omega}{n^2} \sum_{i=1}^{n} \E{ \norm{\Delta_i^k}_2^2} + \frac{\sigma^2}{n} +  \E{\norm{\nabla f(x^k) - h^\star}_2^2}\,. \label{eq:5435}
\end{align}
Note that
\begin{align*}
\norm{\nabla f(x^k) - h^\star}_2^2 \leq \frac{1}{n} \sum_{i=1}^n \norm{\nabla f_i(x^k) - h_i^\star}_2^2\,,
\end{align*}
by Jensen's inequality. Further,
\begin{align*}
\E{ \norm{\Delta_i^k}_2^2 } &= \E{ \norm{g_i^k - h_i^k }_2^2} \stackrel{\eqref{eq:variance}}{=} \E{ \norm{\nabla f_i(x^k) - h_i^k}_2^2 } + \E{\norm{\nabla f_i(x^k) - g_i^k}_2^2} \\
 &\stackrel{\eqref{def:sigmai}}{\leq} \E{\norm{\nabla f_i(x^k) - h_i^k}_2^2} + \sigma_i^2 \\
 &\stackrel{\eqref{eq:sum}}{\leq} 2 \E{\norm{\nabla f_i(x^k) - h_i^\star}_2^2} +2 \E{\norm{h_i^\star - h_i^k}_2^2} + \sigma_i^2.
\end{align*} 
By summing up these bounds and plugging the result into~\eqref{eq:5435}, equation~\eqref{eq:explain2} follows.
\end{proof}

\begin{lemma}
Let $\alpha (\omega + 1) \leq 1$. For $i=1,\dots,n$, we can upper bound the second moment of $h_i^{k+1}$ as
\begin{align}
 \EE{Q}{\norm{h_{i}^{k+1} - h_i^\star}_2^2} \leq (1-\alpha) \norm{h_{i}^{k} - h_i^\star}_2^2 +\alpha\norm{\nabla f_i(x^k)- h_i^\star}_2^2 + \alpha \sigma^2_i \,.
 \label{eq:hdecrement}
\end{align}
\end{lemma}
\begin{proof}
Since $h_i^{k+1} = h_i^k + \alpha \hat{\Delta}_i^k$ we can decompose
\begin{align*}
 \EE{Q}{\norm{h_{i}^{k+1} - h_i^\star}_2^2} &= 
 \EE{Q}{\norm{\alpha \hat{\Delta}_i^k + (h_i^k - h_i^\star)}_2^2} = \norm{h_{i}^{k} - h_i^\star}_2^2 + 2 \EE{Q}{\lin{\alpha \hat{\Delta}_i^k,h_i^k - h_i^\star}} + \EE{Q}{\norm{\alpha \hat{\Delta}_i^k}_2^2} \\
&\stackrel{\eqref{def:Q}}{\leq} \norm{h_{i}^{k} - h_i^\star}_2^2  +  2\lin{\alpha \Delta_i^k,h_i^k - h_i^\star} + \alpha^2 (\omega+1) \norm{\Delta_i^k}_2^2.
\end{align*}
Let plug in the bound $(\omega+1) \alpha \leq 1$ and continue the derivation:
\begin{align*}
 \EE{Q}{\norm{h_{i}^{k+1} - h_i^\star}_2^2} &\leq \norm{h_{i}^{k} - h_i^\star}_2^2  +  2\lin{\alpha \Delta_i^k,h_i^k - h_i^\star} + \alpha \norm{\Delta_i^k}_2^2 \\
 &=\norm{h_{i}^{k} - h_i^\star}_2^2 + \alpha \lin{g_i^k - h_i^k,g_i^k + h_i^k - 2 h_i^\star} \\
 &=\norm{h_{i}^{k} - h_i^\star}_2^2 + \alpha\norm{g_i^k- h_i^\star}_2^2 - \alpha\norm{h_i^k - h_i^\star}_2^2 \\
&\leq (1-\alpha) \norm{h_{i}^{k} - h_i^\star}_2^2 +  \alpha\norm{g_i^k- h_i^\star}_2^2\qedhere.
\end{align*}
The second term can be further upper-bounded by $\norm{\nabla f_i(x^k)- h_i^\star}_2^2 + \sigma^2_i$, where we use \eqref{eq:variance}, which concludes the proof.
\end{proof}

\begin{proof}[Proof of Theorem~\ref{thm:improvedDIANA}]
If $x^\star$ is a solution of~\eqref{eq:probR}, then $x^\star = \prox_{\gamma R}(x^\star - \gamma h^\star)$ (for $\gamma > 0$. Using this identity together with the non-expansiveness of the prox operator we can bound the first term of the Lyapunov function:
\begin{align*}
 \E{\norm{x^{k+1}- x^\star}_2^2} 
 &= \E{\norm{\prox_{\gamma R}(x^k - \gamma \hat{g}^k) - \prox_{\gamma R}(x^\star - \gamma h^\star)}_2^2} \\
 &\leq \E{\norm{x^k - \gamma \hat{g}^k - (x^\star - \gamma h^\star))}_2^2} \\
 &=\E{\norm{x^k - x^\star}_2^2} - 2 \gamma \E{\lin{\hat{g}^k - h^\star, x^k - x^\star}} + \gamma^2 \E{\norm{\hat{g}^k - h^\star}_2^2} \\
 &= \E{\norm{x^k - x^\star}_2^2} - 2 \gamma \lin{\nabla f(x^k) - h^\star, x^k - x^\star} + \gamma^2 \E{\norm{\hat{g}^k - h^\star}_2^2}. 
\end{align*}
It is high time to use strong convexity of each component $f_i$:
\begin{align*}
 \E{ \lin{\nabla f(x^k) - h^\star, x^k - x^\star}} &=
 \frac{1}{n} \sum_{i=1}^n  \E{ \lin{\nabla f_i(x^k) - h_i^\star, x^k - x^\star}} \\
 &\stackrel{\eqref{eq:coerc}}{\geq} \frac{1}{n}\sum_{i=1}^n \left( \frac{\mu L}{\mu +L} \E{\norm{x^k - x^\star}_2^2} + \frac{1}{\mu + L}  \E{\norm{\nabla f_i(x^k) - h_i^\star}_2^2} \right) \\
 &= \frac{\mu L}{\mu +L} \E{\norm{x^k - x^\star}_2^2} + \frac{1}{\mu + L} \frac{1}{n} \sum_{i=1}^n \E{\norm{\nabla f_i(x^k) - h_i^\star}_2^2} \,.
\end{align*}
Hence,
\begin{align*}
\E{\norm{x^{k+1}- x^\star}_2^2}  \leq \left(1- \frac{2 \gamma \mu L}{\mu +L}  \right)  \E{\norm{x^k - x^\star}_2^2} - \frac{2\gamma}{\mu + L} \frac{1}{n} \sum_{i=1}^n \E{\norm{\nabla f_i(x^k) - h_i^\star}_2^2}  + \gamma^2 \E{\norm{\hat{g}^k - h^\star}_2^2} 
\end{align*}
and by Lemma~\ref{lemma:boundg}:
\begin{align}
\begin{split}
\E{\norm{x^{k+1}- x^\star}_2^2}  &\stackrel{\eqref{eq:explain2}}{\leq} \left(1- \frac{2 \gamma \mu L}{\mu +L}  \right)  \E{\norm{x^k - x^\star}_2^2} + \left(\gamma^2 \left(1+ \frac{2 \omega}{n}\right) -  \frac{2\gamma}{\mu + L} \right) \frac{1}{n} \sum_{i=1}^n \E{\norm{\nabla f_i(x^k) - h_i^\star}_2^2}  \\ &\qquad +\gamma^2 \left(1+\omega\right) \frac{\sigma^2}{n} + \left( \gamma^2 \frac{2\omega}{n}\right) \frac{1}{n} \sum_{i=1}^n \E{\norm{h_i^k - h_i^\star}_2^2}\,.
\end{split}
\label{eq:6943}
\end{align}
Now let us consider the Lyapunov function:
\begin{align}
\begin{split}
 \E{\Psi^{k+1}} & \stackrel{\eqref{eq:hdecrement}+\eqref{eq:6943}}{\leq} \left(1- \frac{2 \gamma \mu L}{\mu +L}  \right)  \E{\norm{x^k - x^\star}_2^2} + \gamma^2\left(1+ \omega\right) \frac{\sigma^2}{n} \\
 &\qquad +   \left(\gamma^2 \left(1+ \frac{2\omega}{n} +c\alpha\right) -  \frac{2\gamma}{\mu + L} \right) \frac{1}{n} \sum_{i=1}^n \E{\norm{\nabla f_i(x^k) - h_i^\star}_2^2} \\
 &\qquad + \gamma^2  \left(\frac{2\omega}{n} + (1-\alpha) c \right) \frac{1}{n} \sum_{i=1}^n \norm{h_i^k - h_i^\star}_2^2 + \gamma^2c\alpha \sigma^2 \,.
 \end{split}
 \label{eq:2945}
\end{align}
In view of the assumption on $\gamma$ we have 
$\gamma^2 \left(1+ \frac{2\omega}{n} +c\alpha\right) -  \frac{2\gamma}{\mu + L} \leq 0$.
Since each $f_i$ is $\mu$-strongly convex, we have
$ \mu \norm{x^k - x^\star}_2^2 \stackrel{\eqref{eq:coerc2}}{\leq} \lin{\nabla f_i(x^k)-h_i^\star,x^k-x^\star}$
and thus $\mu^2 \norm{x^k - x^\star}_2^2 \leq \norm{\nabla f_i(x^k)-h_i^\star}_2^2$ with Cauchy-Schwarz. 
Using these observations we can absorb the third term in~\eqref{eq:2945} in the first one:
\begin{align*}
\begin{split}
 \E{\Psi^{k+1}} & \leq \left(1- 2\gamma \mu + \mu^2 \gamma^2 \left(1+ \frac{2\omega}{n} +c\alpha\right) \right)  \E{\norm{x^k - x^\star}_2^2} + \gamma^2 \left(c\alpha+ \frac{\omega+1}{n}\right) \sigma^2 \\
 &\qquad +  \gamma^2\left(\frac{2\omega}{n} + (1-\alpha)c\right) \frac{1}{n} \sum_{i=1}^n \norm{h_i^k - h_i^\star}_2^2\,.
 \end{split}
 \end{align*}
 By the first assumption on $\gamma$ it follows $\left(1- 2\gamma \mu + \mu^2 \gamma^2 \left(1+ \frac{2\omega}{n} +c\alpha\right) \right) \leq( 1-\gamma \mu)$. 
By the assumption on $c$ we have $\left(\frac{2\omega}{n} + (1-\alpha)c\right) \leq \left(1- \frac{\alpha}{2}\right)c$. 
 An the second assumption on $\gamma$ implies $\left(1-\frac{\alpha}{2}\right)\leq(1-\gamma\mu)$. Thus
\begin{align*}
\E{\Psi^{k+1}} \leq (1-\gamma\mu) \Psi^k + \gamma^2 \left(1+ \frac{\omega}{n}\right)\frac{\sigma^2}{n}\,.
\end{align*}
Unrolling the recurrence and the estimate $\sum_{\ell=0}^{k-1} (1-\gamma \mu)^k \leq \frac{1}{\mu \gamma}$ for all $k \geq 1$ leads to
\begin{align*}
 \E{\Psi^k} \leq (1- \gamma \mu)^k \Psi^0 + \frac{\gamma}{\mu}  \left(c\alpha+ \frac{\omega+1}{n}\right) \sigma^2 \leq  (1- \gamma \mu)^k \Psi^0 + \frac{2}{\mu(\mu+L)} \sigma^2\,,
\end{align*}
by the first assumption on $\gamma$.
\end{proof}

\section{Variance Reduced Diana---L-SVRG method and SAGA proof}
\label{sec:vrdiana}

\begin{lemma}
For all iterates $k \leq 0$ of Algorithm~\ref{alg:VR-DIANA}, it holds that $g_i^k$ is an unbiased estimate of the local gradient $\nabla f_i(x^k)$
	\begin{equation*}
		\E{g_i^k} = \nabla f_i(x^k)
	\end{equation*}
	and $g^k$ is that of the full gradient $\nabla f(x^k)$:
	\begin{equation*}
	\E{g^k} = \nabla f(x^k).
	\end{equation*}
\end{lemma}
\begin{proof}
	It is a straightforward consequence of how we define sampling:
	\begin{align*}
		\E{g_i^k}
		=
		\E{\nabla f_{ij_i^k}(x^k) - \nabla f_{ij_i^k}(w_{ik_i^k}^k) + \mu_i^k}
		=
		\nabla f_i(x^k) - \mu_i^k + \mu_i^k
		=
		\nabla f_i(x^k).
	\end{align*}
	Similarly,
	\begin{align*}
		\E{g^k}
		=
		\frac{1}{n}
		\sum\limits_{i=1}^{n}
			\E{Q(g_i^k - h_i^k) + h_i^k}
		=
		\frac{1}{n}
		\sum\limits_{i=1}^{n}
			\E{g_i^k - h_i^k + h_i^k}
		=
		\frac{1}{n}
		\sum\limits_{i=1}^{n}
			\nabla f_i(x^k)
		=
		\nabla f(x^k).
	\end{align*}
\end{proof}

\subsection{Strongly convex case}
\begin{lemma}
We can upper bound the second moment of $x^k $ in the following way
	\begin{equation}
	\E{\norm{x^{k+1} - x^\star}_2^2} \leq \norm{x^k - x^\star}_2^2(1 - \mu \gamma) + 2\gamma(f^\star - f(x^k)) + \gamma^2 \E{\norm{g^k}_2^2}.
\label{lem:up_x_k_VR}	
	\end{equation}
\end{lemma}
\begin{proof}
	\begin{align*}
	\E{\norm{x^{k+1} - x^\star}_2^2} &=
	\E{
		\norm{x^{k} - x^\star}_2^2 + 2\gamma\dotprod{g^k}{x^\star - x^k} + \gamma^2 \norm{g^k}_2^2
	}\\
	&=
	\norm{x^{k} - x^\star}_2^2 + 2\gamma\dotprod{\nabla f(x^k)}{x^\star - x^k} + \gamma^2 \E{\norm{g^k}_2^2}\\
	&\overset{\eqref{def:strongconvex}}{\leq}
	\norm{x^{k} - x^\star}_2^2
	+
	2\gamma
	\left(
	f^\star - f(x^k) - \frac{\mu}{2}\norm{x^k - x^\star}_2^2
	\right)
	+
	\gamma^2 \E{\norm{g^k}_2^2}\\
	&=
	\norm{x^k - x^\star}_2^2(1 - \mu \gamma) + 2\gamma(f^\star - f(x^k)) + \gamma^2 \E{\norm{g^k}_2^2},
	\end{align*}
	where the first equation follows from the definition of $x^{k+1}$ in Algorithm~\ref{alg:VR-DIANA}.
\end{proof}

\begin{lemma} Let $\alpha(\omega+1) \leq 1$. We can upper bound $H^{k+1}$ in the following way 
	\begin{equation}
		\E{H^{k+1}}
		\leq
		(1-\alpha)H^k + \frac{2\alpha}{m} D^k + 8\alpha Ln \left( f(x^k) - f^\star \right),
		\label{lem:up_H_VR}
	\end{equation}
	where 
	\begin{equation}
		H^{k}
		\eqdef
		\sum\limits_{i=1}^{n}
			\norm{h_i^{k} - \nabla f_i(x^\star)}_2^2
		\label{lem:def_H_VR}
	\end{equation}
	and
	\begin{equation}
		D^{k}
		\eqdef
		\sum\limits_{i=1}^{n}
			\sum\limits_{j=1}^{m}
				\norm{\nabla f_{ij}(w_{ij}^{k}) - \nabla f_{ij}(x^\star)}_2^2.
		\label{lem:def_D_VR}
	\end{equation}
\end{lemma}

\begin{proof}
	\begin{align*}
		\E{H^{k+1}}
		&=
		\E{\sum\limits_{i=1}^{n}
			\norm{h_i^{k+1} - \nabla f_i(x^\star)}_2^2} \\
		&=
		\sum\limits_{i=1}^{n}
			\norm{h_i^k - \nabla f_i(x^\star)}_2^2\\
		&\qquad+
		\sum\limits_{i=1}^{n}
			\E{
				2\dotprod{\alpha Q(g_i^k - h_i^k)}{h_i^k - \nabla f_i(x^\star)}
				+
				\alpha^2\norm{Q(g_i^k - h_i^k)}_2^2
			}\\
		&\leq
		H^k
		+
		\sum\limits_{i=1}^{n}\E{
			2\alpha \dotprod{g_i^k - h_i^k}{h_i^k - \nabla f_i(x^\star)}
			+
			\alpha \left(\alpha \cdot (\omega+1)\right) \norm{g_i^k - h_i^k}_2^2}\\
		&\leq
		H^k
		+
		\E{\sum\limits_{i=1}^{n}
			\alpha \dotprod{g_i^k - h_i^k}{g_i^k + h_i^k - 2\nabla f_i(x^\star)}}\\
		&=
		H^k
		+
		\E{\sum\limits_{i=1}^{n}
			\alpha 
			\left(
				\norm{g_i^k - \nabla f_i(x^\star)}_2^2
				-
				\norm{h_i^k - \nabla f_i(x^\star)}_2^2
			\right)}\\
		&=
		H^k(1-\alpha)
		+
		\E{\sum\limits_{i=1}^{n}
			\alpha 
			\left(
				\norm{g_i^k - \nabla f_i(x^\star)}_2^2
			\right)}\\
		&\overset{\eqref{eq:sum}}{\leq}
		H^k(1-\alpha)
		+
		\sum\limits_{i=1}^{n}
		\left(
			2\alpha \E{		
				\norm{g_i^k - \nabla f_i(x^k)}_2^2
			} 
			+ 2\alpha\norm{\nabla f_i(x^k) - \nabla f_i(x^\star)}_2^2
			\right)\\
		&\overset{\text{Alg.}~\ref{alg:VR-DIANA}}{=}
		H^k(1-\alpha)
		+
		\sum\limits_{i=1}^{n}\E{
			2\alpha 
				\norm{\nabla f_{ij_i^k}(x^k) - \nabla f_{ij_i^k}(w_{ij_i^k}^k) - \E{\nabla f_{ij_i^k}			(x^k) - \nabla f_{ij_i^k}(w_{ij_i^k}^k)}}_2^2
			} \\
			&\quad +
			2\alpha
			\sum\limits_{i=1}^{n}
			\norm{\nabla f_i(x^k) - \nabla f_i(x^\star)}_2^2
			\\
		&\overset{\eqref{eq:variance}}{\leq}
		H^k(1-\alpha)
		+
		\sum\limits_{i=1}^{n}\left(
		\E{
			2\alpha 
				\norm{\nabla f_{ij_i^k}(x^k) - \nabla f_{ij_i^k}(w_{ij_i^k}^k) }_2^2
			}
			+
			2\alpha\norm{\nabla f_i(x^k) - \nabla f_i(x^\star)}_2^2
			\right)\\
		&\overset{\eqref{eq:sum}}{\leq}
		H^k(1-\alpha)
		+
		\frac{2\alpha}{m}
		\sum\limits_{i=1}^{n}
			\sum\limits_{j=1}^{m}
			\left( \norm{\nabla f_{ij}(x^k) - \nabla f_{ij}(x^\star)}_2^2 +
				\norm{\nabla f_{ij}(w_{ij}^k) - \nabla f_{ij}(x^\star)}_2^2\right)\\
		&\quad +
			2\alpha
			\sum\limits_{i=1}^{n}
			\norm{\nabla f_i(x^k) - \nabla f_i(x^\star)}_2^2
			\\
		&\overset{\eqref{def:smoothness}}{\leq}
		H^k(1-\alpha)
		+
		\frac{2\alpha}{m}
		\sum\limits_{i=1}^{n}
			\sum\limits_{j=1}^{m}
				\norm{\nabla f_{ij}(w_{ij}^k) - \nabla f_{ij}(x^\star)}_2^2
		+
		8\alpha Ln \left( f(x^k) - f^\star \right) \\
		&=
		H^k(1-\alpha) + \frac{2\alpha}{m} D^k + 8\alpha Ln \left( f(x^k) - f^\star \right) ,
	\end{align*}
	where the second equality uses definition of $h_i^{k+1}$ in Algorithm~\ref{alg:VR-DIANA}  and the first inequality follows from $\alpha(\omega+1) \leq 1$.
\end{proof}

\begin{lemma} We can upper bound $D^{k+1}$ in the following way
	\begin{equation}
		\E{D^{k+1}}
		\leq
		D^k \left(1 - \frac{1}{m}\right)
		+
		2Ln(f(x^k) - f^\star).
		\label{lem:up_D_VR}
	\end{equation}
\end{lemma}
\begin{proof}
	\begin{align*}
		\E{D^{k+1}}
		&=
		\sum\limits_{i=1}^{n}
			\sum\limits_{j=1}^{m}
				\E{\norm{\nabla f_{ij}(w_{ij}^{k+1}) - \nabla f_{ij}(x^\star)}_2^2}\\
		&=
		\sum\limits_{i=1}^{n}
			\sum\limits_{j=1}^{m}
				\left[
					\left(
						1 - \frac{1}{m}
					\right)
					\norm{\nabla f_{ij}(w_{ij}^k) - \nabla f_{ij}(x^\star)}_2^2
					+
					\frac{1}{m}
					\norm{\nabla f_{ij}(x^k) - \nabla f_{ij}(x^\star)}_2^2			
				\right]\\
		&\overset{\eqref{def:smoothness}}{\leq}
		D^k
		\left(
			1 - \frac{1}{m}
		\right)
		+
		2Ln(f(x^k) - f^\star),
	\end{align*}
	where the second equality uses definition of $w_{ij}^{k+1}$ in Algorithm~\ref{alg:VR-DIANA}.
\end{proof}

\begin{lemma}
We can upper bound the second moment of the $g^k$ in the following way
	\begin{equation}
		\E{\norm{g^k}_2^2}
		\leq
		2L(f(x^k) - f^\star)\left(1 + \frac{4\omega+2}{n}\right)
		+
		\frac{2\omega}{mn^2}D^k
		+
		\frac{2(\omega+1)}{n^2} H^k.
		\label{lem:up_g_VR}
	\end{equation}
\end{lemma}
\begin{proof}
	\begin{align*}
		\EE{Q}{\norm{g^k}_2^2} &\overset{\eqref{eq:variance}}{=}
		\underbrace{
			\norm{\EE{Q}{g^k}}_2^2
		}_{T_1}
		+
		\underbrace{
			\EE{Q}{\norm{g^k - \EE{Q}{g^k}}_2^2}
		}_{T_2}.
	\end{align*}
	We proceed with upper bounding terms $T_1$ and $T_2$ separately. For $T_1$ we can use definition of $g^k$ in order to obtain 
	\begin{align*}
		T_1
		&= 
		\norm{
			\frac{1}{n}
			\sum\limits_{i=1}^{n}
				\EE{Q}{Q(g_i^k - h_i^k) + h_i^k}
		}_2^2
		=
		\norm{
			\frac{1}{n}
			\sum\limits_{i=1}^{n}
				g_i^k
		}_2^2
	\end{align*}
	and 
	\begin{align*}
		T_2
		&=
		\EE{Q}{
			\norm{
				\frac{1}{n}
				\sum\limits_{i=1}^{n}
					Q(g_i^k - h_i^k) - (g_i^k - h_i^k)
			}_2^2
		}\\
		&\overset{\eqref{eq:indep}}{=}
		\frac{1}{n^2}
		\sum\limits_{i=1}^{n}
			\EE{Q}{
				\norm{Q(g_i^k - h_i^k) - (g_i^k - h_i^k)}_2^2
			} \\ 
		&\overset{\eqref{def:Q}}{\leq}
		\frac{\omega}{n^2}
		\sum\limits_{i=1}^{n}
			\norm{g_i^k - h_i^k}_2^2.
	\end{align*}
	Let us calculate full expectations conditioned on previous iteration:
	\begin{align*}
		\E{T_2}
		&=
		\frac{\omega}{n^2}
		\sum\limits_{i=1}^{n}
			\E{\norm{g_i^k - h_i^k}_2^2}
		=
		\frac{\omega}{n^2}
		\sum\limits_{i=1}^{n}
		    \left(
			\norm{\E{g_i^k - h_i^k}}_2^2
			+
			\E{\norm{g_i^k - h_i^k - \E{g_i^k - h_i^k}}_2^2}
			\right)\\
		&=
		\frac{\omega}{n^2}
		\sum\limits_{i=1}^{n}
		    \left(
			\norm{\nabla f_i(x^k) - h_i^k}_2^2
			+
			\E{\norm{g_i^k - \nabla f_i(x^k)}_2^2}
			\right)\\
		&=
		\frac{\omega}{n^2}
		\sum\limits_{i=1}^{n}
		    \left(
			\norm{\nabla f_i(x^k) - h_i^k}_2^2
			+
			\E{\norm{\nabla f_{ij_i^k}(x^k) - \nabla f_{ij_i^k}(w_{ij_i^k}^k) - \E{\nabla f_{ij_i^k}(x^k) - \nabla f_{ij_i^k}(w_{ij_i^k}^k)}}_2^2}
			\right)\\
		&\overset{\eqref{eq:variance}}{\leq}
		\frac{\omega}{n^2}
		\sum\limits_{i=1}^{n}
		     \left(
			 \norm{\nabla f_i(x^k) - h_i^k}_2^2
			 +
			 \E{\norm{\nabla f_{ij_i^k}(x^k) - \nabla f_{ij_i^k}(w_{ij_i^k}^k)}_2^2}
			 \right)\\
		&\overset{\eqref{eq:sum}}{\leq}
		\frac{2\omega}{n^2}
		\sum\limits_{i=1}^{n}
		    \left(
			\norm{h_i^k - \nabla f_i(x^\star)}_2^2
			+
			\norm{\nabla f_i(x^k) - \nabla f_i(x^\star)}_2^2
			\right)\\
		&\qquad+
		\frac{2\omega}{mn^2}
		\sum\limits_{i=1}^{n}
			\sum\limits_{j=1}^{m}
			    \left(
				\norm{\nabla f_{ij}(w_{ij}^k) - \nabla f_{ij}(x^\star)}_2^2
				+
				\norm{\nabla f_{ij}(x^k) - \nabla f_{ij}(x^\star)}_2^2
				\right)\\
		&\overset{\eqref{def:smoothness}}{\leq}
		\frac{2\omega}{n^2} H^k
		+
		\frac{2\omega}{mn^2} D^k
		+
		\frac{8L\omega}{n} (f(x^k) - f^\star)
	\end{align*}
	The other term follows in a similar way:
	\begin{align*}
		\E{T_1}
		&=
		\E{\norm{
			\frac{1}{n}
			\sum\limits_{i=1}^{n}
				g_i^k
		}_2^2}
		=
		\norm{
			\frac{1}{n}
			\sum\limits_{i=1}^{n}
			\E{g_i^k}
		}_2^2
		+
		\E{\norm{
				\frac{1}{n}
				\sum\limits_{i=1}^{n}
				\left( g_i^k - \E{g_i^k}\right)
			}_2^2}\\
		&=
		\norm{\nabla f(x^k)}_2^2
		+
		\frac{1}{n^2}
		\sum\limits_{i=1}^{n}
			\E{\norm{g_i^k - \nabla f_i(x^k)}_2^2}\\
		&\overset{\eqref{def:smoothness}}{\leq}
		2L(f(x^k) - f^\star)
		+
		\frac{1}{n^2}
		\sum\limits_{i=1}^{n}
			\E{\norm{\nabla f_{ij_i^k}(x^k) - \nabla f_{ij_i^k}(w_{ij_i^k}^k) - \E{\nabla f_{ij_i^k}(x^k) - \nabla f_{ij_i^k}(w_{ij_i^k}^k)}}_2^2}\\
		&\overset{\eqref{eq:variance}}{\leq}
		2L(f(x^k) - f^\star)
		+
		\frac{1}{n^2}
		\sum\limits_{i=1}^{n}
			\E{\norm{\nabla f_{ij_i^k}(x^k) - \nabla f_{ij_i^k}(w_{ij_i^k}^k)}_2^2}\\
		&\overset{\text{Alg.}~\ref{alg:VR-DIANA}}{=}
		2L(f(x^k) - f^\star)
		+
		\frac{1}{mn^2}
		\sum\limits_{i=1}^{n}
			\sum\limits_{j=1}^{m}
				\norm{\nabla f_{ij}(x^k) - \nabla f_{ij}(w_{ij}^k)}_2^2\\
		&\overset{\eqref{eq:sum}}{\leq}
		2L(f(x^k) - f^\star)
		+
		\frac{2}{mn^2}
		\sum\limits_{i=1}^{n}
			\sum\limits_{j=1}^{m}
			    \left(
				\norm{\nabla f_{ij}(w_{ij}^k) - \nabla f_{ij}(x^\star)}_2^2
				+
				\norm{\nabla f_{ij}(x^k) - \nabla f_{ij}(x^\star)}_2^2
				\right)\\
		&\overset{\eqref{def:smoothness}}{\leq}
		(f(x^k) - f^\star)
		\left(2L + \frac{4L}{n}\right)
		+
		\frac{2}{mn^2} D^k.
	\end{align*}
	Now, summing $\E{T_1}$ and $\E{T_2}$ we get
	\begin{align*}
		\E{\norm{g^k}_2^2}
		&=
		\E{T_1 + T_2}
		\leq
		(f(x^k) - f^\star)
		\left(2L + \frac{4L}{n}\right)
		+
		\frac{2}{mn^2} D^k\\
		&\quad +
		\frac{2\omega}{n^2} H^k
		+
		\frac{2\omega}{mn^2} D^k
		+
		\frac{8L\omega}{n} (f(x^k) - f^\star)\\
		&\leq
		(f(x^k) - f^\star)
		\left(
			2L + \frac{4L}{n} + \frac{8L\omega}{n}
		\right)
		+
		\frac{2\omega}{n^2} H^k
		+
		\frac{2(\omega+1)}{mn^2} D^k,
	\end{align*}
	which concludes the proof.
\end{proof}


\begin{proof}[Proof of \cref{thm:VR-DIANA}]
Combining all lemmas together we may finalize proof. By the definition of Lyapunov function we have
	\begin{align}
		\E{\psi^{k+1}}
		&=
		\E{\norm{x^{k+1} - x^\star}_2^2 + b \gamma^2 H^{k+1} + c \gamma^2 D^{k+1}} \notag \\
		&\overset{\eqref{lem:up_x_k_VR}}{\leq}
		\norm{x^k - x^\star}_2^2(1 - \mu \gamma) + 2\gamma(f^\star - f(x^k)) + \gamma^2 \E{\norm{g^k}_2^2}
		+
		\E{
			b \gamma^2 H^{k+1} + c \gamma^2 D^{k+1}
		}\notag \\
		&\overset{\eqref{lem:up_H_VR}+\eqref{lem:up_D_VR} + \eqref{lem:up_g_VR}}{\leq}
		\norm{x^k - x^\star}_2^2(1 - \mu \gamma) + 2\gamma(f^\star - f(x^k))\notag \\
		&\quad +
		\gamma^2
		\left(
			2L(f(x^k) - f^\star)\left(1 + \frac{4\omega+2}{n}\right)
			+
			\frac{2(\omega+1)}{mn^2}D^k
			+
			\frac{2\omega}{n^2} H^k
		\right)\notag \\
		&\quad +
		b \gamma^2
		\left(
			H^k(1 - \alpha) + \frac{2\alpha}{m} D^k + 8\alpha Ln \left( f(x^k) - f^\star \right)
		\right)
		+
		c \gamma^2
		\left(
			D^k \left(1 - \frac{1}{m}\right)
			+
			2Ln(f(x^k) - f^\star)
		\right)\notag \\
		&=
		\norm{x^k - x^\star}_2^2(1 - \mu \gamma)
		+
		b \gamma^2 H^k
		\left(
			1 - \alpha + \frac{2 \omega}{b n^2}
		\right)
		+
		c \gamma^2 D^k
		\left(
			1 - \frac{1}{m}
			+
			\frac{2b \alpha}{c m}
			+
			\frac{2(\omega+1)}{mn^2 c}
		\right)\notag \\
		& \quad +
		(f^\star - f(x^k))
		\left(
		 	2\gamma - 
		 	2L\gamma^2
		 	\left[
		 		1 + \frac{4\omega+2}{n} + c n + 4b\alpha n	 		
		 	\right]
		\right). \label{conv_VR}
	\end{align}
	Now, choosing
	$b = \frac{4 (\omega+1)}{\alpha n^2}$
	and
	$c = \frac{16 (\omega+1)}{ n^2}$
	, we get
	\begin{align*}
	\E{\psi^{k+1}}
	&\leq
	\norm{x^k - x^\star}_2^2(1 - \mu \gamma)
	+
	b \gamma^2 H^k
	\left(
		1 - \frac{\alpha}{2}
	\right)
	+
	c \gamma^2 D^k
	\left(
		1 - \frac{3}{8m}
	\right)\\
	&+
	(f^\star - f(x^k))
	\left(
		2\gamma - 
		2L\gamma^2
		\left[
			1 + \frac{36(\omega+1)}{n}	 		
		\right]
	\right).
	\end{align*}
	Setting $\gamma = \frac{1}{L\left(1 + 36(\omega+1)/n\right)}$ gives
	\begin{align*}
		E{\psi^{k+1}}
		&\leq
		\norm{x^k - x^\star}_2^2
		\left(
			1  - \frac{\mu}{L\left(1 + 36(\omega+1)/n\right)}
		\right)
		+
		b \gamma^2 H^k
		\left(
			1 - \frac{\alpha}{2}
		\right)
		+
		c \gamma^2 D^k
		\left(
			1 - \frac{3}{8m}
		\right),
	\end{align*}
	which concludes the proof.
\end{proof}

\subsection{Convex case}

\begin{proof}[Proof of~\Cref{thm:VR-DIANAweak}]
Using \eqref{conv_VR} ($\mu = 0$) assuming that $\alpha \geq \frac{2w}{b n^2}$ and
 $1 \geq \frac{2b \alpha}{c} + \frac{2w}{c n^2}$, we obtain
\begin{align*}
&\E{\norm{x^{k+1} - x^\star}_2^2 + b \gamma^2 H^{k+1} + c \gamma^2 D^{k+1}}  \notag\\
		&\quad \leq	
		\norm{x^{k} - x^\star}_2^2 + b \gamma^2 H^{k} + c \gamma^2 D^{k}
		+
		(f^\star - f(x^k))
	\left(
		2\gamma - 
		2L\gamma^2
		\left[
			1 + \frac{28\omega}{n}	 		
		\right]
	\right),
\end{align*}
which implies
\begin{align*}
(f(x^k) - f^\star)
	\left(
		2\gamma - 
		2L\gamma^2
		\left[
			1 + \frac{128\omega}{n}	 		
		\right]
	\right) =\psi^k - \E{\psi^{k+1}},
\end{align*}
which after removing conditional expectation can be summed over all iterations $1,2, \ldots, k$ and one has
\begin{align*}
\E{f(x^a) - f^\star} \leq \frac{\psi_0}{2k\left(
		\gamma - 
		L\gamma^2
		\left[
			1 + \frac{36(\omega+1)}{n}	 		
		\right]
	\right)},
\end{align*}
where index $a$ is uniformly at random picked from ${1,2,\cdots, k}$ ,which concludes the proof.
\end{proof}

\subsection{Non-convex case}

\begin{theorem}\label{thm:VR-DIANA-non-convex}
    Consider Algorithm~\ref{alg:VR-DIANA} with $\omega$-quantization $Q$,
	and stepsize $\alpha \leq \frac{1}{\omega + 1}$.
Choose any $p>0$ (which will appear in sequence $c^k$ below) to consider the following Lyapunov function
\begin{equation*}
	R^k =f(x^k) + c^k W^k + d^k F^k,
\end{equation*}
where
\begin{align*}
	W^k &= \frac{1}{mn}\sum\limits_{i=1}^{n}\sum\limits_{j=1}^{m} \norm{x^k - w_{ij}^k}_2^2\,\\
	F^k &=
	 \frac{1}{n}\sum\limits_{i=1}^{n} 
		\norm{\nabla f_{i}(x^k) - h_i^k}_2^2\,,
\end{align*}
and
\begin{align*}
	c^k &=
	 c^{k+1}\left( 1 - \frac{1}{m} + \gamma p + \frac{\omega+1}{n} L^2 \gamma^2 \right) + d^{k+1} \left(\alpha L^2 +  \left(1+\frac{2}{\alpha}\right)\frac{\omega+1}{n} L^4 \gamma^2\right)+  \frac{\omega+1}{n} \frac{\gamma^2L^3}{2} \,,\\
	d^k &=
	   d^{k+1}\left( 1 - \frac{\alpha}{2} +  \left(1+\frac{2}{\alpha}\right)\frac{\omega}{n}L^2 \gamma^2\right) + c^{k+1}\frac{\omega}{n} \gamma^2 + \frac{\omega}{n} \frac{\gamma^2L}{2}\,.
\end{align*}
Then	
under Assumption~\ref{assumption:VRdiana}
	\begin{equation*}
		\E{R^{k+1}} \leq
		R^k - \Gamma^k \norm{\nabla f(x^k)}^2_2\, ,
	\end{equation*}
	where 
	\begin{equation*}
		 \Gamma^k = 
		\gamma - \frac{\gamma^2L}{2}  - c^{k+1} \left(\gamma^2 + \frac{\gamma}{p}\right) - d^{k+1} \left(1+\frac{2}{\alpha}\right) L^2 \gamma^2\, .
	\end{equation*}
Taking $x^a \sim_{u.a.r.} \{x^0, x^1,\dots, x^(k-1)\}$ of Algorithm~\ref{alg:VR-DIANA} one obtains
	\begin{equation}
	\label{up:delta}
		\E{\norm{\nabla f(x^a)}^2_2} \leq
		\frac{R^0 - R^k}{k\Delta},
	\end{equation}
where $\Delta = \min_{t \in [k]} \Gamma^t > 0$ .
\end{theorem}

\begin{lemma}  We can upper bound $W^{k+1}$ in the following way 
	\begin{equation}
		\E{W^{k+1}}
		\leq
		\E{\norm{x^{k+1} - x^k}_2^2} + \left(1 - \frac{1}{m} + \gamma p\right)W^k +\frac{\gamma}{p}\norm{\nabla f(x^k)}_2^2.
		\label{lem:up_W_VR_nc}
	\end{equation}
	
\end{lemma}
\begin{proof}
	\begin{align*}
		\E{W^{k+1}}
		&=
		\E{\frac{1}{mn} \sum\limits_{i=1}^{n}\sum\limits_{j=1}^{m} \norm{x^{k+1} - w_{ij}^{k+1}}_2^2}\\
		&=
		\frac{1}{mn} \sum\limits_{i=1}^{n}\sum\limits_{j=1}^{m}\left( \frac{1}{m}\E{\norm{x^{k+1} - x^k}_2^2} + \frac{m-1}{m}\E{\norm{x^{k+1} - w_{ij}^{k}}_2^2} \right)\\
		&=
		\frac{1}{m}\E{\norm{x^{k+1} - x^k}_2^2} + \frac{1}{mn} \sum\limits_{i=1}^{n}\sum\limits_{j=1}^{m} \frac{m-1}{m}\left(\E{\norm{x^{k+1} -x^k + x^k - w_{ij}^{k}}_2^2} \right)
		\\
		&=
		\frac{1}{m}\E{\norm{x^{k+1} - x^k}_2^2} + \frac{1}{mn} \sum\limits_{i=1}^{n}\sum\limits_{j=1}^{m} \frac{m-1}{m}\E{\norm{x^{k+1} -x^k}_2^2} + \norm{x^k - w_{ij}^{k}}_2^2 + 2\dotprod{\E{x^{k+1} -x^k}}{x^k - w_{ij}^{k}}
		\\
		&\leq
		\E{\norm{x^{k+1} - x^k}_2^2} + \frac{1}{mn} \sum\limits_{i=1}^{n}\sum\limits_{j=1}^{m} \frac{m-1}{m} \norm{x^k - w_{ij}^{k}}_2^2 + 2\gamma\left(\frac{1}{2p}\norm{\nabla f(x^k)}_2^2 + \frac{p}{2}\norm{x^k - w_{ij}^{k}}_2^2\right)
		\\
		&=
		\E{\norm{x^{k+1} - x^k}_2^2} + \left(1 - \frac{1}{m} + \gamma p\right)W^k +\frac{\gamma}{p}\norm{\nabla f(x^k)}_2^2,
	\end{align*}
	where the second equality uses the update of $w_{ij}^{k+1}$ in Algorithm~\ref{alg:VR-DIANA} and the inequality uses Cauchy-Schwarz and Young inequalities.
\end{proof}

\begin{lemma} We can upper bound quantity 
$\frac{1}{n}
\sum\limits_{i=1}^{n}
\E{\norm{g_i^k - h_i^k}_2^2}$
in the following way
	\begin{equation}
		\frac{1}{n}
		\sum\limits_{i=1}^{n}
		\E{\norm{g_i^k - h_i^k}_2^2}
		\leq 
		F^k + L^2 W^k.
		\label{lem:up_g-h_VR_nc}
	\end{equation}
\end{lemma}

\begin{proof}
\begin{align*}
\frac{1}{n}
		\sum\limits_{i=1}^{n}
		\E{\norm{g_i^k - h_i^k}_2^2}
		&=
		\frac{1}{n}
		\sum\limits_{i=1}^{n}
		    \left(
			\norm{\E{g_i^k - h_i^k}}_2^2
			+
			\E{\norm{g_i^k - h_i^k - \E{g_i^k - h_i^k}}_2^2}
			\right)\\
		&=
		\frac{1}{n}
		\sum\limits_{i=1}^{n}
		    \left(
			\norm{\nabla f_i(x^k) - h_i^k}_2^2
			+
			\E{\norm{g_i^k - \nabla f_i(x^k)}_2^2}
			\right)\\
		&=
		F^k +
		\frac{1}{n}
		\sum\limits_{i=1}^{n}
			\E{\norm{\nabla f_{ij_i^k}(x^k) - \nabla f_{ij_i^k}(w_{ij_i^k}^k) - \E{\nabla f_{ij_i^k}(x^k) - \nabla f_{ij_i^k}(w_{ij_i^k}^k)}}_2^2}
		\\
		&\overset{\eqref{eq:variance}}{\leq}
		F^k +
		\frac{1}{n}
		\sum\limits_{i=1}^{n}
			 \E{\norm{\nabla f_{ij_i^k}(x^k) - \nabla f_{ij_i^k}(w_{ij_i^k}^k)}_2^2}
			 \\
		&=
		F^k +
		\frac{1}{n}
		\sum\limits_{i=1}^{n}
		\frac{1}{m}
		\sum\limits_{i=1}^{m}
			 \E{\norm{\nabla f_{ij}(x^k) - \nabla f_{ij}(w_{ij}^k)}_2^2}
			 \\
		&\overset{\eqref{def:smoothness}}{\leq}
		F^k + L^2 W^k.
	\end{align*}
\end{proof}
Equipped with this lemma, we are ready to prove a recurrence inequality for $F^k$:
\begin{lemma} Let $\alpha (\omega + 1)\leq 1$. We can upper bound $F^{k+1}$ in the following way
	\begin{equation}
		\E{F^{k+1}}
		\leq
		\left(1+\frac{2}{\alpha}\right)L^2\E{\norm{x^{k+1} - x^k}_2^2} + 
	\left(1-\frac{\alpha}{2}\right) F^k + \alpha L^2 W^k
		\label{lem:up_F_VR_nc}
	\end{equation}
\end{lemma}
\begin{proof}
	\begin{align*}
		\E{F^{k+1}}
		&=
		\frac{1}{n}
		\sum\limits_{i=1}^{n}
				\E{\norm{\nabla f_{i}(x^{k+1}) - h_i^{k+1}}_2^2}\\
		&=
		\frac{1}{n}
		\sum\limits_{i=1}^{n}
				\E{\norm{\nabla f_{i}(x^{k+1}) - \nabla f_{i}(x^k) + \nabla f_{i}(x^k) - h_i^k - \alpha Q\left( g_i^k - h_i^k\right)}_2^2}\\
		& = 
		\frac{1}{n}
		\sum\limits_{i=1}^{n}
		\left(
				 \E{\norm{\nabla f_{i}(x^{k+1}) - \nabla f_{i}(x^k)}_2^2} + \E{\norm{\nabla f_{i}(x^k) - h_i^k - \alpha Q\left( g_i^k - h_i^k\right)}_2^2}
		\right)\\
		& \quad +  (1-\alpha)\frac{1}{n}
		\sum\limits_{i=1}^{n}
		\dotprod{\nabla f_{i}(x^{k+1}) - \nabla f_{i}(x^k)}{ \nabla f_{i}(x^k) - h_i^k } \\
		&\overset{\eqref{def:smoothness}}{\leq}
		\frac{1}{n}
		\sum\limits_{i=1}^{n}
		\left(
				 \left(1 + \frac{1-\alpha}{\tau}\right)L^2\E{\norm{x^{k+1} - x^k}_2^2} + 
		 (1+(1-\alpha)\tau)\norm{\nabla f_{i}(x^k) - h_i^k}_2^2 +  \alpha^2\E{\norm{ Q\left( g_i^k - h_i^k\right)}_2^2}\right) \\
		&\quad - 2 \frac{\alpha}{n} \sum\limits_{i=1}^{n}\dotprod{\nabla f_{i}(x^k) - h_i^k}{\E{Q\left( g_i^k - h_i^k\right)}}
		\\
	&\overset{\eqref{def:omega}}{\leq}
	 \left(1 + \frac{1-\alpha}{\tau}\right)L^2\E{\norm{x^{k+1} - x^k}_2^2} + 
	\frac{1}{n}
		\sum\limits_{i=1}^{n}
		\left( (1+ (1-\alpha)\tau)
		\norm{\nabla f_{i}(x^k) - h_i^k}_2^2 + \alpha^2(\omega+1)\E{\norm{g_i^k - h_i^k}_2^2}\right) \\
		&\quad - 2 \frac{\alpha}{n}\sum\limits_{i=1}^{n}\dotprod{\nabla f_{i}(x^k) - h_i^k}{\nabla f_{i}(x^k) - h_i^k}\\
	&\leq
	 \left(1 + \frac{1-\alpha}{\tau}\right)L^2\E{\norm{x^{k+1} - x^k}_2^2} + 
	(1+(1-\alpha)\tau-2\alpha) F^k
	 + \alpha\frac{1}{n}
		\sum\limits_{i=1}^{n}\E{\norm{g_i^k - h_i^k}_2^2} 
		\\	
	&\overset{\eqref{lem:up_g-h_VR_nc}}{\leq}
	 \left(1 + \frac{1-\alpha}{\tau}\right)L^2\E{\norm{x^{k+1} - x^k}_2^2} + 
	(1+\tau -\alpha) F^k + \alpha L^2 W^k,
	\end{align*}
	where the second equality uses definition of $h_i^{k+1}$ in Algorithm~\ref{alg:VR-DIANA} and the first inequality follows from Cauchy inequality and holds for any $\tau > 0$. 
	
	Taking $\tau = \alpha/2$, we obtain desired inequality.
\end{proof}

\begin{lemma}
We can upper bound the second moment of the $g^k$ in the following way
	\begin{equation}
		\E{\norm{g^k}_2^2}
		\leq
		\frac{\omega}{n}F^k + \frac{\omega+1}{n} L^2 W^k + \norm{\nabla f(x^k)}_2^2
		.
		\label{lem:up_g_VR_nc}
	\end{equation}
\end{lemma}
\begin{proof}
	\begin{align*}
		\EE{Q}{\norm{g^k}_2^2} &\overset{\eqref{eq:variance}}{=}
		\underbrace{
			\norm{\EE{Q}{g^k}}_2^2
		}_{T_1}
		+
		\underbrace{
			\EE{Q}{\norm{g^k - \EE{Q}{g^k}}_2^2}
		}_{T_2}.
	\end{align*}
	We can use the definition of $g^k$ in order to obtain 
	\begin{align*}
		T_1
		&= 
		\norm{
			\frac{1}{n}
			\sum\limits_{i=1}^{n}
				\EE{Q}{Q(g_i^k - h_i^k) + h_i^k}
		}_2^2
		=
		\norm{
			\frac{1}{n}
			\sum\limits_{i=1}^{n}
				g_i^k
		}_2^2
	\end{align*}
	and 
	\begin{align*}
		T_2
		&=
		\EE{Q}{
			\norm{
				\frac{1}{n}
				\sum\limits_{i=1}^{n}
					Q(g_i^k - h_i^k) - (g_i^k - h_i^k)
			}_2^2
		}\\
		&\overset{\eqref{eq:indep}}{=}
		\frac{1}{n^2}
		\sum\limits_{i=1}^{n}
			\EE{Q}{
				\norm{Q(g_i^k - h_i^k) - (g_i^k - h_i^k)}_2^2
			} \\ 
		&\overset{\eqref{def:omega}}{\leq}
		\frac{\omega}{n^2}
		\sum\limits_{i=1}^{n}
			\norm{g_i^k - h_i^k}_2^2.
	\end{align*}
	Now we calculate full expectations conditioned on previous iteration:
	\begin{align*}
		\E{T_2}
		&=
		\frac{\omega}{n^2}
		\sum\limits_{i=1}^{n}
			\E{\norm{g_i^k - h_i^k}_2^2} \\
		&\overset{\eqref{lem:up_g-h_VR_nc}}{\leq}
		\frac{\omega}{n}F^k + \frac{\omega}{n}L^2 W^k.
	\end{align*}
	As for $T_1$, we have
	\begin{align*}
		\E{T_1}
		&=
		\E{\norm{
			\frac{1}{n}
			\sum\limits_{i=1}^{n}
				g_i^k
		}_2^2}
		=
		\norm{
			\frac{1}{n}
			\sum\limits_{i=1}^{n}
			\E{g_i^k}
		}_2^2
		+
		\E{\norm{
				\frac{1}{n}
				\sum\limits_{i=1}^{n}
				g_i^k - \E{g_i^k}
			}_2^2}\\
		&=
		\norm{\nabla f(x^k)}_2^2
		+
		\frac{1}{n^2}
		\sum\limits_{i=1}^{n}
			\E{\norm{g_i^k - \nabla f_i(x^k)}_2^2}\\
		&=
		\norm{\nabla f(x^k)}_2^2
		+
		\frac{1}{n^2}
		\sum\limits_{i=1}^{n}
			\E{\norm{\nabla f_{ij_i^k}(x^k) - \nabla f_{ij_i^k}(w_{ij_i^k}^k) - \E{\nabla f_{ij_i^k}(x^k) - \nabla f_{ij_i^k}(w_{ij_i^k}^k)}}_2^2}\\
		&\overset{\eqref{eq:variance}}{\leq}
		\norm{\nabla f(x^k)}_2^2
		+
		\frac{1}{n^2}
		\sum\limits_{i=1}^{n}
			\E{\norm{\nabla f_{ij_i^k}(x^k) - \nabla f_{ij_i^k}(w_{ij_i^k}^k)}_2^2}\\
		&\overset{\text{Alg.}~\ref{alg:VR-DIANA}}{=}
		\norm{\nabla f(x^k)}_2^2
		+
		\frac{1}{mn^2}
		\sum\limits_{i=1}^{n}
			\sum\limits_{j=1}^{m}
				\norm{\nabla f_{ij}(x^k) - \nabla f_{ij}(w_{ij}^k)}_2^2\\
		&\overset{\eqref{def:smoothness}}{\leq}
		\norm{\nabla f(x^k)}_2^2
		+
		\frac{1}{n}L^2 W^k.
	\end{align*}
	Now, summing $\E{T_1}$ and $\E{T_2}$ we get
	\begin{align*}
		\E{\norm{g^k}_2^2}
		&=
		\E{T_1 + T_2}
		\leq
		\frac{\omega}{n}F^k + \frac{\omega+1}{n} L^2 W^k + \norm{\nabla f(x^k)}_2^2,
	\end{align*}
	which concludes the proof.
\end{proof}

\begin{proof}[Proof of \cref{thm:VR-DIANA-non-convex}]
Using $L$-smoothness one gets
	\begin{align}
		\E{f(x^{k+1})} &\leq f(x^k) + \dotprod{\nabla f(x^k)}{\E{x^{k+1}-x^k}} + \frac{L}{2}\E{\norm{x^{k+1} - x^k}_2^2} \notag \\
		& = f(x^k) - \gamma\norm{\nabla f(x^k)}_2^2 + \frac{L\gamma^2}{2}\E{\norm{g^k}_2^2},
		\label{lem:nc_smooth}
	\end{align}
where we use the definition of $x^{k+1}$ in Algorithm~\ref{alg:VR-DIANA}.

By combining definition of $\E{R^{k+1}}$ with \eqref{lem:up_W_VR_nc}, \eqref{lem:up_F_VR_nc} and \eqref{lem:nc_smooth} one obtains
\begin{align*}
\E{R^{k+1}} &\leq  f(x^k) + \dotprod{\nabla f(x^k)}{\E{x^{k+1}-x^k}} + \frac{L}{2}\E{\norm{x^{k+1} - x^k}_2^2} \\
&\quad + c^{k+1}\left( \E{\norm{x^{k+1} - x^k}_2^2} + \left(1 - \frac{1}{m} + \gamma p\right)W^k +\frac{\gamma}{p}\norm{\nabla f(x^k)}_2^2\right) \\
&\quad + d^{k+1}\left(  \left(1+\frac{2}{\alpha}\right)L^2\E{\norm{x^{k+1} - x^k}_2^2} + \left(1-\frac{\alpha}{2}\right) F^k +\alpha L^2 W^k\right) \\
&=
f(x^k) - \gamma\norm{\nabla f(x^k)}_2^2 + \left(\frac{\gamma^2L}{2} + c^{k+1} \gamma^2 + d^{k+1}  \left(1+\frac{2}{\alpha}\right)L^2 \gamma^2\right)\E{\norm{g^k}_2^2} \\
&\quad + c^{k+1}\left( \left(1 - \frac{1}{m} + \gamma p\right)W^k +\frac{\gamma}{p}\norm{\nabla f(x^k)}_2^2\right) \\
&\quad + d^{k+1}\left(\left(1-\frac{\alpha}{2}\right)F^k + \left(1+\frac{2}{\alpha}\right)\alpha L^2 W^k\right) \\
&\overset{\eqref{lem:up_g_VR_nc}}{\leq}
f(x^k) - \left(\gamma - \frac{\gamma^2L}{2}  - c^{k+1} \gamma^2 - d^{k+1} \left(1+\frac{2}{\alpha}\right) L^2 \gamma^2 - c^{k+1}\frac{\gamma}{p}\right)\norm{\nabla f(x^k)}_2^2\\ 
&\quad + \left( c^{k+1}\left( 1 - \frac{1}{m} + \gamma p\right) + d^{k+1}\left(1+\frac{2}{\alpha}\right) \alpha L^2 +  \frac{\omega+1}{n}  L^2\left(\frac{\gamma^2L}{2}  + c^{k+1} \gamma^2 + d^{k+1}  \left(1+\frac{2}{\alpha}\right)L^2 \gamma^2\right) \right) W^k\\
&\quad + \left( d^{k+1}\left(1-\frac{\alpha}{2}\right) + \frac{\omega}{n} \left(\frac{\gamma^2L}{2}  + c^{k+1} \gamma^2 + d^{k+1}  \left(1+\frac{2}{\alpha}\right)L^2 \gamma^2\right) \right) F^k\\
&=
f(x^k) - \left(\gamma - \frac{\gamma^2L}{2}  - c^{k+1} \left(\gamma^2 + \frac{\gamma}{p}\right) - d^{k+1} \left(1+\frac{2}{\alpha}\right) L^2 \gamma^2\right)\norm{\nabla f(x^k)}_2^2\\ 
&\quad + \left( c^{k+1}\left( 1 - \frac{1}{m} + \gamma p + \frac{\omega+1}{n} L^2 \gamma^2 \right) + d^{k+1} \left(\alpha L^2 +  \left(1+\frac{2}{\alpha}\right)\frac{\omega+1}{n} L^4 \gamma^2\right)+  \frac{\omega+1}{n} \frac{\gamma^2L^3}{2} \right) W^k\\
&\quad + \left( d^{k+1}\left( 1 - \frac{\alpha}{2} +  \left(1+\frac{2}{\alpha}\right)\frac{\omega}{n}L^2 \gamma^2\right) + c^{k+1}\frac{\omega}{n} \gamma^2 + \frac{\omega}{n} \frac{\gamma^2L}{2} \right) F^k \\
&= R^k - \Gamma^k \norm{\nabla f(x^k)}_2^2.
\end{align*}
Applying the full expectation and telescoping the equation, one gets desired inequality.
\end{proof}
We can proceed to the proof of Theorem~\ref{thm:VR-DIANAnc}.
\begin{proof}[Proof of \cref{thm:VR-DIANAnc}]
Recursion for $c^t, d^t$ can be written in a form
\begin{align*}
y^t = Ay^{t+1} + b,
\end{align*}
where 
\begin{align*}
A &= \begin{bmatrix}
 1 - \frac{1}{m} + \gamma p + \frac{\omega+1}{n} L^2 \gamma^2 &\alpha L^2 +  \left(1+\frac{2}{\alpha}\right)\frac{\omega+1}{n} L^4 \gamma^2\\
\frac{\omega}{n} \gamma^2  & 1 - \frac{\alpha}{2} +  \left(1+\frac{2}{\alpha}\right)\frac{\omega}{n}L^2 \gamma^2
\end{bmatrix} , \\
y^t &=  \begin{bmatrix}
 c^t \\
 d^t
\end{bmatrix} , \\
b &=  \begin{bmatrix}
 \frac{\omega+1}{n} \frac{\gamma^2L^3}{2} \\
 \frac{\omega}{n} \frac{\gamma^2L}{2}
\end{bmatrix} .
\end{align*}
Recall that we once used Young's inequality with an arbitrary parameter $p$, so we can now specify it. Choosing $p =  \frac{L\left(1+ \frac{\omega}{n} \right)^{1/2}}{ (m^{2/3} + \omega + 1)^{1/2}}$, $\gamma = \frac{1}{10L\left( 1 + \frac{\omega}{n} \right)^{1/2} (m^{2/3} + \omega + 1)}$, and $\alpha = \frac{1}{\omega+1}$, where  $c^k = d^k = 0$ we can upper bound each element of matrix $A$ and construct its upper bound $\hat{A}$ and a corresponding vector $\hat b$, where 
\begin{align*}
\hat{A} &= \begin{bmatrix}
 1 - \frac{89}{100m}  & L^2 \frac{103}{100(\omega+1)} \\
 \frac{1}{100L^2m} & 1 -  \frac{47}{100(\omega+1)}
 \end{bmatrix},\qquad
 \hat{b} = \left(1+\frac{\omega}{n}\right)\frac{\gamma^2 L}{2}\begin{bmatrix}
 L^2 \\
 1
 \end{bmatrix}.
\end{align*}
Due to the structure of $\hat{A}$ and $\hat{b}$ we can work with matrices 
\begin{align*}
\tilde{A} &= \begin{bmatrix}
 1 - \frac{89}{100m}  & \frac{103}{100(\omega+1)} \\
 \frac{1}{100m} & 1 -  \frac{47}{100(\omega+1)}
 \end{bmatrix},\qquad
 \tilde{b} = \left(1+\frac{\omega}{n}\right)\frac{\gamma^2 L}{2}\begin{bmatrix}
 1 \\
 1
 \end{bmatrix},
\end{align*}
where it holds $\tilde{A}^k \tilde{b} = (y_1, y_2)^\top \implies \hat{A}^k \hat{b} = (L^2y_1, y_2)^\top$, thus we can work with $\tilde{A}$ which is independent of $L$.
In the sense of Lemma~\ref{lem:sequence}, we have that eigenvalues of $\tilde{A}$ are less than $1-\frac{1}{3\max\{m, \omega + 1\}}, 1-\frac{1}{6\min\{m, \omega + 1\}}$, respectively, and $\abs{x} \leq \frac{2}{\min\{m, \omega + 1\}}$, thus
\begin{align*}
c^k &\leq 90\max\{m, \omega+1\}\left(1+\frac{\omega}{n}\right) \frac{\gamma^2L^3}{2} \\
&=  90\max\{m, \omega+1\}\left(1+\frac{\omega}{n}\right) \frac{L^3}{200\left(1 + \frac{\omega}{n}\right) L^2 (m^{2/3} + \omega + 1)^2} \\
&\leq \frac{L}{2(m^{2/3} + \omega + 1)^{1/2}}.
\end{align*}
By the same reasoning
\begin{align*}
d^k \leq \frac{1}{2L(m^{2/3} + \omega + 1)^{1/2}}\, .
\end{align*}
This implies 
\begin{align*}
\Gamma^k & = \gamma - \frac{\gamma^2L}{2}  - c^{k+1} \left(\gamma^2 + \frac{\gamma}{p}\right) - d^{k+1}  \left(1 + \frac{2}{\alpha} \right) L^2 \gamma^2 \\
& \geq \gamma - \frac{\gamma^2L}{2}  -   \frac{L}{2(m^{2/3} + \omega + 1)^{1/2}}\left(\left(2 + \frac{2}{\alpha} \right)\gamma^2 + \frac{\gamma}{p}\right) \\
&\geq  \frac{1}{10L\left( 1 + \frac{\omega}{n} \right)^{1/2} (m^{2/3} + \omega + 1)} -  \frac{1}{200L \left(1 + \frac{\omega}{n}\right)(m^{2/3} + \omega + 1)^2 } \\
& \quad  -  \frac{1}{2L(m^{2/3} + \omega + 1)^{1/2}}\left(\frac{4(\omega + 1)}{100\left(1 + \frac{\omega}{n} \right) (m^{2/3} + \omega + 1)^2} + \frac{1}{10\left(1 + \frac{\omega}{n} \right) (m^{2/3} + \omega + 1)^{1/2}}\right) \\
& \geq \frac{1}{40L\left(1 +  \frac{\omega}{n} \right)^{1/2}(m^{2/3} + \omega + 1)},
\end{align*}
which guarantees $\Delta \geq \frac{1}{40L\left(1 +  \frac{\omega}{n} \right)^{1/2}(m^{2/3} + \omega + 1)}$. Plugging the lower bound on $\Delta$ into \eqref{up:delta} one completes the proof.
\end{proof}

\section{Variance Reduced Diana - SVRG proof}

\begin{figure}[t]
\begin{algorithm}[H]
	\removelatexerror
	\begin{algorithm2e}[H]
	\DontPrintSemicolon
		\KwIn{$p \geq 1$, learning rates $\alpha, \gamma>0$, initial vectors $x_0, h_{1}^0, \ldots, h_{n}^0 \in \R^d$, $p_0, p_1, \cdots, p_{l-1} \in \R$}
		$s = 0$\\
		$x^0 = x_0$\\
		$z^0 = x^0$ \\
		\For{$k = 1,2,\ldots$}{
			Broadcast $x^k$ to all workers\\
			\For(in parallel){$i = 1, \ldots, n$}{
				\uIf{$k \equiv 0 \mod l$}{
    					$s = s+1$ \\
    					$z^s =\sum_{r=0}^{l-1} p_r x^{(s-1)l+r}$\\
 				 }
				Pick random $j_i^k \in [m]$ uniformly\\
				$g_i^k = \nabla f_{ij_i^k}(x^k) - \nabla f_{ij_i^k}(z^s)  + \nabla f_{i}(z^s)$ \\ 
				$\hat{\Delta}_i^k = Q( g_i^k - h_{i}^k)$\\
				$h_{i}^{k+1} = h_{i}^k + \alpha \hat{\Delta}_i^k$\\
			}
			$g^k = \frac{1}{n}\sum\limits_{i=1}^{n} (\hat{\Delta}_i^k + h_i^k)$\\
			$x^{k+1} = x^k - \gamma g^k$\\
		}
	\end{algorithm2e}	
	\caption{SVRG-DIANA}
	\label{alg:SVRG_DIANA}
	\end{algorithm}
\end{figure}

\begin{lemma}
For all iterates $k \geq 0$ of Algorithm~\ref{alg:SVRG_DIANA}, it holds
	\begin{equation*}
		\E{g^k} = \nabla f(x^k)
	\end{equation*}
\end{lemma}
\begin{proof}
	\begin{align*}
		\E{g^k}
		&=
		\frac{1}{n}\sum\limits_{i=1}^n \E{Q\left( g_i^k - h_{i}^k\right) + h_i^k}
		=
		\frac{1}{n}\sum\limits_{i=1}^n \E{ g_i^k- h_{i}^k + h_i^k}\\
		&=
		\frac{1}{n}\sum\limits_{i=1}^n \E{ g_i^k}\\
		&=
		\frac{1}{n}\sum\limits_{i=1}^n \nabla f_i(x^k) = \nabla f(x^k),
	\end{align*}
	where the first inequality follows from definition of $g^k$ in Algorithm~\ref{alg:SVRG_DIANA}.
\end{proof}

\subsection{Strongly convex case}

To prove the convergence of \cref{alg:SVRG_DIANA} we consider Lyapunov function of the following form:
\begin{equation}
\label{def:SVRG_psi^k}
	\psi^s = \left( f(z^s) - f^\star\right) + \bar{b} \gamma^2 \bar{H}^s,
\end{equation}
where
\begin{equation}
\label{def:SVRG_H^k}
	H^k \eqdef \sum\limits_{i=1}^{n} \norm{h_{i}^k - \nabla f_{i}(x^\star)}_2^2
\end{equation}
and $\bar{H}^s = H^{ls}$.

The following theorem establishes linear convergence rate of \cref{alg:SVRG_DIANA}.

\begin{theorem}\label{thm:SVRG-DIANA}
	Under Assumptions~\ref{assumption:VRdiana} and \ref{assumption:VRdianasc} For $\alpha \leq \frac{1}{\omega + 1}$, the following inequality holds:
	\begin{equation}
	\E{\psi^{s+1}} \leq \psi^s \max
	\left\{
	\frac{(1-\theta)^l}{1-(1-\theta)^l}\frac{2\theta + (1-(1-\theta)^l)c\mu}{\mu (2\gamma-c)}
	,
	\left(1
	-
	\theta\right)^l
	\right\},
	\end{equation}
	where $c = \frac{6L\omega}{n} \gamma^2
		 + 
		\left(2L+\frac{4L}{n}\right)\gamma^2
		 +
		 4b \gamma^2 L\alpha n$, $\theta = \min\{\mu\gamma, \alpha
		-
		\frac{3\omega}{n^2b}\}$, $p_r = \frac{(1-\theta)^{l-1-r}}{\sum_{t=0}^{l-1}(1-\theta)^{l-1-t}}$ for $r = 0,1,\dots, l-1$, and $b = \bar{b}l(2\gamma-c)$.
\end{theorem}

\begin{corollary}
	Taking $\alpha = \frac{1}{\omega + 1}$, $b = 6\frac{\omega}{n^2 \alpha}$, $\gamma = \frac{1}{10L\left( 2+ \frac{4}{n} + 30\frac{\omega}{n} \right)}$, and $l = \frac{2}{\theta}$
 SVRG-DIANA needs
	$
		O\left(
			\left( \kappa + \kappa\frac{\omega}{n} + \omega + m\right) \log \frac{1}{\epsilon}
		\right)
	$
	iterations
	to achieve precision $\E{\psi^s} \leq \varepsilon \psi^0$.
\end{corollary}

\begin{lemma}
We can upper bound the second moment of the $g^k$ in the following way
	\begin{equation}
		\E{\norm{g^k}_2^2} \leq \left(2L + \frac{4L}{n} + \frac{6L\omega}{n} \right) (f(x^k) - f^\star  + f(z^s) - f^\star )
		+
		\frac{3\omega}{n^2} H^k,
		\label{lem:def_g_SVRG}
	\end{equation}
\end{lemma}
\begin{proof}
	\begin{align*}
	&\E{\norm{g^k}_2^2}
	=
	\E{\EE{Q}{\norm{g^k}_2^2}}
	\overset{\eqref{eq:variance}}{=}
	\E{
		\norm{\EE{Q}{g^k}}_2^2
		+
		\EE{Q}{\norm{g^k - \EE{Q}{g^k}}_2^2}
	}
	\\
	&=
	\E{
		\underbrace{
			\norm{
				\frac{1}{n}\sum\limits_{i=1}^{n} g_i^k 
			}_2^2
		}_{T_1}
		+
		\underbrace{
			\EE{Q}{
				\norm{
					\frac{1}{n}\sum\limits_{i=1}^{n}
					Q\left(g_i^k - h_{i}^k\right)
					-
					\left(g_i^k - h_{i}^k\right)
				}_2^2
			}	
		}_{T_2}
	},
	\end{align*}
where the third inequality uses the definition of $g^k$ in Algorithm~\ref{alg:SVRG_DIANA}.  We can further bound $\E{T_1}$ and $\E{T_2 }$.
\begin{align*}
		  \E{T_1}
		&=
		\E{\norm{
			\frac{1}{n}
			\sum\limits_{i=1}^{n}
				g_i^k
		}_2^2}
		=
		\norm{
			\frac{1}{n}
			\sum\limits_{i=1}^{n}
			\E{g_i^k}
		}_2^2
		+
		\E{\norm{
				\frac{1}{n}
				\sum\limits_{i=1}^{n}
				\left( g_i^k - \E{g_i^k}\right)
			}_2^2}\\
		&=
		\norm{\nabla f(x^k)}_2^2
		+
		\frac{1}{n^2}
		\sum\limits_{i=1}^{n}
			\E{\norm{g_i^k - \nabla f_i(x^k)}_2^2}\\
		&\overset{\eqref{def:smoothness}}{\leq}
		2L(f(x^k) - f^\star)
		+
		\frac{1}{n^2}
		\sum\limits_{i=1}^{n}
			\E{\norm{\nabla f_{ij_i^k}(x^k) - \nabla f_{ij_i^k}(z^s) - \E{\nabla f_{ij_i^k}(x^k) - \nabla f_{ij_i^k}(z^s)}}_2^2}\\
		&\overset{\eqref{eq:variance}}{\leq}
		2L(f(x^k) - f^\star)
		+
		\frac{1}{n^2}
		\sum\limits_{i=1}^{n}
			\E{\norm{\nabla f_{ij_i^k}(x^k) - \nabla f_{ij_i^k}(z^s)}_2^2}\\
		&\overset{\text{Alg.}~\ref{alg:SVRG_DIANA}}{=}
		2L(f(x^k) - f^\star)
		+
		\frac{1}{mn^2}
		\sum\limits_{i=1}^{n}
			\sum\limits_{j=1}^{m}
				\norm{\nabla f_{ij}(x^k) - \nabla f_{ij}(z^s)}_2^2\\
		&\overset{\eqref{eq:sum}}{\leq}
		2L(f(x^k) - f^\star)
		+
		\frac{2}{mn^2}
		\sum\limits_{i=1}^{n}
			\sum\limits_{j=1}^{m}
			    \left(
				\norm{\nabla f_{ij}(x^k) - \nabla f_{ij}(x^\star)}_2^2
				+
				\norm{\nabla f_{ij}(z^s) - \nabla f_{ij}(x^\star)}_2^2
				\right)\\
		&\overset{\eqref{def:smoothness}}{\leq}
		\left(2L + \frac{4L}{n}\right)
		(f(x^k) - f^\star  + f(z^s) - f^\star)
		\\
		\E{T_2}
		&\overset{\eqref{eq:indep}}{=}
		\frac{1}{n^2}
		\E{
			\sum\limits_{i=1}^{n}
			\EE{Q}{
				\norm{
					Q\left( g_i^k - h_{i}^k\right)
					-
					\left(g_i^k - h_{i}^k\right)
				}_2^2
			}
		}\\
		&
		\overset{\eqref{eq:variance}}{\leq}
		\frac{\omega}{n^2}
		\sum\limits_{i=1}^{n} 
		\E{\norm{ g_i^k - h_{i}^k}_2^2}\\
		&
		\overset{\eqref{eq:sum}+\text{Alg.}~\ref{alg:SVRG_DIANA}}{\leq}
		\frac{3\omega}{n^2}
		\sum\limits_{i=1}^{n} 
		\E{
			\norm{\nabla f_{ij_i^k}(x^k) - \nabla f_{ij_i^k}(x^\star)}_2^2
			+
			\norm{\nabla f_{ij_i^k}(z^s) - \nabla f_{ij_i^k}(x^\star) - (\nabla f_{i}(z^s) - \nabla f_{i}(x^\star))}_2^2} \\
			& \quad +			
			\E{\norm{h_{i}^k - \nabla f_{i}(x^\star)}_2^2
		}\\
		&
		\overset{\eqref{eq:variance}}{\leq}
		\frac{3\omega}{n^2}
		\sum\limits_{i=1}^{n} 
		\E{
			\norm{\nabla f_{ij_i^k}(x^k) - \nabla f_{ij_i^k}(x^\star)}_2^2
			+
			\norm{\nabla f_{ij_i^k}(z^s) - \nabla f_{ij_i^k}(x^\star)}_2^2}
			 +			
			\E{\norm{h_{i}^k - \nabla f_{i}(x^\star)}_2^2
		}\\
		&\overset{\eqref{def:smoothness}}{\leq}
		\frac{6L\omega}{n} (f(x^k) - f^\star  + f(z^s) - f^\star )
		+
		\frac{3\omega}{n^2} H^k
	\end{align*}
	
	Summing up $\E{T_1}$ and $\E{T_2}$ we conclude the proof:
	\begin{align*}
		\E{\norm{g^k}_2^2} = \E{T_1 + T_2} \leq
		\left(2L + \frac{4L}{n} + \frac{6L\omega}{n} \right) (f(x^k) - f^\star  + f(z^s) - f^\star )
		+
		\frac{3\omega}{n^2} H^k.
	\end{align*}

\end{proof}

\begin{lemma}
 Let $\alpha(\omega+1) \leq 1$. We can upper bound $H^{k+1}$ in the following way 
	\begin{equation}
		\E{H^{k+1}} \leq H^k\left(1 - \alpha\right)
		+
		4L\alpha n (f(x^k) - f^\star + f(z^s) - f^\star).
		\label{lem:def_H_SVRG}
	\end{equation}
\end{lemma}
\begin{proof}
	\begin{align*}
		\E{H^{k+1}}&=
		\sum\limits_{i=1}^n \E{\norm{h_{i}^{k+1} - h_{i}^{k} + h_{i}^{k} - \nabla f_{i}(x^\star)}_2^2}
		\\
		&=
		\sum\limits_{i=1}^n
		\E{
			\norm{h_{i}^k - \nabla f_{i}(x^\star)}_2^2
			+
			2\dotprod{h_{i}^{k+1} - h_{i}^k}{h_{i}^k - \nabla f_{ij}(x^\star)}
			+
			\norm{h_{i}^{k+1} - h_{i}^k}_2^2
		}\\
		&=
		H^k + 
		\E{\sum\limits_{i=1}^n
		\left(
			2\dotprod{\EE{Q}{h_{i}^{k+1} - h_{i}^k}}{h_{i}^k - \nabla f_{i}(x^\star)}
			+
			\EE{Q}{\norm{h_{i}^{k+1} - h_{i}^k}_2^2}
		\right)}
	\end{align*}
	No we calculate expectations:
	\begin{align*}
		\EE{Q}{h_{i}^{k+1} - h_{i}^k}
		= \alpha
		\EE{Q}{\hat{\Delta}_i^k}
		=
		\alpha
		(g_i^k - h_{i}^k)
	\end{align*}
	\begin{align*}
		\EE{Q}{\norm{h_{i}^{k+1} - h_{i}^k}_2^2}
		&=
		\alpha^2
		\EE{Q}{\norm{\hat{\Delta}_i^k}_2^2}
		=
		\alpha^2
		(\omega + 1) \norm{g_i^k - h_{i}^k}_2^2\\
		&\leq
		\alpha \norm{g_i^k - h_{i}^k}_2^2
	\end{align*}
	Finally we obtain
	\begin{align*}
		\E{H^{k+1}}
		&=
		H^k
		+
		\E{\sum\limits_{i=1}^n
		\left(
			2\alpha\dotprod{g_i^k - h_{i}^k}{h_{i}^k - \nabla f_{i}(x^\star)}
			+
			\alpha\norm{g_i^k - h_{i}^k}_2^2
		\right)}\\
		&=
		H^k
		+
		\frac{\alpha}{m}
		\E{\sum\limits_{i=1}^n\sum\limits_{j=1}^m
		\dotprod{g_i^k - h_{i}^k}{g_i^k + h_{i}^k - 2\nabla f_{i}(x^\star)}}\\
		&=
		H^k
		+
		\E{\frac{\alpha}{m}
		\sum\limits_{i=1}^n\sum\limits_{j=1}^m
		\left(
			\norm{g_i^k - \nabla f_{i}(x^\star) }_2^2
			-
			\norm{h_{i}^k - \nabla f_{i}(x^\star) }_2^2
		\right)}\\
		&\overset{\eqref{def:smoothness}}{\leq}
		H^k\left(1 - \alpha\right)
		+
		4L\alpha n (f(x^k) - f^\star + f(z^s) - f^\star),
	\end{align*}
	which concludes the proof.
\end{proof}

\begin{proof}[Proof of \cref{thm:SVRG-DIANA}]
	\begin{align}
		\E{\norm{x^{k+1} - x^\star}^2_2 + b \gamma^2 H^{k+1}} 
		&= 
		\norm{x^{k} - x^\star}^2_2 
		+ 
		2\gamma\dotprod{x^{k} - x^\star}{\E{g^k}} 
		+
		\gamma^2 \E{\norm{g^k}_2^2}
		+
		b \gamma^2\E{ H^{k+1}} \notag\\
		&=
		\norm{x^{k} - x^\star}^2_2 
		+ 
		2\gamma \dotprod{x^{k} - x^\star}{\nabla f(x^k)}
		+
		\gamma^2 \E{\norm{g^k}_2^2}
		+
		b \gamma^2\E{ H^{k+1}}\notag \\	
		&\overset{\eqref{lem:def_g_SVRG}+\eqref{lem:def_H_SVRG}+\eqref{def:strongconvex}}{\leq}
		(1-\mu\gamma)\norm{x^{k} - x^\star}^2_2 
		+ 
		\left(
		L\frac{6\omega}{n} \gamma^2
		 + 
		 \left(2L + \frac{4L}{n}\right)\gamma^2
		 +
		 4b \gamma^2 L\alpha n
		 - 
		 2\gamma
		\right)
		(f(x^k)- f^\star) \notag\\
		&\quad +
		\left(
		L\frac{6\omega}{n} \gamma^2
		 + 
		 \left(2L + \frac{4L}{n}\right)\gamma^2
		 +
		 4b \gamma^2 L\alpha n
		\right)
		(f(z^s)- f^\star) \notag\\
		&\quad +
		b \gamma^2 H^k
		\left(
		1
		-
		\alpha
		+
		\frac{3\omega}{n^2b}
		\right)		
		\label{dasdmasndlka}
	\end{align}
	Let $c = L\frac{6\omega}{n} \gamma^2
		 + 
		  \left(2L + \frac{4L}{n}\right)\gamma^2
		 +
		 4b \gamma^2 L\alpha n$, $\theta = \min\{\mu\gamma, \alpha
		-
		\frac{3\omega}{n^2b}\}$, $p_r = \frac{(1-\theta)^{l-1-r}}{\sum_{t=0}^{l-1}(1-\theta)^{l-1-t}}$ for $r = 0,1,\dots, l-1$, and assume that $b$ is picked such that  $\alpha
		-
		\frac{3\omega}{n^2b} > 0$ . We can apply previous inequality recursively for $ k = (s+1)l, (s+1)l-1,\dots, sl+1$, which implies
		\begin{align}
		&\E{\norm{x^{(s+1)l} - x^\star}^2_2 + b \gamma^2 \bar{H}^{s+1} + \frac{1-(1-\theta)^l}{\theta}(2\gamma - c) (f(z^{s+1})- f^\star)} \notag\\
		&\quad \leq	
		(1-\theta)^l\norm{z^s - x^\star}^2_2 
		+ 
		\frac{1-(1-\theta)^l}{\theta}c
		(f(z^s)- f^\star) 
		+
		b \gamma^2 \bar{H}^s
		\left(
		1
		-
		\theta
		\right)^l \label{conv_SVRG}\\
		&\quad \leq	
		\frac{2}{\mu}(1-\theta)^l(f(z^s)- f^\star) 
		+ 
		\frac{1-(1-\theta)^l}{\theta}c
		(f(z^s)- f^\star) 
		+
		b \gamma^2 \bar{H}^s
		\left(
		1
		-
		\theta
		\right)^l	. \notag		
	\end{align}
	Choosing $\bar{b} = \frac{b}{l(2\gamma-c)}$ we got
	\begin{align*}
		\E{\psi^{k+1}}
		 \leq	
		\frac{(1-\theta)^l}{1-(1-\theta)^l}\frac{2\theta + (1-(1-\theta)^l)c\mu}{\mu (2\gamma-c)}(f(z^s)- f^\star) 
		+ 
		\bar{b} \gamma^2 \bar{H}^s
		\left(
		1
		-
		\theta
		\right)^l
	\end{align*}
	which concludes the proof.
\end{proof}
\subsection{Convex case}

Let us look at the convergence under weak convexity assumption, thus $\mu = 0$.

\begin{theorem}
\label{thm:conv_SVRG-DIANA}
Let $p_r = 1/l$ for $r = 0,1,\dots, l-1$ and $\alpha \leq \frac{1}{\omega + 1}$.
Under Assumptions~\ref{assumption:VRdiana} and \ref{assumption:VRdianac}, output   $x^a \sim_{u.a.r.} \{x^0 ,x^1,\dots, x^{k-1}\}$ of Algorithm~\ref{alg:SVRG_DIANA} satisfies
\begin{equation}
\E{f(x^a) - f^\star} \leq \frac{\norm{x_0 - x^\star}_2^2 + lc(f(x_0)-f^\star)) + b \gamma^2H^0}{2k(\gamma-c)}, 
\end{equation} 
where $c = L\gamma^2\left(\frac{6\omega}{n}
		 + 
		 2 + \frac{1}{n}
		 +
		 4b\alpha n\right)$
		 and $k$ is number of iterations, which is multiple of $l$, the inner loop size.
\end{theorem}

\begin{corollary}
    Let $\gamma = \frac{1}{L\sqrt{m}\left(2+\frac{4}{n} + 18\frac{\omega}{n}\right)}$, $b = \frac{3\omega(\omega +1)}{ n^2}$, $l = m$ and $\alpha = \frac{1}{\omega + 1}$.
	To achieve precision $\E{f(x^a) - f^\star} \leq\varepsilon$ SVRG-DIANA needs
	$
	\cO\left(
	\frac{\left(1+\frac{\omega}{n}\right)\sqrt{m} + \frac{\omega}{\sqrt{m}}}{\epsilon}
	\right)
	$
	iterations.
\end{corollary}

\begin{proof}[Proof of \cref{thm:conv_SVRG-DIANA}]
Using \eqref{dasdmasndlka} assuming that $b$ is picked such that  $\alpha
		-
		\frac{3\omega}{n^2b} \geq 0$, we obtain
\begin{align}
&\E{\norm{x^{k+1} - x^\star}^2_2 + b \gamma^2 H^{k+1}}  \notag\\
		&\quad \leq	
		\norm{x^{k} - x^\star}^2_2 
		+ 
		\left(
		c
		 - 
		 2\gamma
		\right)
		(f(x^k)- f^\star) 
		+
		c
		(f(z^s)- f^\star) 
		 +
		b \gamma^2 H^k		
\end{align}
Taking $P^s = \E{\norm{x^{sl} - x^\star}^2_2 
		+ 
		lc
		(f(z^s)- f^\star)
		+
		b \gamma^2 \bar{H}^{s}}$, full expectation, and using 
		\begin{align*}
		\E{l(f(z^{s+1})- f^\star)} = \sum_{j = sl}^{(s+1)l-1}(f(x^j)- f^\star),
		\end{align*}
		  we get
\begin{align*}
 (2\gamma - 2c)\sum_{j = sl}^{(s+1)l-1}(f(x^j)- f^\star) \leq P^s - P^{s+1},
\end{align*}
which can be summed over all epochs and one obtains
\begin{align*}
\E{(f(x^a)- f^\star)} \leq \frac{P^0}{2k(\gamma - c)},
\end{align*}
which concludes the proof.
\end{proof}

\subsection{Non-convex case}

\begin{theorem}\label{thm:SVRG-DIANA-non-convex}
    Consider Algorithm~\ref{alg:SVRG_DIANA} with $\omega$-quantization $Q$,
	and stepsize $\alpha \leq \frac{1}{\omega + 1}$.
We consider the following Lyapunov function
\begin{equation*}
	R^k =f(x^k) + c^k Z^k + d^k F^k,
\end{equation*}
where
\begin{equation*}
	Z^k = \norm{x^k - z^s}_2^2\,
\end{equation*}
and
\begin{equation*}
	F^k =
	 \frac{1}{n}\sum\limits_{i=1}^{n} 
		\norm{\nabla f_{i}(x^k) - h_i^k}_2^2\,,
\end{equation*}
and 
\begin{equation*}
	c^k =
	 c^{k+1}\left( 1 + \gamma p + \frac{\omega+1}{n} L^2 \gamma^2 \right) + d^{k+1} \left(\alpha L^2 +  \left(1+\frac{2}{\alpha}\right)\frac{\omega+1}{n} L^4 \gamma^2\right)+  \frac{\omega+1}{n} \frac{\gamma^2L^3}{2} \,,
\end{equation*}
and
\begin{equation*}
	d^k =
	   d^{k+1}\left( 1 - \frac{\alpha}{2} +  \left(1+\frac{2}{\alpha}\right)\frac{\omega}{n}L^2 \gamma^2\right) + c^{k+1}\frac{\omega}{n} \gamma^2 + \frac{\omega}{n} \frac{\gamma^2L}{2}\,.
\end{equation*}
Then	
under Assumption~\ref{assumption:VRdiana}
	\begin{equation*}
		\E{R^{k+1}} \leq
		R^k - \Gamma^k \norm{\nabla f(x^k)}^2_2\, ,
	\end{equation*}
	where 
	\begin{equation*}
		 \Gamma^k = 
		\gamma - \frac{\gamma^2L}{2}  - c^{k+1} \left(\gamma^2 + \frac{\gamma}{p}\right) - d^{k+1} \left(1+\frac{2}{\alpha}\right) L^2 \gamma^2\, .
	\end{equation*}
Taking $x^a \sim_{u.a.r.} \{x^0,\dots, x^{l-1}\}$ of Algorithm~\ref{alg:SVRG_DIANA} one obtains
	\begin{equation}
	\label{up:delta_SVRG}
		\E{\norm{\nabla f(x^a)}^2_2} \leq
		\frac{R^0 - R^l}{k\Delta},
	\end{equation}
where $\Delta = \min_{t \in [k]} \Gamma^t > 0$ .
\end{theorem}

\begin{theorem}\label{thm:SVRG-DIANAnc}
Let Assumption~\ref{assumption:VRdiana} hold. Moreover, let   $\gamma = \frac{1}{10L\left( 1 + \frac{\omega}{n} \right)^{1/2} (m^{2/3} + \omega + 1)}$, $l = m$, $p_{l-1} = 1$, $p_r = 0$ for $r = 0,1,\hdots, l-2$, and $\alpha = \frac{1}{\omega+1}$, then  a randomly chosen iterate $x^a \sim_{u.a.r.} \{x^0,x^1,\dots, x^{k-1}\}$ of Algorithm~\ref{alg:SVRG_DIANA} satisfies
\begin{align*}
\E{\norm{\nabla f(x^a)}_2^2} \leq 
\frac{40(f(x^0) - f^\star)L\left(1 +  \frac{\omega}{n} \right)^{1/2}(m^{2/3} + \omega + 1)}{k}, 
\end{align*}
where $k$ denotes the number of iterations, which is multiple of $m$.
\end{theorem}

\begin{corollary}
	To achieve precision $\E{\norm{\nabla f(x^a)}_2^2} \leq\varepsilon$ SVRG-DIANA needs
	$
	\cO\left(
	\left(1 + \frac{\omega}{n} \right)^{1/2} \frac{ m^{2/3} + \omega}{\varepsilon}
	\right)
	$
	iterations.
\end{corollary}

\begin{lemma}  We can upper bound $Z^{k+1}$ in the following way 
	\begin{equation}
		\E{Z^{k+1}}
		\leq
		\E{\norm{x^{k+1} - x^k}_2^2} + \left(1 + \gamma p\right)Z^k +\frac{\gamma}{p}\norm{\nabla f(x^k)}_2^2.
		\label{lem:up_W_VR_nc_SVRG}
	\end{equation}
	
\end{lemma}
\begin{proof}
	\begin{align*}
		\E{Z^{k+1}}
		&=
		\E{\norm{x^{k+1} - z^s}_2^2} = \E{\norm{x^{k+1} - x^k + x^k - z^s}_2^2}\\
		&=
		 \E{\norm{x^{k+1} - x^k}_2^2} + \norm{x^{k} - z^s}_2^2
		 + 2\dotprod{\E{x^{k+1} -x^k}}{x^k - w_{ij}^{k}}
		\\
		&\leq
		\E{\norm{x^{k+1} - x^k}_2^2} + \norm{x^k - z^s}_2^2 + 2\gamma\left(\frac{1}{2p}\norm{\nabla f(x^k)}_2^2 + \frac{p}{2}\norm{x^k - z^s}_2^2\right)
		\\
		&=
		\E{\norm{x^{k+1} - x^k}_2^2} + \left(1 + \gamma p\right)Z^k +\frac{\gamma}{p}\norm{\nabla f(x^k)}_2^2,
	\end{align*}
	where the inequality uses Cauchy-Schwarz and Young inequalities with $p>0$.
\end{proof}

\begin{lemma} We can upper bound quantity 
$\frac{1}{n}
\sum\limits_{i=1}^{n}
\E{\norm{g_i^k - h_i^k}_2^2}$
in the following way
	\begin{equation}
		\frac{1}{n}
		\sum\limits_{i=1}^{n}
		\E{\norm{g_i^k - h_i^k}_2^2}
		\leq 
		F^k + L^2 Z^k.
		\label{lem:up_g-h_VR_nc_SVRG}
	\end{equation}
\end{lemma}

\begin{proof}
\begin{align*}
\frac{1}{n}
		\sum\limits_{i=1}^{n}
		\E{\norm{g_i^k - h_i^k}_2^2}
		&=
		\frac{1}{n}
		\sum\limits_{i=1}^{n}
		    \left(
			\norm{\E{g_i^k - h_i^k}}_2^2
			+
			\E{\norm{g_i^k - h_i^k - \E{g_i^k - h_i^k}}_2^2}
			\right)\\
		&=
		\frac{1}{n}
		\sum\limits_{i=1}^{n}
		    \left(
			\norm{\nabla f_i(x^k) - h_i^k}_2^2
			+
			\E{\norm{g_i^k - \nabla f_i(x^k)}_2^2}
			\right)\\
		&=
		F^k +
		\frac{1}{n}
		\sum\limits_{i=1}^{n}
			\E{\norm{\nabla f_{ij_i^k}(x^k) - \nabla f_{ij_i^k}(z^s) - \E{\nabla f_{ij_i^k}(x^k) - \nabla f_{ij_i^k}(z^s)}}_2^2}
		\\
		&\overset{\eqref{eq:variance}}{\leq}
		F^k +
		\frac{1}{n}
		\sum\limits_{i=1}^{n}
			 \E{\norm{\nabla f_{ij_i^k}(x^k) - \nabla f_{ij_i^k}(z^s)}_2^2}
			 \\
		&=
		F^k +
		\frac{1}{n}
		\sum\limits_{i=1}^{n}
		\frac{1}{m}
		\sum\limits_{i=1}^{m}
			 \E{\norm{\nabla f_{ij}(x^k) - \nabla f_{ij}(z^s)}_2^2}
			 \\
		&\overset{\eqref{def:smoothness}}{\leq}
		F^k + L^2 Z^k.
	\end{align*}
\end{proof}
Equipped with this lemma, we are ready to prove a recurrence inequality for $F^k$:
\begin{lemma} Let $\alpha (\omega + 1)\leq 1$. We can upper bound $F^{k+1}$ in the following way
	\begin{equation}
		\E{F^{k+1}}
		\leq
		\left(1+\frac{2}{\alpha}\right)L^2\E{\norm{x^{k+1} - x^k}_2^2} + 
	\left(1-\frac{\alpha}{2}\right) F^k + \alpha L^2 Z^k
		\label{lem:up_F_VR_nc_SVRG}
	\end{equation}
\end{lemma}
\begin{proof}
	\begin{align*}
		\E{F^{k+1}}
		&=
		\frac{1}{n}
		\sum\limits_{i=1}^{n}
				\E{\norm{\nabla f_{i}(x^{k+1}) - h_i^{k+1}}_2^2}\\
		&=
		\frac{1}{n}
		\sum\limits_{i=1}^{n}
				\E{\norm{\nabla f_{i}(x^{k+1}) - \nabla f_{i}(x^k) + \nabla f_{i}(x^k) - h_i^k - \alpha Q\left( g_i^k - h_i^k\right)}_2^2}\\
		& = 
		\frac{1}{n}
		\sum\limits_{i=1}^{n}
		\left(
				 \E{\norm{\nabla f_{i}(x^{k+1}) - \nabla f_{i}(x^k)}_2^2} + \E{\norm{\nabla f_{i}(x^k) - h_i^k - \alpha Q\left( g_i^k - h_i^k\right)}_2^2}
		\right)\\
		& \quad +  (1-\alpha)\frac{1}{n}
		\sum\limits_{i=1}^{n}
		\dotprod{\nabla f_{i}(x^{k+1}) - \nabla f_{i}(x^k)}{ \nabla f_{i}(x^k) - h_i^k } \\
		&\overset{\eqref{def:smoothness}}{\leq}
		\frac{1}{n}
		\sum\limits_{i=1}^{n}
		\left(
				 \left(1 + \frac{1-\alpha}{\tau}\right)L^2\E{\norm{x^{k+1} - x^k}_2^2} + 
		 (1+(1-\alpha)\tau)\norm{\nabla f_{i}(x^k) - h_i^k}_2^2 +  \alpha^2\E{\norm{ Q\left( g_i^k - h_i^k\right)}_2^2}\right) \\
		&\quad - 2 \frac{\alpha}{n} \sum\limits_{i=1}^{n}\dotprod{\nabla f_{i}(x^k) - h_i^k}{\E{Q\left( g_i^k - h_i^k\right)}}
		\\
	&\overset{\eqref{def:omega}}{\leq}
	 \left(1 + \frac{1-\alpha}{\tau}\right)L^2\E{\norm{x^{k+1} - x^k}_2^2} + 
	\frac{1}{n}
		\sum\limits_{i=1}^{n}
		\left( (1+ (1-\alpha)\tau)
		\norm{\nabla f_{i}(x^k) - h_i^k}_2^2 + \alpha^2(\omega+1)\E{\norm{g_i^k - h_i^k}_2^2}\right) \\
		&\quad - 2 \frac{\alpha}{n}\sum\limits_{i=1}^{n}\dotprod{\nabla f_{i}(x^k) - h_i^k}{\nabla f_{i}(x^k) - h_i^k}\\
	&\leq
	 \left(1 + \frac{1-\alpha}{\tau}\right)L^2\E{\norm{x^{k+1} - x^k}_2^2} + 
	(1+(1-\alpha)\tau-2\alpha) F^k
	 + \alpha\frac{1}{n}
		\sum\limits_{i=1}^{n}\E{\norm{g_i^k - h_i^k}_2^2} 
		\\	
	&\overset{\eqref{lem:up_g-h_VR_nc_SVRG}}{\leq}
	 \left(1 + \frac{1-\alpha}{\tau}\right)L^2\E{\norm{x^{k+1} - x^k}_2^2} + 
	(1+\tau-\alpha) F^k + \alpha L^2 Z^k,
	\end{align*}
	where the second equality uses definition of $h_i^{k+1}$ in Algorithm~\ref{alg:SVRG_DIANA} and the first inequality follows from Cauchy  inequality and holds for any $\tau > 0$. 
	
	Taking $\tau = \alpha/2$, we obtain desired inequality.
\end{proof}

\begin{lemma}
We can upper bound the second moment of the $g^k$ in the following way
	\begin{equation}
		\E{\norm{g^k}_2^2}
		\leq
		\frac{\omega}{n}F^k + \frac{\omega+1}{n} L^2 Z^k + \norm{\nabla f(x^k)}_2^2
		.
		\label{lem:up_g_VR_nc_SVRG}
	\end{equation}
\end{lemma}
\begin{proof}
	\begin{align*}
		\EE{Q}{\norm{g^k}_2^2} &\overset{\eqref{eq:variance}}{=}
		\underbrace{
			\norm{\EE{Q}{g^k}}_2^2
		}_{T_1}
		+
		\underbrace{
			\EE{Q}{\norm{g^k - \EE{Q}{g^k}}_2^2}
		}_{T_2}.
	\end{align*}
	We can use the definition of $g^k$ in order to obtain 
	\begin{align*}
		T_1
		&= 
		\norm{
			\frac{1}{n}
			\sum\limits_{i=1}^{n}
				\EE{Q}{Q(g_i^k - h_i^k) + h_i^k}
		}_2^2
		=
		\norm{
			\frac{1}{n}
			\sum\limits_{i=1}^{n}
				g_i^k
		}_2^2
	\end{align*}
	and 
	\begin{align*}
		T_2
		&=
		\EE{Q}{
			\norm{
				\frac{1}{n}
				\sum\limits_{i=1}^{n}
					Q(g_i^k - h_i^k) - (g_i^k - h_i^k)
			}_2^2
		}\\
		&\overset{\eqref{eq:sum}}{=}
		\frac{1}{n^2}
		\sum\limits_{i=1}^{n}
			\EE{Q}{
				\norm{Q(g_i^k - h_i^k) - (g_i^k - h_i^k)}_2^2
			} \\ 
		&\overset{\eqref{def:Q}}{\leq}
		\frac{\omega}{n^2}
		\sum\limits_{i=1}^{n}
			\norm{g_i^k - h_i^k}_2^2.
	\end{align*}
	Now we calculate full expectations conditioned on previous iteration:
	\begin{align*}
		\E{T_2}
		&=
		\frac{\omega}{n^2}
		\sum\limits_{i=1}^{n}
			\E{\norm{g_i^k - h_i^k}_2^2} \\
		&\overset{\eqref{lem:up_g-h_VR_nc_SVRG}}{\leq}
		\frac{\omega}{n}F^k + \frac{\omega}{n}L^2 Z^k.
	\end{align*}
	As for $T_1$, we have
	\begin{align*}
		\E{T_1}
		&=
		\E{\norm{
			\frac{1}{n}
			\sum\limits_{i=1}^{n}
				g_i^k
		}_2^2}
		=
		\norm{
			\frac{1}{n}
			\sum\limits_{i=1}^{n}
			\E{g_i^k}
		}_2^2
		+
		\E{\norm{
				\frac{1}{n}
				\sum\limits_{i=1}^{n}
				g_i^k - \E{g_i^k}
			}_2^2}\\
		&=
		\norm{\nabla f(x^k)}_2^2
		+
		\frac{1}{n^2}
		\sum\limits_{i=1}^{n}
			\E{\norm{g_i^k - \nabla f_i(x^k)}_2^2}\\
		&=
		\norm{\nabla f(x^k)}_2^2
		+
		\frac{1}{n^2}
		\sum\limits_{i=1}^{n}
			\E{\norm{\nabla f_{ij_i^k}(x^k) - \nabla f_{ij_i^k}(z^s) - \E{\nabla f_{ij_i^k}(x^k) - \nabla f_{ij_i^k}(z^s)}}_2^2}\\
		&\overset{\eqref{eq:variance}}{\leq}
		\norm{\nabla f(x^k)}_2^2
		+
		\frac{1}{n^2}
		\sum\limits_{i=1}^{n}
			\E{\norm{\nabla f_{ij_i^k}(x^k) - \nabla f_{ij_i^k}(w_{ij_i^k}^k)}_2^2}\\
		&\overset{\text{Alg.}~\ref{alg:SVRG_DIANA}}{=}
		\norm{\nabla f(x^k)}_2^2
		+
		\frac{1}{mn^2}
		\sum\limits_{i=1}^{n}
			\sum\limits_{j=1}^{m}
				\norm{\nabla f_{ij}(x^k) - \nabla f_{ij}(z^s)}_2^2\\
		&\overset{\eqref{def:smoothness}}{\leq}
		\norm{\nabla f(x^k)}_2^2
		+
		\frac{1}{n}L^2 Z^k.
	\end{align*}
	Now, summing $\E{T_1}$ and $\E{T_2}$ we get
	\begin{align*}
		\E{\norm{g^k}_2^2}
		&=
		\E{T_1 + T_2}
		\leq
		\frac{\omega}{n}F^k + \frac{\omega+1}{n} L^2 Z^k + \norm{\nabla f(x^k)}_2^2,
	\end{align*}
	which concludes the proof.
\end{proof}

\begin{proof}[Proof of \cref{thm:SVRG-DIANA-non-convex}]
Using $L$-smoothness one gets
	\begin{align}
		\E{f(x^{k+1})} &\leq f(x^k) + \dotprod{\nabla f(x^k)}{\E{x^{k+1}-x^k}} + \frac{L}{2}\E{\norm{x^{k+1} - x^k}_2^2} \notag \\
		& = f(x^k) - \gamma\norm{\nabla f(x^k)}_2^2 + \frac{L\gamma^2}{2}\E{\norm{g^k}_2^2},
		\label{lem:nc_smooth_SVRG}
	\end{align}
where we use the definition of $x^{k+1}$ in Algorithm~\ref{alg:SVRG_DIANA}.

By combining definition of $\E{R^{k+1}}$ with \eqref{lem:up_W_VR_nc_SVRG},\eqref{lem:up_F_VR_nc_SVRG},\eqref{lem:nc_smooth_SVRG} one obtains
\begin{align*}
\E{R^{k+1}} &\leq  f(x^k) + \dotprod{\nabla f(x^k)}{\E{x^{k+1}-x^k}} + \frac{L}{2}\E{\norm{x^{k+1} - x^k}_2^2} \\
&\quad + c^{k+1}\left( \E{\norm{x^{k+1} - x^k}_2^2} + \left(1 + \gamma p\right)Z^k +\frac{\gamma}{p}\norm{\nabla f(x^k)}_2^2\right) \\
&\quad + d^{k+1}\left(  \left(1+\frac{2}{\alpha}\right)L^2\E{\norm{x^{k+1} - x^k}_2^2} + \left(1-\frac{\alpha}{2}\right) F^k +\alpha L^2 Z^k\right) \\
&=
f(x^k) - \gamma\norm{\nabla f(x^k)}_2^2 + \left(\frac{\gamma^2L}{2} + c^{k+1} \gamma^2 + d^{k+1}  \left(1+\frac{2}{\alpha}\right)L^2 \gamma^2\right)\E{\norm{g^k}_2^2} \\
&\quad + c^{k+1}\left( \left(1 - \frac{1}{m} + \gamma p\right)Z^k +\frac{\gamma}{p}\norm{\nabla f(x^k)}_2^2\right) \\
&\quad + d^{k+1}\left(\left(1-\frac{\alpha}{2}\right)F^k + \left(1+\frac{2}{\alpha}\right)\alpha L^2 Z^k\right) \\
&\overset{\eqref{lem:up_g_VR_nc_SVRG}}{\leq}
f(x^k) - \left(\gamma - \frac{\gamma^2L}{2}  - c^{k+1} \gamma^2 - d^{k+1} \left(1+\frac{2}{\alpha}\right) L^2 \gamma^2 - c^{k+1}\frac{\gamma}{p}\right)\norm{\nabla f(x^k)}_2^2\\ 
&\quad + \left( c^{k+1}\left( 1 - \frac{1}{m} + \gamma p\right) + d^{k+1} \alpha L^2 +  \frac{\omega+1}{n}  L^2\left(\frac{\gamma^2L}{2}  + c^{k+1} \gamma^2 + d^{k+1}  \left(1+\frac{2}{\alpha}\right)L^2 \gamma^2\right) \right) Z^k\\
&\quad + \left( d^{k+1}\left(1-\frac{\alpha}{2}\right) + \frac{\omega}{n} \left(\frac{\gamma^2L}{2}  + c^{k+1} \gamma^2 + d^{k+1}  \left(1+\frac{2}{\alpha}\right)L^2 \gamma^2\right) \right) F^k\\
&=
f(x^k) - \left(\gamma - \frac{\gamma^2L}{2}  - c^{k+1} \left(\gamma^2 + \frac{\gamma}{p}\right) - d^{k+1} \left(1+\frac{2}{\alpha}\right) L^2 \gamma^2\right)\norm{\nabla f(x^k)}_2^2\\ 
&\quad + \left( c^{k+1}\left( 1 - \frac{1}{m} + \gamma p + \frac{\omega+1}{n} L^2 \gamma^2 \right) + d^{k+1} \left(\alpha L^2 +  \left(1+\frac{2}{\alpha}\right)\frac{\omega+1}{n} L^4 \gamma^2\right)+  \frac{\omega+1}{n} \frac{\gamma^2L^3}{2} \right) Z^k\\
&\quad + \left( d^{k+1}\left( 1 - \frac{\alpha}{2} +  \left(1+\frac{2}{\alpha}\right)\frac{\omega}{n}L^2 \gamma^2\right) + c^{k+1}\frac{\omega}{n} \gamma^2 + \frac{\omega}{n} \frac{\gamma^2L}{2} \right) F^k \\
&= R^k - \Gamma^k \norm{\nabla f(x^k)}_2^2.
\end{align*}
Applying the full expectation and telescoping the equation, one gets desired inequality.
\end{proof}
We can proceed to the proof of Theorem~\ref{thm:SVRG-DIANAnc}.
\begin{proof}[Proof of \cref{thm:SVRG-DIANAnc}]
Recursion for $c^t, d^t$ can be written in a form
\begin{align*}
y^t = Ay^{t+1} + b,
\end{align*}
where 
\begin{align*}
A &= \begin{bmatrix}
 1 + \gamma p + \frac{\omega+1}{n} L^2 \gamma^2 &\alpha L^2 +  \left(1+\frac{2}{\alpha}\right)\frac{\omega+1}{n} L^4 \gamma^2\\
\frac{\omega}{n} \gamma^2  & 1 - \frac{\alpha}{2} +  \left(1+\frac{2}{\alpha}\right)\frac{\omega}{n}L^2 \gamma^2
\end{bmatrix} , \\
y^t &=  \begin{bmatrix}
 c^t \\
 d^t
\end{bmatrix} , \\
b &=  \begin{bmatrix}
 \frac{\omega+1}{n} \frac{\gamma^2L^3}{2} \\
 \frac{\omega}{n} \frac{\gamma^2L}{2}
\end{bmatrix} .
\end{align*}

Choosing $\gamma = \frac{1}{10L\left( 1 + \frac{\omega}{n} \right)^{1/2} (m^{2/3} + \omega + 1)}$, $p =  \frac{L\left(1+ \frac{\omega}{n} \right)^{1/2}}{ (m^{2/3} + \omega + 1)^{1/2}}$, $l = m$, and $\alpha = \frac{1}{\omega+1}$, where  $c^l = d^l = 0$ we can upper bound each element of matrix $A$ and construct its upper bound $\hat{A}$, where 
\begin{align*}
\hat{A} &= \begin{bmatrix}
 1 + \frac{11}{100m}  & L^2 \frac{103}{100(\omega+1)} \\
 \frac{1}{100L^2m} & 1 -  \frac{47}{100(\omega+1)}
 \end{bmatrix},\qquad
 \hat{b} = \left(1+\frac{\omega}{n}\right)\frac{\gamma^2 L}{2}\begin{bmatrix}
 L^2 \\
 1
 \end{bmatrix}.
\end{align*}
The same as for Proof of Theorem~\ref{thm:VR-DIANAnc}, due to structure of $\hat{A}$ and $\hat{b}$ we can work with matrices 
\begin{align*}
\tilde{A} &= \begin{bmatrix}
 1 + \frac{11}{100m}  & \frac{103}{100(\omega+1)} \\
 \frac{1}{100m} & 1 -  \frac{47}{100(\omega+1)}
 \end{bmatrix},\qquad
 \tilde{b} = \left(1+\frac{\omega}{n}\right)\frac{\gamma^2 L}{2}\begin{bmatrix}
 1 \\
 1
 \end{bmatrix},
\end{align*}
where it holds $\tilde{A}^k \tilde{b} = (y_1, y_2)^\top \implies \hat{A}^k \hat{b} = (L^2y_1, y_2)^\top$, thus we can work with $\tilde{A}$ which is independent of $L$.
In the sense of Lemma~\ref{lem:sequence}, we have that eigenvalues of $\tilde{A}$ are less than $1-\frac{1}{3(\omega + 1)}, 1+\frac{1}{m}$, respectively, and $\abs{x} \leq \frac{2}{\min\{m, \omega + 1\}}$, thus for $k \leq l-1$
\begin{align*}
c^k &\leq 20\max\{m, \omega+1\}\left(\left(1+\frac{1}{m}\right)^m -1 \right)\left(1+\frac{\omega}{n}\right) \frac{\gamma^2L^3}{2} \\
&=  20(e - 1)\max\{m, \omega+1\}\left(1+\frac{\omega}{n}\right) \frac{L^3}{200\left(1 + \frac{\omega}{n}\right) L^2 (m^{2/3} + \omega + 1)^2} \\
&\leq \frac{L}{2(m^{2/3} + \omega + 1)^{1/2}}.
\end{align*}
By the same reasoning
\begin{align*}
d^k \leq \frac{1}{2L(m^{2/3} + \omega + 1)^{1/2}}\, .
\end{align*}
This implies 
\begin{align*}
\Gamma^k & = \gamma - \frac{\gamma^2L}{2}  - c^{k+1} \left(\gamma^2 + \frac{\gamma}{p}\right) - d^{k+1}  \left(1 + \frac{2}{\alpha} \right) L^2 \gamma^2 \\
& \geq \gamma - \frac{\gamma^2L}{2}  -   \frac{L}{2(m^{2/3} + \omega + 1)^{1/2}}\left(\left(2 + \frac{2}{\alpha} \right)\gamma^2 + \frac{\gamma}{p}\right) \\
&\geq  \frac{1}{10L\left( 1 + \frac{\omega}{n} \right)^{1/2} (m^{2/3} + \omega + 1)} -  \frac{1}{200L \left(1 + \frac{\omega}{n}\right)(m^{2/3} + \omega + 1)^2 } \\
& \quad  -  \frac{1}{2L(m^{2/3} + \omega + 1)^{1/2}}\left(\frac{4(\omega + 1)}{100\left(1 + \frac{\omega}{n} \right) (m^{2/3} + \omega + 1)^2} + \frac{1}{10\left(1 + \frac{\omega}{n} \right) (m^{2/3} + \omega + 1)^{1/2}}\right) \\
& \geq \frac{1}{40L\left(1 +  \frac{\omega}{n} \right)^{1/2}(m^{2/3} + \omega + 1)},
\end{align*}
which guarantees $\Delta \geq \frac{1}{40L\left(1 +  \frac{\omega}{n} \right)^{1/2}(m^{2/3} + \omega + 1)}$ for $k = 0, 1, \hdots, l-1$. Using iterates $k = cl, cl+1, \hdots, (c+1)l-1$, where $c$ is any positive integer, one can obtain the same bound of $\Gamma^k$ for arbitrary $k$.  Plugging this uniform lower bound on $\Delta$ into \eqref{up:delta_SVRG} for all iterates and using the fact that $p_{l-1} = 1$ and all other $p_r$'s are zeros, one obtains
\begin{align*}
\E{\norm{\nabla f(x^a)}_2^2} \leq 
\frac{40(f(x^0) - f^\star)L\left(1 +  \frac{\omega}{n} \right)^{1/2}(m^{2/3}+ \omega + 1)}{m}.
\end{align*}
where $x^a \sim_{u.a.r.} \{x^0, x^1, \dots, x^{k-1}\}$, which concludes the proof.
\end{proof}

\section{Technical Lemma}

\begin{lemma}
\label{lem:sequence}
Let $A$ be a $2 \times 2$ matrix of which all entries are non-negative and $y^k$ be a sequence of vectors for which $y^k = Ay^{k+1} + b$ and $y^T = (0, 0)$, where $b$ is a vector with non-negative entries, and $\hat{A} = A+B$, $\hat{b} = b+ y$, where $B$ and $y$ have all entries non-negative. Then for the sequence $\hat{y}^k = \hat{A}\hat{y}^{k+1} + \hat{b}$ it always holds that 
$\hat{y}^k \geq y^k$ (coordinate-wise) for $k = 0,1, \dots, T$. Moreover, let $\hat{A}$ has positive real eigenvalues $\lambda_1 \geq \lambda_2 > 0$, thus there exists a real Schur decomposition of matrix $\hat{A}  = UTU^{\top}$, where 
\begin{align*}
T  = \begin{bmatrix}
 \lambda_1  & x\\
0 & \lambda_2
 \end{bmatrix}
\end{align*} and $x \in \R$, and $U$ is real unitary matrix, then for every element of $\hat{y}^k$ it holds
\begin{equation*}
\hat{y}^k_j \leq \left(\frac{(1 - \lambda_1^T)(1 - \lambda_2^T)}{(1 - \lambda_1)(1 - \lambda_2)}\abs{x} + \frac{(1 - \lambda_1^T)}{(1 - \lambda_1)} + \frac{(1 - \lambda_2^T)}{(1 - \lambda_2)}\right)(b_1 + b_2)
\end{equation*}
for $k = 0,1,2, \hdots, T-1$.
\end{lemma}

\begin{proof}
From $y^k = Ay^{k+1} + b$ and $y^T = (0, 0)$ one can obtain
\begin{align*}
y^{T-k} &= A^{k-1}b + A^{k-2}b + \dots + b \\
\hat{y}^{T-k} &= (A+B)^{k-1}(b+y) + (A+B)^{k-2}(b+y) + \dots + (b+y) .
\end{align*}
From these equalities it is trivial to see that $\hat{y}^k \leq y^k$, because $\hat{y}^k$ contains at least all the elements of $y^k$ and every element is non-negative.

For the second part of the claim, we have for every element of  $\hat{y}^{T-k}$
\begin{align*}
\hat{y}^{T-k}_j &\leq \norm{\hat{y}^{T-k}}_2 = \norm{\hat{A}^{k-1}\hat{b} + \hat{A}^{k-2}\hat{b} + \dots + \hat{b}}_2 \\
&\leq \norm{\hat{A}^{k-1}\hat{b}}_2 + \norm{\hat{A}^{k-2}\hat{b}}_2 + \dots + \norm{\hat{b}}_2 \\
& =  \norm{UT^{k-1}U^{\top}\hat{b}}_2 + \norm{UT^{k-2}U^{\top}\hat{b}}_2 + \dots + \norm{\hat{b}}_2 \\
& \leq \left(\norm{T^{k-1}}_2 + \norm{T^{k-2}}_2 + \dots + \norm{I}\right)\norm{\hat{b}}_2 \\
& \leq  \left(\frac{(1 - \lambda_1^T)(1 - \lambda_2^T)}{(1 - \lambda_1)(1 - \lambda_2)}\abs{x} + \frac{(1 - \lambda_1^T)}{(1 - \lambda_1)} + \frac{(1 - \lambda_2^T)}{(1 - \lambda_2)}\right)(\hat{b}_1 + \hat{b}_2),
\end{align*}
where the last inequality follows from fact that $\norm{A}_2 \leq \norm{A}_F \leq \abs{a_{11}} + \abs{a_{12}} + \abs{a_{21}} + \abs{a_{22}}$ and 
\begin{align*}
T^k  = \begin{bmatrix}
 \lambda_1^k  & x \sum_{i=0}^k \lambda_1^i  \lambda_2^{k-i}\\
0 & \lambda_2^k
 \end{bmatrix},
\end{align*}
which concludes the proof.
\end{proof}